\documentclass[onefignum,onetabnum,final]{siamonline190516}

\usepackage{amsfonts}
\usepackage{epstopdf}
\usepackage{framed}
\usepackage[normalem]{ulem}
\usepackage{xcolor}

\usepackage{graphicx}
\usepackage[export]{adjustbox}
\usepackage{subcaption}
\usepackage{float}
\usepackage{dsfont}

\def\R{{\mathbb R}}
\def\N{{\mathbb N}}
    
\def\Sp{{\mathbb S}}  
\def\Discr{\Delta}

\def\Diff{{\operatorname{Diff}}}
\def\Adj{{\mathcal{A}}}
\def\Immr{{\operatorname{Imm}^r([0,1], \R^d)}}
\def\Imm{\operatorname{Imm}}
\def\GraphrA{{\operatorname{Graph}^r(\Adj)}}
\def\GraphnA{{\operatorname{Graph}^n(\Adj)}}
\def\Graphr{{\operatorname{Graph}^r}}
\def\Graphn{{\operatorname{Graph}^n}}
\def\WGraphrA{{\operatorname{\mathcal{W}Graph}^r(\Adj)}}
\def\WGraphnA{{\operatorname{\mathcal{W}Graph}^n(\Adj)}}
\def\WGraphr{{\operatorname{\mathcal{W}Graph}^r}}
\def\WGraphn{{\operatorname{\mathcal{W}Graph}^n}}
\def\WGraph{{\operatorname{\mathcal{W}Graph}}}
\def\Graphtwo{{\operatorname{Graph}^2}}
\def\GraphtwoA{{\operatorname{Graph}^2(\mathcal{A})}}
\def\K{\mathcal{K}}
\def\H{\mathcal{H}}
\def\V{\mathcal{V}}
\def\argmin{{\operatorname{argmin}}}
\def\deltarho{\delta\hspace{-.1em} \rho}
\def\red{\textcolor{red}}

\newsiamremark{remark}{Remark}
\usepackage{blkarray}

\title{A new variational model for shape graph registration with partial matching constraints
\thanks{Y. Sukurdeep and N. Charon were supported by NSF grants 1945224 and 1953267. M. Bauer was supported by NSF grants 1953244 and 1912037}}

\author{Yashil Sukurdeep\thanks{Department of Applied Mathematics and Statistics, Johns Hopkins University.}
\and
Martin Bauer\thanks{Department of Mathematics, Florida State University.}
\and
Nicolas Charon\thanks{Center for Imaging Science \& Department of Applied Mathematics and Statistics, Johns Hopkins University.}}

\ifpdf
\hypersetup{
  pdftitle={Shape Graphs},
  pdfauthor={Y. Sukurdeep, M. Bauer and N. Charon}
}
\fi

\begin{document}

\maketitle

% ------------------------ ABSTRACT ---------------------------------------
\begin{abstract}
This paper introduces a new extension of Riemannian elastic curve matching to a general class of geometric structures, which we call (weighted) shape graphs, that allows for shape registration with partial matching constraints and topological inconsistencies. Weighted shape graphs are the union of an arbitrary number of component curves in Euclidean space with potential connectivity constraints between some of their boundary points, together with a weight function defined on each component curve. The framework of higher order invariant Sobolev metrics is particularly well suited for constructing notions of distances and geodesics between unparametrized curves. The main difficulty in adapting this framework to the setting of shape graphs is the absence of topological consistency, which typically results in an inadequate search for an exact matching between two shape graphs. We overcome this hurdle by defining an inexact variational formulation of the matching problem between (weighted) shape graphs of any underlying topology, relying on the convenient measure representation given by varifolds to relax the exact matching constraint. We then prove the existence of minimizers to this variational problem when we choose Sobolev metrics of sufficient regularity and a total variation (TV) regularization on the weight function. We propose a numerical optimization approach which adapts the smoothed fast iterative shrinkage-thresholding (SFISTA) algorithm to deal with TV norm minimization and allows us to reduce the matching problem to solving a sequence of smooth unconstrained minimization problems. We finally illustrate the capabilities of our new model through several examples showcasing its ability to tackle partially observed and topologically varying data. 
\end{abstract}

\begin{keywords}
  elastic shape analysis, shape graphs, partial matching, Sobolev metrics, varifold, total variation norm, SFISTA algorithm.
\end{keywords}

\begin{AMS}
  49J15, 49Q20, 53A04
\end{AMS}

% -------------------------- INTRODUCTION ---------------------------------
\section{Introduction} 

Fueled by transformative progress in medical and biological imaging, the past decade has seen a remarkable explosion in the quantity and quality of data whose predominantly interesting features are of geometric and topological nature. These developments urged the need for new mathematical and algorithmic approaches for dealing with such data, which led to the growth of the new areas of (geometric and topological) data and shape analysis, see e.g.~\cite{younes2010shapes,srivastava2016functional,kendall1989survey,bauer2014overview,edelsbrunner2008persistent}. 

Among the plethora of techniques that have emerged, several methods from the so-called field of elastic shape analysis (ESA) have proven successful in numerous applications with geometric data such as functions, curves or surfaces~\cite{srivastava2016functional}. The guiding principle in ESA lies in the concept of invariance: an object's shape is the information that remains after factoring out the action of certain groups. This includes the finite dimensional rigid transformation groups that act on the ambient space of the objects (rotations, translations, scalings) and, more importantly, also includes the infinite dimensional reparametrization group on the domain of the function, curve or surface. The latter action encodes the usually unknown point-to-point correspondences between geometric objects, and finding these correspondences, which is known as shape registration, forms the basis for comparing, classifying, clustering and performing statistical analysis on geometric data. 

However, shape registration is the main source of difficulty in the area of ESA. While other techniques often solve for the registration between geometric objects as a pre-processing step under a certain metric or objective function, and follow it up with statistical analysis that is independent of this registration metric, ESA is an approach where both registration and comparisons of shapes are performed in a unified fashion, under the same metric. ESA is performed in a Riemannian setting using a metric that is invariant to the action of shape-preserving transformations, and thereby descends to a Riemannian metric on the quotient shape space. For functions, curves and surfaces, this framework is well-developed both from an algorithmic and a theoretical point of view. In particular, for the case of invariant Sobolev metrics on spaces of curves, powerful existence results for optimal reparametrizations and optimal deformations have been obtained~\cite{lahiri2015precise,bruveris2014geodesic,nardi2016geodesics,bruveris2016optimal,bauer2020sobolev,bauer2019relaxed}, and efficient numerical algorithms for the computation of the resulting elastic distances and geodesics are available~\cite{srivastava2016functional,bauer2017numerical,bauer2019relaxed}. Notably, the so called Square Root Velocity (SRV) framework~\cite{srivastava2016functional,srivastava2010shape} enables near real-time computations under a certain type of first order Sobolev metric, even in the context of large datasets.

In recent work by Srivastava et. al.~\cite{srivastava2020advances,duncan2018statistical,guo2020statistical,guo2020representations,wang2020statistical} the SRV framework has been extended to more general geometric objects such as shape graphs, i.e., shapes with network or branching structures where each branch is a geometric curve. In their framework, ESA is combined with graph matching algorithms in order to determine correspondences between shape graphs. In particular, the authors use standard methods developed for the space of curves in order to find correspondences between each separate branch of the shape graphs. Aside from ESA techniques, other methods to tackle the comparison and/or matching of shape graphs include the model of the space of unlabelled trees with the quotient Euclidean distance \cite{jain2009structure,feragen2012toward,calissano2020populations,feragen2020statistics}, or diffeomorphic registration approaches such as in \cite{pan2016current}. Yet, the former approach requires optimizing over permutations of the branches and more importantly restricts to modeling each branch by a finite set of landmark points and comparing those through the usual Euclidean metric, which does not embed the fundamental reparametrization invariance of the shape space of geometric curves. Diffeomorphic transformations on the other hand typically impose stronger constraints on curve deformations and usually suffer from higher numerical cost compared to elastic models, as pointed out in \cite{bauer2019relaxed}.

Another related persistent challenge in shape analysis is the issue of partial correspondences or changes in topology. This is of particular importance in the context of applications involving partial observations (e.g. corrupted data) and/or topological inconsistencies (e.g. airway vessels or brain arteries with distinct graph structure). Several works have attempted to address the issue of partial matching, in particular frameworks that build on the Gromov-Hausdorff distance \cite{bronstein2009partial}, or functional maps \cite{rodola2017partial} in the context of 3D surfaces, and very recently diffeomorphic models \cite{kaltenmark2019estimation,antonsanti2021partial} in which asymmetric data attachment terms are introduced in the registration functional. In ESA, however, not much work has been achieved to date. In the thesis \cite{robinson2012functional}, a certain type of partial matching constraint was included into the SRV framework, which essentially searches for a single ``subcurve'' of the source curve to be matched to a single ``subcurve'' of the target curve. Although this approach has led to promising results, it remains quite restrictive as to the class of partial correspondences that can be accounted for, e.g., it does not allow for several missing parts or different numbers of branches in the source and target.  

Finally, while the SRV framework has tremendous algorithmic advantages, it corresponds to a very specific choice of Riemannian metric, which may be undesirable in situations where one would prefer to have a data-driven selection of the metric, see~e.g. the results of Needham and Kurtek~\cite{needham2020simplifying}. In the context of more general metrics, such as higher order Sobolev metrics, the main challenge is to find an efficient way to deal with the action of the reparametrization group. In a recent publication by the authors \cite{bauer2019relaxed}, an efficient numerical approach for the registration of standard Euclidean curves was proposed, in which the endpoint constraint in the matching functional is relaxed based on metrics from geometric measure theory.  

\subsection*{Contributions of the article}
Our first contribution in this article is to introduce an extension of the inexact matching framework for general elastic metrics of \cite{bauer2019relaxed} to the space of shape graphs, thereby also complementing the recent work of Srivastava et. al. Our model differs from the SRV shape graph framework in several ways. Most significantly, instead of registering shape graphs by solving a series of independent exact matching problems which are preceded by a branch matching algorithm, we directly solve a single inexact matching problem where the matching of component curves from the shape graphs is geometrically driven through a varifold relaxation term. Secondly, our approach allows for a wide class of Riemannian metrics instead of focusing on one particular metric for computational ease. This allows us to recover a quite general result on the existence of minimizers to the shape graph matching variational problem (Theorem \ref{thm:existence_graph_match}), the proof of which leverages the recent theoretical results of \cite{bruveris2017completeness} for higher order elastic metrics combined with the derivation of specific semi-continuity properties for varifold relaxation terms.  

Yet, the main contribution and original motivation of the present article is to incorporate partial matching constraints and topological inconsistencies into the shape matching framework. Towards this end, we consider \textit{weighted} shape graphs, which augment the shape graph model with a spatially varying weight function. When registering weighted shape graphs, this weight function is jointly estimated together with the geometric deformation of the source shape graph, which allows us to discard or ``create'' specific parts of the shape graph thanks to the varifold relaxation term. By penalizing the weight variation through a TV-norm based regularization, we show that the existence of minimizers again holds in this generalized setting (Theorem \ref{thm:existence_weighted_graph_match}), which is the major theoretical result of the article. 

Finally, our model is accompanied with an open source implementation, available on Github\footnote{\url{https://github.com/charoncode/ShapeGraph_H2match}}. It builds up on the $H^2$-metrics matching code of the second and third author and collaborators, which relies on a discretization of the curves using smooth splines. In addition, we deal with the numerical optimization over the weight function with the non-smooth TV regularizer by adapting the smooth FISTA algorithm of \cite{tan2014smoothing}. This allows us to reduce the matching problem to solving a sequence of smooth unconstrained minimization problems on the spline control points and weights of the source shape graph. We illustrate the capabilities of our new model through several examples using both real and artificial data, thereby showcasing its ability to tackle partially observed and topologically varying data. 

\subsection*{Paper structure}
Our paper is structured as follows. In Section~\ref{sec:riemannian_metrics_on_curve_spaces}, we introduce the space of shape graphs, which will allow us to model geometric objects having arbitrary topological structures. In Section~\ref{sec:relaxed_matching_problem}, we formulate and analyze an inexact matching approach for comparing shape graphs having the same topology, which relies on a varifold-based relaxation of the exact registration problem. We then generalize our discussion to the setting of weighted shape graphs in Section~\ref{sec:model_weights}, where we introduce a variational approach for the registration of shape graphs having different topological structures in \eqref{eq:relaxed_weighted_shape_graph_match}, and prove the existence of solutions to this problem, which is the main theoretical result from this paper. Finally, we conclude the manuscript with a description of our proposed numerical approach for solving \eqref{eq:relaxed_weighted_shape_graph_match} in Section~\ref{sec:optim_algo}, and discuss numerical experiments on synthetic and real data in Section~\ref{sec:numerical_results}.

%----------------- RIEMANNIAN METRICS ON SPACE OF CURVES ------------------%
\section{Riemannian metrics on curves and shape graphs}
\label{sec:riemannian_metrics_on_curve_spaces}
In this section, we introduce key theoretical background on higher-order elastic Sobolev metrics, which we then use to define a reparametrization invariant Riemannian metric on the space of shape graphs.

\subsection{Metrics on open and closed curves}
\label{sec:curves}
We begin by reviewing the main definitions and known theoretical results on higher-order Sobolev metrics for closed and open curves, which we will rely on for the rest of this paper. A \textit{parametrized immersed curve} of regularity $r$ (with $r> 3/2$) in $\R^d$ is a mapping $c:D \rightarrow \R^d$ such that $c \in H^r(D,\R^d)$ (the space of Sobolev functions of order $r$ on $D$), where for all $\theta \in D$ we have $\partial_{\theta} c(\theta) \neq 0$. We shall denote the space of all such parametrized curves by $\Imm^r(D,\R^d)$. In this section, we consider the two cases where $D=[0,1]$ or $D=S^1$, corresponding to open or closed curves respectively. We also introduce the \textit{reparametrization group} $\Diff^r(D)$, which is the group of orientation-preserving $H^r$ diffeomorphisms of $D$, i.e., the space of all $\varphi \in H^r(D)$ such that $\partial_{\theta} \varphi(\theta) > 0$ for all $\theta$ and $\varphi^{-1} \in H^r(D)$. Note that the assumption $r>3/2$ ensures that the condition on the first derivative being non-zero is well-defined. For any immersed curve $c \in \Imm^r(D,\R^d)$ and $\varphi \in \Diff^r(D)$, we say that $c\circ \varphi \in  \Imm^r(D,\R^d)$ is a \textit{reparametrization} of $c$ by $\varphi$. Finally, for any $c\in \Imm^r(D,\R^d)$, we denote the arc length integration form by $ds = |\partial_{\theta} c(\theta)| d\theta$, the total length of the curve by $\ell_c = \int_D ds$, and the arc length differentiation along $c$ by $\partial_s = \frac{1}{|\partial_{\theta} c(\theta)|} \partial_{\theta}$.

As an open subset of $H^r(D,\R^d)$, the space of parametrized curves $\Imm^r(D,\R^d)$ is an infinite dimensional (Hilbert) manifold, where the tangent space at any $c \in  \Imm^r(D,\R^d)$ is given by $H^r(D,\R^d)$. Any tangent vector $h$ to this manifold at the curve $c$ can be canonically identified with a vector field along the curve by simply viewing $h(\theta)$ as the tangent vector attached to the point $c(\theta)$. Now, one of the main goals in shape analysis is to endow this manifold of curves with a metric in order to define a notion of distance between pairs of curves, with Riemannian metrics being preferred as they also yield a notion of geodesics between the curves. While the usual Sobolev metric on $H^r(D,\R^d)$ at each $c\in  \Imm^r(D,\R^d)$ defined by 
\begin{equation}
    \|h\|_{H^r}^2 = \sum_{i=0}^{r} \int_D \langle\partial^i_\theta h,\partial^i_\theta h \rangle d\theta    
    \label{eq:Sobolev_norm}
\end{equation}
may seem like a natural choice, a critical issue when using it to compare the shape of two curves is its lack of invariance to the choice of parametrization for the curves. This prevents such a metric from descending to a metric on the true shape space of curves modulo reparametrizations, which we will introduce later in this section.

The construction of parametrization invariant Riemannian metrics on $\Imm^r(D,\R^d)$ has been the subject of intense study since the work of \cite{younes1998computable}. In this paper, we focus on the family of so-called invariant elastic Sobolev metrics \cite{michor2007overview,mennucci2008properties,srivastava2016functional}, and more specifically on two subclasses of these metrics on $\Imm^r(D,\R^d)$. For any $c \in \Imm^r(D,\R^d)$ and tangent vector $h \in H^r(D,\R^d)$, the metric takes one of the following forms: 
\begin{equation}
\label{eq:def_Sobolev_metric}
G_c^{n}(h,h) = \sum_{i=0}^{n} a_i \int_{D} \langle \partial_s^i h , \partial_s^i h \rangle ds ,
\end{equation}
which we will refer to as constant-coefficient Sobolev metrics of order $n$, or  
\begin{equation}
\label{eq:def_Sobolev_metric_length_inv}
G_c^{n}(h,h) = \sum_{i=0}^{n} a_i \ell_c^{2r-i} \int_{D} \langle \partial_s^i h , \partial_s^i h \rangle ds ,
\end{equation}
which are known as scale invariant Sobolev metrics of order $n$. In both equations, $n\leq r$ and $a_0,a_1,\ldots,a_n$ are positive coefficients weighting the different terms in the Sobolev metrics. Note that the key difference of \eqref{eq:def_Sobolev_metric} compared to a standard Sobolev norm is that the derivatives and integration are taken with respect to arc length. This is precisely what makes the metric invariant to reparametrizations in the sense that $G_{c\circ \varphi}^{n}(h\circ \varphi,h \circ \varphi) = G_c^{n}(h,h)$ for all $\varphi \in \Diff(D)$, which can be obtained by simple change of variable in the integrals. In the same way, the metric \eqref{eq:def_Sobolev_metric_length_inv} is also reparametrization invariant, while also being invariant to the action of scaling since it can be shown that for all $\lambda >0$, one has $G_{\lambda c}^{n}(\lambda h, \lambda h ) = G_c^{n}(h,h)$.

The Riemannian distance between two parametrized curves $c_0$ and $c_1$ in $\Imm^r(D,\R^d)$ is then given by the usual minimum of the energy taken over all paths connecting these two curves, specifically:
\begin{equation}
\label{eq:Riemannian_dist}
    \operatorname{dist}_{G^n}(c_0,c_1) = \inf\left\{\int_0^1 G_{c(t)}^{n}(\partial_t{c}(t),\partial_t{c}(t)) dt \right\} ,
\end{equation}
with the infimum being taken over all paths of immersions $c\in H^1([0,1],\Imm^r(D,\R^d))$ such that $c(0)=c_0$, $c(1)=c_1$. Here, $\partial_t{c}(t) \in H^r(D,\R^d)$ denotes the derivative with respect to $t$ of this path.

There are however several important subtleties compared to standard finite dimensional Riemannian geometry. One is the fact that \eqref{eq:Riemannian_dist} may only be a pseudo-distance. In fact for $n=0$, it is known since \cite{Michor2005,bauer2012vanishing} that $\operatorname{dist}_{G^n}$ is completely degenerate, i.e., that $\operatorname{dist}_{G^n}(c_0,c_1) = 0$ for any $c_0,c_1 \in \Imm^r(D,\R^d)$. Fortunately, whenever $r\geq 1$ (which we always assume here), it was shown that $\text{dist}_{G^n}$ is a true distance for both the metrics from \eqref{eq:def_Sobolev_metric} and \eqref{eq:def_Sobolev_metric_length_inv}, see~\cite{michor2007overview}. Yet another fundamental and desirable property of Riemannian distances is \textit{metric completeness}. For finite dimensional Riemannian manifolds, the Hopf-Rinow theorem guarantees that completeness of the corresponding metric space is equivalent to geodesic completeness and also equivalent to the existence of minimizing geodesics between any two points on the manifold. However, this equivalence no longer holds in the infinite dimensional case~\cite{atkin1975hopf} and therefore significant effort has been invested in the study of completeness properties for the parametrization invariant Sobolev metrics on immersed curves~\cite{bruveris2014geodesic,nardi2016geodesics,bruveris7completeness,bruveris2017completeness,bauer2020sobolev}. We summarize this collection of results by the following theorem, and note that we write the statement for $r=n$ for simplicity although this can be made more general: 

\begin{theorem}\label{thm:completeness_curve}
Assume that $n\geq 2$. Then, given the scale-invariant Sobolev metric \eqref{eq:def_Sobolev_metric_length_inv} and its associated geodesic distance \eqref{eq:Riemannian_dist}, the following holds:
\begin{remunerate}
    \item The space $\left(\Imm^n(D,\R^d),\operatorname{dist}_{G^n}\right)$ is a complete metric space.
    \item The space $\left(\Imm^n(D,\R^d),\operatorname{dist}_{G^n}\right)$ is geodesically complete i.e. solutions of the geodesic equation for any choice of initial condition are defined at all times.   
    \item For any $c_0, c_1$ in the same connected component of $\Imm^n(D,\R^d)$, there exists a minimizing geodesic in $\Imm^n(D,\R^d)$ between $c_0,c_1$, i.e., there exists $c\in H^1([0,1],\Imm^n(D,\R^d))$ achieving the infimum in \eqref{eq:Riemannian_dist}.
\end{remunerate}
For $D=S^1$ these statements also hold for constant-coefficient Sobolev metrics \eqref{eq:def_Sobolev_metric}.
\end{theorem}
This result justifies our focus, for the rest of this paper, on Sobolev metrics on the space of parametrized curves which are of order $n \geq 2$. Moreover, it is important to stress that completeness does not hold in the case of open curves ($D=[0,1]$) for constant-coefficient Sobolev metrics (see the counterexample of \cite{bauer2019relaxed}), which is also why the scale-invariant version of \eqref{eq:def_Sobolev_metric_length_inv} will be more relevant from a theoretical standpoint in the following sections. As a by-product of Theorem \ref{thm:completeness_curve}, one also obtains the following bounds relating invariant Sobolev metrics and the usual Sobolev norm, which we will rely on in our later proofs:
\begin{lemma}
\label{lemma:bounds_Sobolev_metrics}
Let $n\geq 2$ and $G^n$ the scale invariant metric from \eqref{eq:def_Sobolev_metric_length_inv}. Given $c_0 \in \Imm^n(D,\R^d)$ and a metric ball $B(c_0, \delta)$ of radius $\delta > 0$ in $\Imm^n(D,\R^d)$, there exists a constant $C > 0$ such that for any $c\in B(c_0, \delta)$, it holds that
\begin{equation}
C^{-1} \|h\|_{H^n}^2 \leq G^n_c(h,h)\leq C \|h\|_{H^n}^2
\end{equation}
for all $h\in H^n(D,\R^d)$, and we have the lower bound $|\partial_{\theta} c(\theta)|\geq C^{-1}$ for all $\theta \in D$. 
The same result also holds for the metric \eqref{eq:def_Sobolev_metric} when $D=S^1$.
\end{lemma}

Finally, we can introduce the space of \textit{unparametrized immersed curves}, which is defined as the quotient of parametrized immersed curves by the reparametrization group, i.e., $\Imm^r(D,\R^d)/\Diff^r(D)$. In other words, an unparametrized curve is an equivalence class of $c \in \Imm^r(D,\R^d)$ modulo its reparametrizations, which we denote by $[c] = \{c \circ \varphi \ | \ \varphi \in \Diff^r(D)\}$. As we are ultimately interested in comparing such unparametrized shapes, the question is whether the distance \eqref{eq:Riemannian_dist} descends to a distance on the quotient space. The following theorem from \cite{bruveris7completeness} addresses this point:
\begin{theorem}\label{thm:quotient_distance_curves}
Let $n\geq 2$ and $\operatorname{dist}_{G^n}$ the Riemannian distance \eqref{eq:Riemannian_dist} associated to the scale-invariant Sobolev metric \eqref{eq:def_Sobolev_metric_length_inv}. Then $\Imm^n(D,\R^d)/\Diff^n(D)$ equipped with the quotient distance
\begin{equation}
\label{eq:quotient_distance}
\operatorname{dist}_{G^n}([c_0] , [c_1]) = \inf_{\varphi \in \Diff^n(D)} \operatorname{dist}_{G^n}(c_0, c_1 \circ \varphi)
\end{equation} 
is a length space, and any two unparametrized curves in the same connected component can be joined by a minimizing geodesic, i.e., there exists $\varphi \in \Diff^n(D)$ achieving the minimum in \eqref{eq:quotient_distance} and an optimal path in $\Imm^n(D,\R^d)$ connecting $c_0$ and $c_1 \circ \varphi$.\\
This result also holds for the constant coefficient metric \eqref{eq:def_Sobolev_metric} when $D=S^1$.
\end{theorem}

In practice, computing the distance between two unparametrized curves in this framework thus requires solving a \textit{matching problem} that consists in finding the optimal reparametrization and immersion path. We will discuss numerical aspects later in the paper once we have introduced our generalization of this setting to shape graphs. 

\subsection{Metrics on shape graphs}
\label{ssec:shape_graphs}
We now introduce the space of parametrized shape graphs, see Fig.~\ref{fig:ex_shape_graph} for examples of objects that lie in this space. Shape graphs generalize the concept of open and closed curves, providing us with a framework to perform elastic shape analysis for shapes with non-standard topologies, such as trees, venation networks and shapes with multiple connected components. In this section, our aim is to equip the space of shape graphs with a reparametrization invariant Riemannian metric and show that the completeness properties, as obtained for open curves, continue to hold on this space. In recent work by Srivastava et. al. a particular Sobolev metric of order $n=1$, namely the square-root velocity metric, has been introduced on this space~\cite{srivastava2020advances,duncan2018statistical,guo2020statistical,guo2020representations,wang2020statistical}. 

A \textit{parametrized shape graph} $c = \prod_{k=1}^K c^k$ is a Cartesian product of $K$ component curves $c^1,...,c^K$, where $c^k \in \Immr$ for each $k = 1,...,K$. To describe the connectivity between component curves $c^k$ and $c^l$, where $k,l = 1,...,K$, we fix an adjacency matrix $\Adj \in \{0,1\}^{2K \times 2K}$ which encodes if a boundary point of the $k^{th}$ component curve is connected to a boundary point of the $l^{th}$ component curve. More precisely, the adjacency matrix $\Adj$ is defined as follows:
\begin{align}
    \Adj_{2k-1,2l-1} &= 
    \begin{cases}
    1 \quad \mbox{if $c^k(0) = c^l(0)$,} \\
    0 \quad \mbox{else.}
    \end{cases}  \label{conn1}  \\ 
    \Adj_{2k,2l} &= 
    \begin{cases}
    1 \quad \mbox{if $c^k(1) = c^l(1)$,} \\
    0 \quad \mbox{else.}
    \end{cases}  \label{conn2}  \\ 
    \Adj_{2k-1,2l} &= 
    \begin{cases}
    1 \quad \mbox{if $c^k(0) = c^l(1)$,} \\
    0 \quad \mbox{else.}
    \end{cases}  \label{conn3} 
\end{align}
Note that when $k=l$, condition~\eqref{conn3} encodes the topology of component curve $c^k$, i.e., it encodes if the curve $c^k$ is closed or open:
\begin{align*}
    \Adj_{2k-1,2k} =
    \begin{cases}
    1 \quad \mbox{if $c^k$ is a closed curve,} \\
    0 \quad \mbox{if $c^k$ is an open curve.}
    \end{cases}
\end{align*}
We point out that the topology of the shape graph is entirely encoded by the adjacency matrix, while its geometry is encoded by the component curves. See Fig.~\ref{fig:ex_shape_graph} for an illustrative example of this construction.

\begin{figure}[h]
\begin{center}
    \begin{subfigure}{0.31\textwidth}
        \includegraphics[width=1\linewidth]{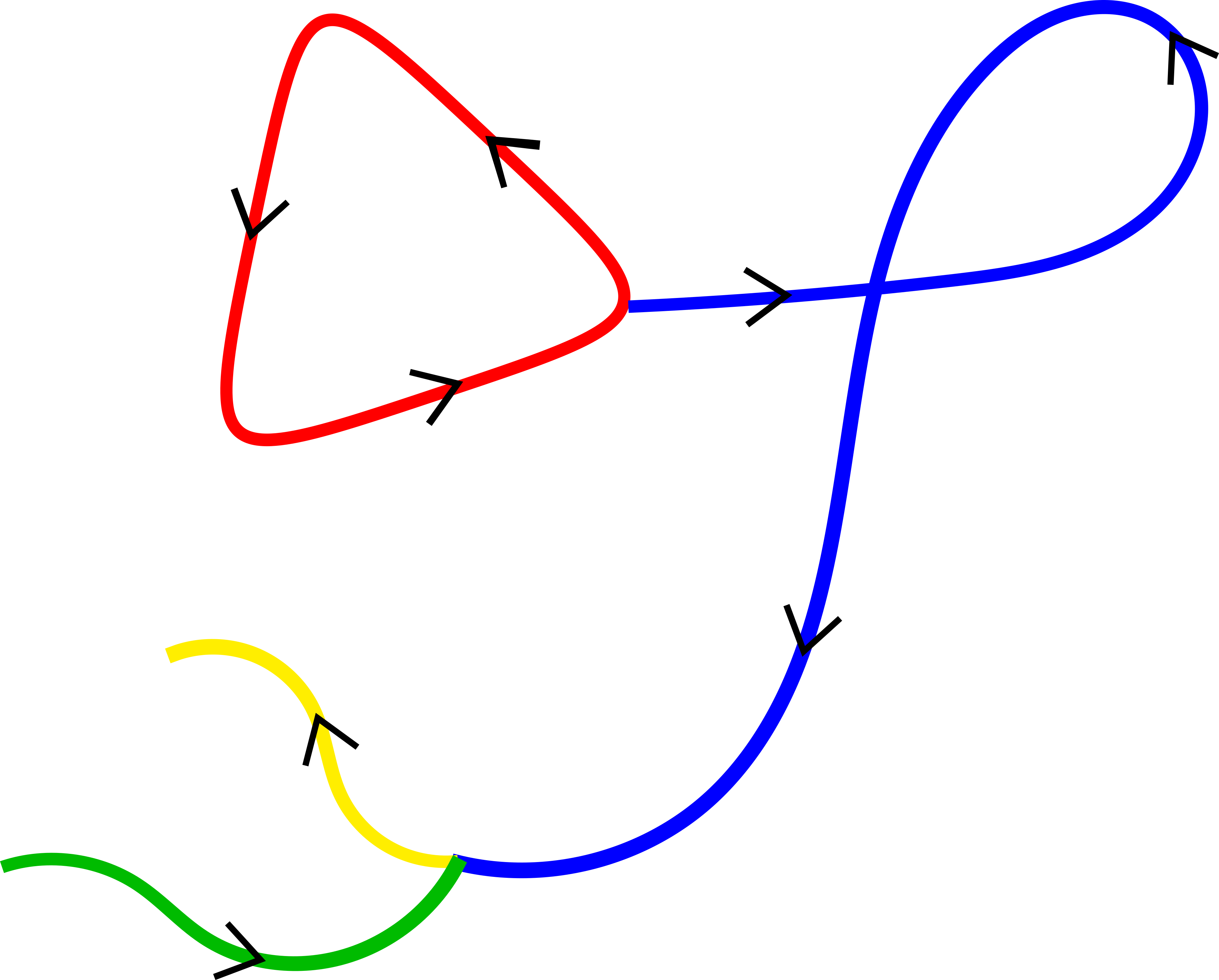}
    \end{subfigure}
    \hspace{1cm}
    \begin{subfigure}{0.44\textwidth}
         \begin{align*}
            \Adj =\begin{pmatrix}
                  1 & 1 & 1 & 0 & 0 & 0 & 0 & 0 \\
                  1 & 1 & 1 & 0 & 0 & 0 & 0 & 0 \\
                  1 & 1 & 1 & 0 & 0 & 0 & 0 & 0 \\
                  0 & 0 & 0 & 1 & 1 & 0 & 0 & 1 \\
                  0 & 0 & 0 & 1 & 1 & 0 & 0 & 1 \\
                  0 & 0 & 0 & 0 & 0 & 1 & 0 & 0 \\
                  0 & 0 & 0 & 0 & 0 & 0 & 1 & 0 \\
                  0 & 0 & 0 & 1 & 1 & 0 & 0 & 1 
            \end{pmatrix}
        \end{align*}
    \end{subfigure}
\end{center}    
\caption{Parametrized shape graph (left) with associated adjacency matrix (right). The shape graph $c = \prod_{k=1}^K c^k$ has $K=4$ component curves, where $c^1$ is a closed curve (red), $c^2$ is an immersion with self intersection (blue), and $c^3$ and $c^4$ are open curves (yellow and green respectively).}
\label{fig:ex_shape_graph}
\end{figure}
This leads us to define the space of parametrized shape graphs of regularity $r$ and adjacency matrix $\Adj$ as
\begin{equation*}
    \GraphrA = \left\{ c\in \prod_{k=1}^K \Immr: c\text{ satisfies }  \eqref{conn1}-\eqref{conn3} \right\},
\end{equation*}
where, for any $r>3/2$, the space of Sobolev immersions $\Immr$ is defined as in Section~\ref{sec:curves}. We will also denote the set of shape graphs of regularity $r$ with arbitrary adjacency structure by $\Graphr = \bigcup_{\Adj} \GraphrA$. The space $\GraphrA$ is a Hilbert manifold as it is a product of Hilbert manifolds with linear constraints. Its tangent space at a shape graph $c$ is given by:
\begin{equation*}
    T_c\GraphrA=\left\{h\in \prod_{k=1}^K H^r([0,1],\mathbb R^d): h\text{ satisfies }  \eqref{conn1}-\eqref{conn3} \right\}.
\end{equation*}
The product reparametrization group $\prod_{k=1}^K \Diff^{r}([0,1])$ acts component-wise on the space of parametrized shape graphs as follows:
\begin{equation}
    \GraphrA \times \prod_{k=1}^K \Diff^{r}([0,1]) \to \GraphrA, \qquad (c,\varphi) \mapsto c^k \circ \varphi^k.
\end{equation}
Due to the product structure of the space of shape graphs, the class of elastic Sobolev metrics on the space of immersed curves, as defined in~\eqref{eq:def_Sobolev_metric} and~\eqref{eq:def_Sobolev_metric_length_inv}, generalizes to a reparametrization invariant metric on the space of shape graphs with fixed adjacency structure. This allows us to directly obtain the analogues of the completeness results of Theorem~\ref{thm:completeness_curve} on this product space. We formulate the result again for $n=r$ only:
\begin{theorem}\label{thm:grap:completeness}
For $n\geq 2$ let
\begin{equation}
\label{eq:def_bar_G}
    \bar G^n_c(h,h)=\sum_{k=1}^{K} G^{n}_{c^k}(h^k,h^k),
\end{equation}
where $c\in \GraphnA$, $h\in T_c\GraphnA$ and where 
$G^{n}$ is the scale invariant Sobolev metric as defined in~\eqref{eq:def_Sobolev_metric_length_inv}. Then $\bar G^n$ defines a smooth, strong and reparametrization invariant Riemannian metric on $\GraphnA$. 
Let $\overline{\operatorname{dist}}_{\bar G^n}$ denote the induced geodesic distance of $\bar G^n$ on $\GraphnA$. The following properties hold:
\begin{remunerate}
    \item The space $\left( \GraphnA, \overline{\operatorname{dist}}_{\bar G^n} \right)$ is a complete metric space.
    \item The space $\left( \GraphnA, \bar G^n \right)$ is geodesically complete.
    \item For any $c_0, c_1 \in \GraphnA$ there exists a minimizing geodesic in $\GraphnA$ (w.r.t. to the metric $\bar G^n$), that connects $c_0$ to $c_1$. 
\end{remunerate}
\end{theorem}
\begin{proof}
Taking into account the linear constraints, this result follows directly from the corresponding completeness result for the component curves, c.f.~Theorem~\ref{thm:completeness_curve}.
\end{proof}

From expression \eqref{eq:def_bar_G} of the Riemannian metric  on $\GraphnA$, we observe that the metric (and consequently the induced geodesic distance) is not only reparametrization invariant, but also invariant to permutations of the shape graph components, i.e., it remains unchanged when applying a reordering of the components of $c$. This follows from the fact that the Riemannian metric is defined as a sum of the individual Riemannian metrics of the components.

This allows us to consider the induced geometry on the space of unparametrized shape graphs, i.e., on the quotient space 
\begin{equation}
\GraphrA \Big/  \prod_{k=1}^K \Diff^{r}([0,1]) \Big/ \operatorname{Sym}(K),
\end{equation}
where $\operatorname{Sym}(K)$ denotes the permutation group.
This space is not a smooth manifold, but it is a Hausdorff topological space; this can be seen similarly as in~\cite{bruveris7completeness}. In the following we will denote elements of the quotient space by $[c]$, i.e., equivalence classes of parametrized shape graphs.

Using the reparametrization and permutation invariance of the geodesic distance $\overline{\operatorname{dist}}_{\bar G^n}$ induced by the metric $\bar G^n$, it follows that $\overline{\operatorname{dist}}_{\bar G^n}$ descends to a metric on the quotient space, which is given by
\begin{align*}
 \overline{\operatorname{dist}}_{\bar G^n}([c_0],[c_1])= 
 \inf_{\substack{
\varphi \in \prod_{k=1}^K \Diff^{r}([0,1])\\ \sigma\in \operatorname{Sym}(K)}} \overline{\operatorname{dist}}_{\bar G^n}(c_0, c_1 \circ \varphi\circ\sigma)
\end{align*}

The following theorem studies the completeness properties of the resulting metric space:
\begin{theorem}
Let $n\geq 2$. Then  
\begin{equation}
\left(\GraphrA \Big/  \prod_{k=1}^K \Diff^{r}([0,1]) \Big/ \operatorname{Sym}(K),  \overline{\operatorname{dist}}_{\bar G^n}\right)  
\end{equation} 
is a length space and any two shape graphs in the same connected component can be joined by a minimizing geodesic.
\end{theorem}
\begin{proof}
This follows from Theorem~\ref{thm:grap:completeness}, using the same arguments as~\cite{bruveris7completeness}. Note that the permutation group does not lead to any additional difficulties, as this is only a finite group.
\end{proof}

Thus to obtain a distance on the space of unparametrized shape graphs modulo permutations of their components, it is necessary to optimize over both the product reparametrization group and over all permutations of the component curves, which a priori presents a significant computational challenge. We will discuss how we address this issue indirectly in the next section by introducing a relaxation term for the terminal matching constraint that is fully blind to both these actions.

\begin{remark}\label{remark:shape_graphs_modulo_rotations}
We can also consider the space of unparametrized shape graphs modulo rotations, where the rotation operation acts on a given shape graph by rotating each of its component curves by a given angle of rotation. Since our Riemannian metric is also invariant with respect to this finite dimensional group action, it also descends to a Riemannian metric on this quotient space. Computing the induced geodesic distance  would thus involve a minimization over the rotation group $\operatorname{SO}(d)$ in addition to minimizing over the product reparametrization group and permutation group.
\end{remark}

%------------------- RELAXED MATCHING PROBLEM ------------------------------%
\section{Relaxed matching problem}
\label{sec:relaxed_matching_problem}
For closed or open curves and a fortiori in the more general setting of shape graphs as introduced in Section \ref{ssec:shape_graphs}, computing the Riemannian distance between pairs of unparametrized shapes requires solving a matching problem involving an optimization over paths of immersions and reparametrizations, plus permutations of the components for the case of shape graphs. In practice, paths of immersions can be discretized by considering piecewise linear curves or more generally splines as in \cite{bauer2017numerical}, and thus the minimization over those paths can be framed quite naturally as a standard finite dimensional optimization problem, as we will outline in Section \ref{sec:optim_algo}. However, dealing with reparametrizations is typically more difficult as the minimization is on an infinite dimensional group, and discretizing such a group and its action on curves is not straightforward, c.f. the discussions in \cite{trouve2000diffeomorphic,mio2007shape,bauer2017numerical}. Recently, an alternative approach was proposed for closed and open curves \cite{bauer2019relaxed} in which this minimization over reparametrizations of $c_1$ in \eqref{eq:quotient_distance} is dealt with indirectly by instead introducing a relaxation of the end time constraint using a parametrization blind data attachment term. At a high level, this consists in considering the relaxed matching problem:
\begin{equation*}
 \inf\left\{\int_0^1 G_{c(t)}^{n}(\partial_t{c}(t),\partial_t{c}(t)) dt  + \lambda \Discr(c(1),c_1) \right\} ,
\end{equation*}
where the minimization is here only over paths $c(\cdot) \in H^1([0,1],\Imm^n(D,\R^d))$ that satisfy the initial constraint $c(0) = c_0$, and $\Discr(c(1),c_1)$ is a measure of discrepancy between the deformed source $c(1)$ and the true target curve $c_1$, with $\lambda>0$ being a balancing parameter. Formally, if $\Discr$ is independent of the parametrization of either of the two curves, then imposing $\Discr(c(1),c_1) \approx 0$ yields $c(1)\approx c_1 \circ \varphi$, which approximates the end time constraint in \eqref{eq:quotient_distance} without the need to actually model the reparametrization itself. Furthermore, it allows for inexact matching when computing the distance, which will prove even more critical when extending this framework to shape graphs that can exhibit different topologies and involve only partial correspondences. In addition, if $\Discr$ is also independent of the ordering of the curves in either of the shape graphs, the relaxed matching problem will also allow us to circumvent the minimization over the finite (but potentially large) permutation group.

\subsection{Varifold representation and distance} 
\label{ssec:var_metrics}
We now describe how to construct the key ingredient in the relaxed model described above, namely, an effective and simple to compute data attachment term $\Discr$ to compare unparametrized shapes. The concepts introduced in the field of geometric measure theory, such as currents, varifolds or normal cycles provide very useful frameworks to that end \cite{glaunes2008large,charon2013varifold,kaltenmark2017general,charon2020fidelity}. In particular, metrics based on the oriented varifold setting of \cite{kaltenmark2017general} were used as the discrepancy term $\Discr$ in the authors' previous work \cite{bauer2019relaxed}. We will see that these metrics and their nice properties can be even further exploited for shape graphs. First, we shall briefly recap the construction of varifold metrics for curves.

Using the notations and setting of Section \ref{sec:curves}, the central idea is to embed immersed parametrized curves of $\Imm^n(D,\R^d)$ for $n\geq 2$ into the space of positive Radon measures on $\R^d \times S^{d-1}$, where $S^{d-1}$ is the unit sphere of $\R^d$. This space is known as the space of oriented 1-varifolds, which we will denote by $\V$ in what follows. This embedding is specifically defined as follows: given $c \in \Imm^n(D,\R^d)$, we can consider the varifold $\mu_c \in \V$ which is given for any Borel set $A \subset \R^d \times S^{d-1}$ by 
\begin{equation*}
    \mu_c(A) = \int_{D} \mathds{1}_{\left(c(\theta),\frac{\partial_{\theta}c(\theta)}{|\partial_{\theta}c(\theta)|}\right)\in A} \, |\partial_{\theta} c(\theta)| d\theta.
\end{equation*}
Alternatively, it will be more convenient to view this measure through its action on a general test function $\omega \in C_0(\R^d \times S^{d-1},\R)$ (the space of continuous functions going to $0$ at infinity) :
\begin{equation}
\label{eq:def_mu_c}
    (\mu_c | \omega) = \int_{D} \omega\left(c(\theta),\frac{\partial_{\theta} c(\theta)}{|\partial_{\theta} c(\theta)|}\right) |\partial_{\theta} c(\theta)| d\theta.
\end{equation}
This mapping $\Imm^n(D,\R^d) \ni c \mapsto \mu_c \in \V$ has several important properties. First, it is invariant to reparametrizations; indeed a simple change of variables in the integral \eqref{eq:def_mu_c} leads to $\mu_{c \circ \varphi} = \mu_{c}$ for all $\varphi \in \Diff^n(D)$. As a consequence, this varifold mapping descends to a well-defined mapping on the quotient space of unparametrized curves. Furthermore, the measure $\mu_c$ uniquely determines the equivalence class $c \circ \Diff^n(D)$ under some mild technical assumptions (see Theorem 3.6 in \cite{bauer2019relaxed} for specific details). 

Consequently, it becomes possible to compare unparametrized curves by introducing a metric on the space of varifolds $\V$. A quite general class of such metrics is in particular obtained by considering positive definite kernels on $\R^d \times S^{d-1}$ and their associated Reproducing Kernel Hilbert Spaces (RKHS) norms, which was the construction proposed in \cite{charon2013varifold,kaltenmark2017general}. Following the same approach and without going through all the details, we consider a kernel function $\K(x,u,y,v)$ for $x,y \in \R^d$ and $u,v \in S^{d-1}$, whose associated RKHS we write $\H$. We make the assumption that the kernel $\K$ is such that $\H$ is continuously embedded into the space $C^1_0(\R^d \times S^{d-1},\R)$ of continuously differentiable functions on $\R^d \times S^{d-1}$ which vanish at infinity (as well as their first derivatives), equipped with the norm $\|\omega\|_{1,\infty} = \|\omega\|_{\infty} + \|d\omega\|_{\infty}$. Examples of such kernels are given in \cite{kaltenmark2017general,bauer2019relaxed} and we shall discuss specific choices in Section \ref{ssec:varifold_norm_computation}. This kernel then induces a Hilbert norm on $\H$ and the embedding assumption implies that there exists $C_{\H}>0$ such that for all $\omega \in \H$, we have $\|\omega\|_{1,\infty} \leq C_{\H} \|\omega\|_{\H}$. In turn, the dual norm leads to the following Hilbert metric on the space of measures of $\R^d \times S^{d-1}$ i.e. on $\V$:
\begin{equation}
\label{eq:norm_var}
    \|\mu\|_{\V} = \sup_{\omega \in \H,\|\omega\|_{\H}\leq 1} (\mu|\omega).
\end{equation}
Note that, although the notation does not emphasize it, $\|\cdot\|_{\V}$ depends on the kernel $\K$. In fact, given two curves $c_0,c_1$ in $\Imm^n(D,\R^d)$, it can be shown from the reproducing kernel property that the inner product between their associated varifolds can be expressed explicitly as: 
\begin{equation}
\label{eq:inner_prod_var_curve}
    \langle \mu_{c_0}, \mu_{c_1} \rangle_{\V} = \iint_{D\times D} \mathcal{K}\left(c_0(\theta),\frac{c_0(\theta)}{|c_0(\theta)|},c_1(\tilde \theta),\frac{c_1(\tilde \theta)}{|c_1(\tilde \theta)|}\right) |\partial_{\theta} c_0(\theta)| |\partial_{\theta} c_1(\tilde \theta)| d\theta d\tilde \theta.
\end{equation}
Then, by setting $\Delta(c_0,c_1) = \|\mu_{c_1}-\mu_{c_0}\|_{\V}^2 = \|\mu_{c_0}\|_{\V}^2 +\|\mu_{c_1}\|_{\V}^2 - 2 \langle \mu_{c_0}, \mu_{c_1} \rangle_{\V}$, we obtain a discrepancy term that is independent of the parametrizations of $c_0$ and $c_1$, and which can be expressed explicitly from \eqref{eq:inner_prod_var_curve}. Note that in general $\|\cdot\|_{\V}$ may only be a pseudo-metric on $\V$ but with adequate extra properties on the kernel $\K$ (namely $C_0$-universality), one can recover a true metric. We refer the reader to the results in \cite[Section 3.2]{hsieh2021metrics} and \cite[Section 3.2]{bauer2019relaxed} for more details. 

\subsection{Relaxed shape graph matching}
\label{ssec:relaxed_shape_graph_matching}
A further advantage of using varifolds to construct the parametrization invariant data attachment term is that they provide an adequate measure of discrepancy between shape graphs having arbitrary adjacency matrices. This follows mainly from the additive property of the varifold representation. For $n\geq 2$ and an adjacency matrix $\Adj$, given $c \in \GraphnA$ with $c = \prod_{k=1}^K c^k$, the varifold associated to $c$ can be simply defined as $\mu_c = \sum_{k=1}^{K} \mu_{c^k}  \in \V$, with each $\mu_{c^k}$ being the varifold representation of the curve $c^k$ defined as in Section \ref{ssec:var_metrics}. Then, the dual RKHS metric $\|\cdot\|_{\V}$ defined above extends directly to shape graphs. Due to the invariance properties of $\|\cdot\|_{\V}$, the resulting distance induces a (pseudo-) distance on the quotient space of unparametrized shape graphs with $K$ components. This follows from the observation that the varifold construction is invariant to the ordering of the curves of a shape graph, and to reparametrizations of these component curves.

Remarkably, $\|\cdot\|_{\V}$, and hence the varifold data attachment term, remain well-defined on the entire set $\Graphn$ of all shape graphs having an arbitrary number of components and arbitrary adjacency matrices. If $c_0 = \prod_{k=1}^K c_0^k$ and $c_1 = \prod_{l=1}^L c_1^l$ are two shape graphs with $K$ and $L$ components respectively, we obtain that $\Delta(c_0,c_1) = \|\mu_{c_0}\|_{\V}^2 +\|\mu_{c_1}\|_{\V}^2 - 2 \langle \mu_{c_0}, \mu_{c_1} \rangle_{\V}$ with
\begin{equation}
    \label{eq:inner_prod_var_graph}
    \langle \mu_{c_0}, \mu_{c_1} \rangle_{\V} = \sum_{k=1}^{K} \sum_{l=1}^{L} \langle \mu_{c^k_0}, \mu_{c^l_1} \rangle_{\V} ,
\end{equation}
where each $\langle \mu_{c^k_0}, \mu_{c^l_1} \rangle_{\V}$ can be expressed through \eqref{eq:inner_prod_var_curve}. 

This finally allows us to formulate the following relaxed shape graph matching problem:
\begin{framed}
Given $c_0,c_1 \in \Graphn$ and $\Adj$ the adjacency matrix of $c_0$, we consider the variational problem:
\begin{equation}
    \label{eq:relaxed_shape_graph_match}
    \inf\left\{\int_0^1 \bar{G}_{c(t)}^{n}(\partial_t{c}(t),\partial_t{c}(t)) dt  + \lambda \|\mu_{c(1)} - \mu_{c_1}\|_{\V}^2 \right\} 
\end{equation}
where the infimum is taken over all paths $c(\cdot) \in H^1([0,1],\GraphnA)$ satisfying the initial constraint $c(0)=c_0$. Note that in the relaxed shape graph matching framework, we refer to $c_0$ as the source shape graph, $c_1$ as the target shape graph, and $c(1)$ as the deformed source.
\end{framed}

Note that for all $t\in [0,1]$, $c(t)$ is by construction a shape graph with adjacency matrix $\Adj$ which may differ from the one of the target shape graph $c_1$. Thus, we can interpret \eqref{eq:relaxed_shape_graph_match} as the problem of finding the optimal path with respect to the metric $\bar{G}^n$ between the source shape $c_0$ and a shape graph $c(1)$ having the same graph structure as $c_0$, which is ``close'' to the target $c_1$ in terms of the varifold distance. As a particular case, when $K=L=1$, \eqref{eq:relaxed_shape_graph_match} reduces to the relaxed curve matching problem introduced in \cite{bauer2019relaxed}.

\subsection{Existence of minimizers} 
The existence of a minimizing path $c$ in \eqref{eq:relaxed_shape_graph_match} is not a priori guaranteed since, unlike with the existence of geodesics given by Theorem \ref{thm:grap:completeness}, the deformed source $c(1)$ is not fixed anymore in this relaxed variational problem. In fact, this very question was also left aside by the authors in \cite{bauer2019relaxed} for the special case of open and closed curves. In this section, we will show that the existence of minimizers holds for the scale invariant metric \eqref{eq:def_Sobolev_metric_length_inv} of regularity $n\geq 2$:
\begin{theorem}
\label{thm:existence_graph_match}
 Let $c_0,c_1 \in \Graphn$ and $\Adj$ the adjacency matrix of $c_0$. Assume that $n\geq 2$ and $\bar{G}^n$ is a scale invariant metric on $\GraphnA$. We further assume that $\H$ is continuously embedded into $C_0^1(\R^{d}\times \Sp^{d-1})$. Then the infimum in \eqref{eq:relaxed_shape_graph_match} is achieved by a path $c(\cdot) \in  H^1([0,1],\GraphnA)$.
\end{theorem}
\begin{remark}
For $K=1$, i.e., when $c_0$ is a single open curve, Theorem \ref{thm:existence_graph_match} gives in particular the existence of solutions to the relaxed open curve matching problem introduced in \cite{bauer2019relaxed} for the class of scale-invariant metrics of order $n\geq 2$. We point out that the  proof can be adapted almost verbatim to also recover the existence of solutions in the case of a single closed curve for both the constant coefficient and scale invariant Sobolev metrics \eqref{eq:def_Sobolev_metric} and \eqref{eq:def_Sobolev_metric_length_inv} as long as $n\geq 2$. 
\end{remark}

Our proof will follow the standard approach from the calculus of variations.  We will  need the following lemma on the convergence of varifold norms, whose proof we postpone to the appendix. 
\begin{lemma}
\label{lemma:var_norm_continuity}
 Let $(c_p)_{p \in \N}$ be a sequence of $C^1$ immersions of $D$ (with $D=[0,1]$ or $D=S^1$) into $\R^d$ such that $c_p$ converges to a $C^1$ immersion $c_{\infty}$ for the $\|\cdot\|_{1,\infty}$ norm. Then $\mu_{c_p}$ converges to $\mu_{c_\infty}$ in $\V$. 
\end{lemma}

\begin{proof}[Proof of Theorem~\ref{thm:existence_graph_match}]
Let us define the energy of \eqref{eq:relaxed_shape_graph_match}:
\begin{equation*}
    E(c) = \int_0^1 \bar{G}_{c(t)}^{n}(\partial_t{c}(t),\partial_t{c}(t)) dt  + \lambda \|\mu_{c(1)} - \mu_{c_1}\|_{\V}^2 .
\end{equation*}
Note that the infimum of $E$ is finite as one can for instance consider the constant path $c(t)=c_0$ for all $t \in [0,1]$ for which $E(c)=\|\mu_{c_0} - \mu_{c_1}\|_\V^2<+\infty$. Then, introduce a minimizing sequence $(\tilde{c}_p)_{p \in \N}$ in $H^1([0,1],\GraphnA)$, i.e. $E(\tilde{c}_p)\rightarrow \inf E(c)<+\infty$. We denote the different component curves of the shape graph $\tilde{c}_p(t)$ by ${\tilde{c}}^k_p(t)$. Both terms in the energies $E(\tilde{c}_p)$ are bounded, meaning that there exists $\delta>0$ such that for all $p \in \mathbb{N}$ and $k=1,\ldots,K$: 
\begin{equation*}
  \int_0^1 G_{\tilde{c}^k_p(t)}^{n}(\partial_t{\tilde{c}}^k_p(t),\partial_t{\tilde{c}}^k_p(t)) dt \leq   \int_0^1 \bar{G}_{\tilde{c}_p(t)}^{n}(\partial_t{\tilde{c}}_p(t),\partial_t{\tilde{c}}_p(t)) dt < \delta .
\end{equation*}
Since $\tilde{c}^k_p(0) = c^k_0$, we obtain that for all $t\in[0,1]$, $\tilde{c}^k_p(t)$ belongs to the ball $B(c^k_0,\delta)$ for the Sobolev distance $\operatorname{dist}_{G^n}$. It follows from Lemma \ref{lemma:bounds_Sobolev_metrics} that there exists $C>0$ such that
\begin{equation*}
    \|\partial_t{\tilde{c}}_p^k(t)\|^2_{H^n}\leq C G_{\tilde{c}^k_p(t)}^{n}(\partial_t{\tilde{c}}^k_p(t),\partial_t{\tilde{c}}^k_p(t))
\end{equation*}
and $|\partial_{\theta}{\tilde{c}}^k_p(t,\theta)|\geq C^{-1}$ for all $p\in \mathbb{N}$, $k=1,\ldots,K$, $t\in [0,1]$ and $\theta \in [0,1]$. Since ${\tilde{c}}^k_p(t) = c_0^k + \int_0^t {\partial_t{\tilde{c}}}^k_p(u) du$, we get that $\|{\tilde{c}}^k_p(t)\|_{H^n} \leq \|c_0\|_{H^n} + \sqrt{C \delta}$ and thus 
\begin{equation*}
\|{\tilde{c}}^k_p\|_{H^1([0,1],H^n([0,1])}^2 \leq (\|c_0\|_{H^n}+\sqrt{C\delta})^2+C\delta.
\end{equation*}
Therefore the sequence $({\tilde{c}}^k_p)_{p\in\mathbb{N}}$ is bounded in $H^1([0,1],H^n([0,1],\R^d))$, i.e., up to extracting a subsequence, ${\tilde{c}}^k_p$ converges weakly to some $\tilde{c}^k$ in $H^1([0,1],H^n([0,1],\R^d))$. As $H^1([0,1],H^n([0,1],\R^d))$ is compactly embedded into $C([0,1],C_0^1([0,1]))$ by the Aubin-Dubinskii lemma \cite{amann2000compact}, we deduce that for all $t\in[0,1]$, ${\tilde{c}}^k_p(t)$ converges to $\tilde{c}^k(t)$ in $\|\cdot\|_{1,\infty}$ on $[0,1]$. Also, the above lower bound $|\partial_{\theta}{\tilde{c}}^k_p(t,\theta)|\geq C^{-1}$ leads to $|\partial_{\theta}\tilde{c}^k(t,\theta)|\geq C^{-1}$ for all $t$ and $\theta$, from which we deduce that $\tilde{c}^k \in H^1([0,1],\Imm^n([0,1],\R^d)$ for each $k=1,\ldots,K$. Furthermore, the convergence in $\|\cdot\|_{1,\infty}$ of the $c^k_p$ also implies that the set of linear constraints defining the graph structure given by $\Adj$ via \eqref{conn1},\eqref{conn2} and \eqref{conn3} are also satisfied by the limit curves $\tilde{c}^k$. Thus, defining $\tilde{c}(t) = \prod_{k=1}^{K} \tilde{c}^k(t)$, we have that $\tilde{c} \in H^1([0,1],\GraphnA)$. 

Moreover, by Lemma \ref{lemma:var_norm_continuity}, as $\tilde{c}^k_p(1) \xrightarrow[p\rightarrow \infty]{} \tilde{c}^k(1)$ in $\|\cdot\|_{1,\infty}$ for all $k$, we get that $\mu_{{\tilde{c}}^k_p(1)} \rightarrow \mu_{\tilde{c}^k(1)}$ in $\V$. Then $\mu_{\tilde{c}_p}= \sum_{k=1}^{K} \mu_{\tilde{c}^k_p} \xrightarrow[\|\cdot\|_{\V}]{} \sum_{k=1}^{K} \mu_{\tilde{c}^k} = \mu_{\tilde{c}}$ and as a result $\|\mu_{\tilde{c}_p(1)} - \mu_{c_1}\|_\V \xrightarrow[p\rightarrow \infty]{} \|\mu_{\tilde{c}(1)} - \mu_{c_1}\|_\V$. Finally, as follows from the proof of Theorem 5.2 in \cite{bruveris7completeness}, the mapping $c\mapsto \int_0^1 G^n_{c(t)}(\partial_t{c}(t),\partial_t{c}(t)) dt$ is weakly lower semicontinuous on $H^1([0,1],H^n([0,1],\R^d))$, from which we deduce that the first term in the energy $E$ is weakly lower semicontinuous with respect to the convergence of $\tilde{c}_p$ to $\tilde{c}$ in $\prod_{k=1}^{K} H^1([0,1],H^n([0,1],\R^d))$. We conclude that
$E(\tilde{c}) \leq \liminf\limits_{p \rightarrow +\infty} E(\tilde{c}_p) = \inf E(c)$ and consequently that $\tilde{c}$ is a minimizer of $E$.  
\end{proof}

%---------------- ELASTIC MATCHING WITH WEIGHTS ----------------------------%
\section{Elastic matching model with weights: mathematical analysis}
\label{sec:model_weights}
In this section, we extend the elastic Sobolev metric shape graph matching framework of the previous section by incorporating weights and weight changes along the shape graphs.   

\subsection{Limitations of the previous elastic matching model}
\label{ssec:limitations}
We start by motivating the need for such an extended approach. Indeed, the model presented so far is primarily built to compare shape graphs of the same topology, such as in the example shown on Fig. \ref{fig:match_2branches_fixed_weights}, which was obtained with our proposed shape graph matching algorithm (with fixed weights) that will be introduced in Section \ref{sec:optim_algo}. Although the matching in \eqref{eq:relaxed_shape_graph_match} is inexact and may in practice be able to handle small inconsistencies including topological noise, it remains inadequate for many typical datasets (e.g. trees, arterial networks) which routinely involve shape graphs with significant topological differences. Attempting to compare two such shape graphs based on \eqref{eq:relaxed_shape_graph_match} can lead to highly singular and unnatural behaviour in the estimated geodesic and distance, as illustrated in Fig. \ref{fig:match_branches_compare_fixed_vs_source_weights}. This is in great part due to the fact that our model does not yet allow for partial matching constraints. 

\begin{figure}[htbp]
\begin{center}
\begin{tabular}{ccccc}
\includegraphics[trim = 45mm 15mm 45mm 10mm ,clip,width=2.5cm]{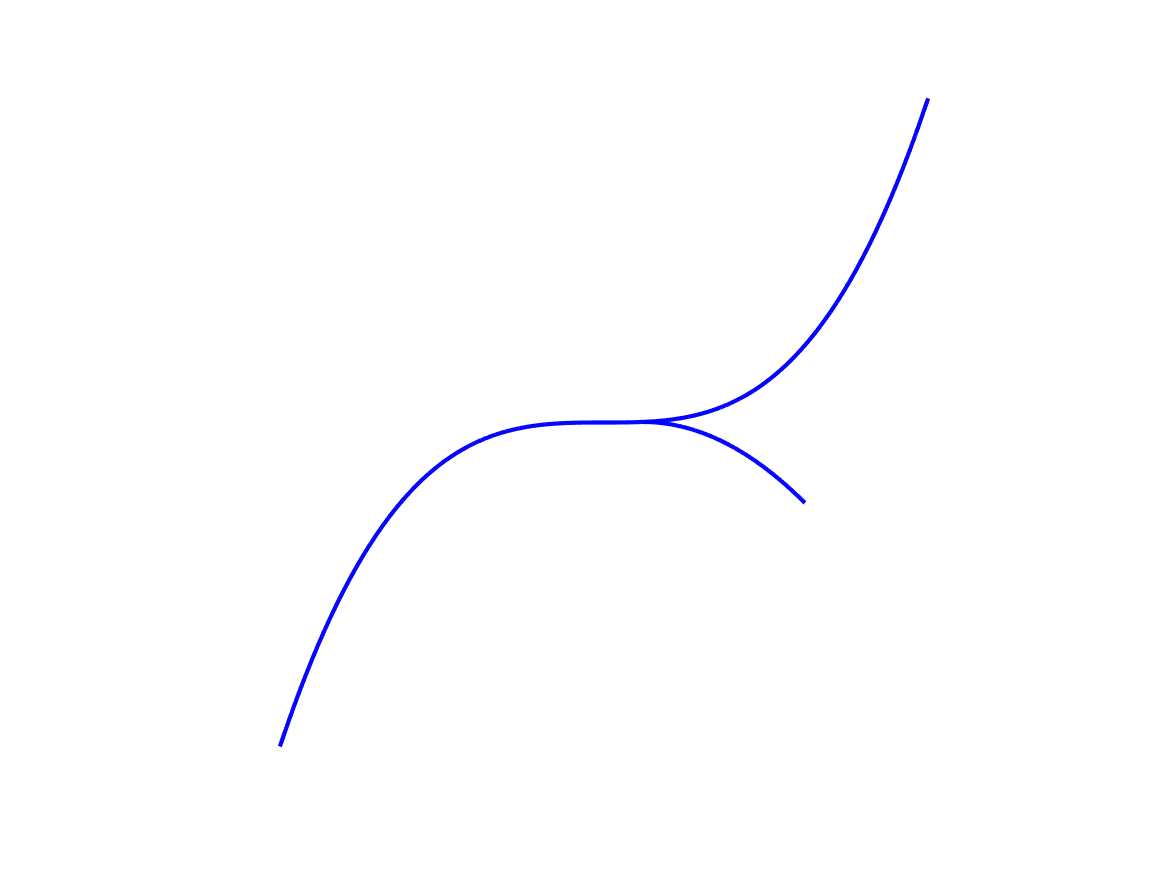}&
\includegraphics[trim = 45mm 15mm 45mm 10mm ,clip,width=2.5cm]{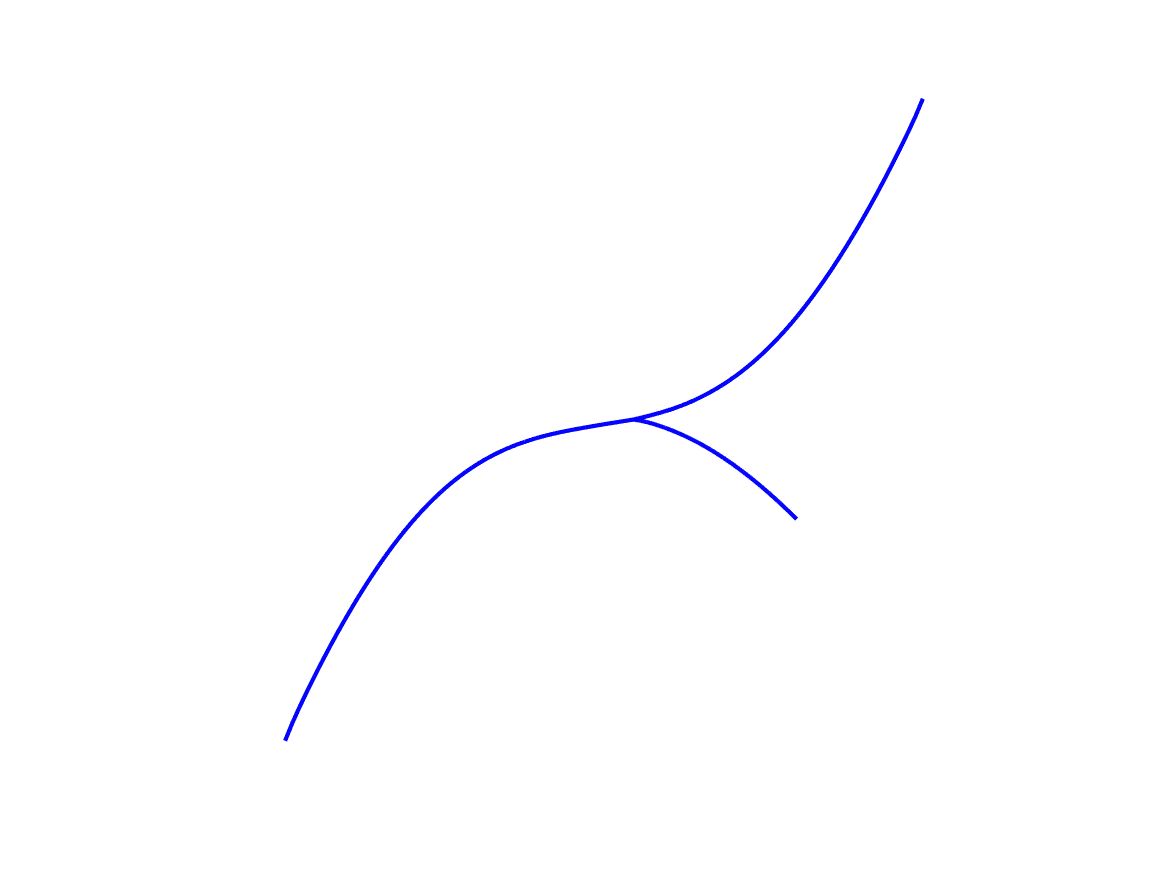}&
\includegraphics[trim = 45mm 15mm 45mm 10mm ,clip,width=2.5cm]{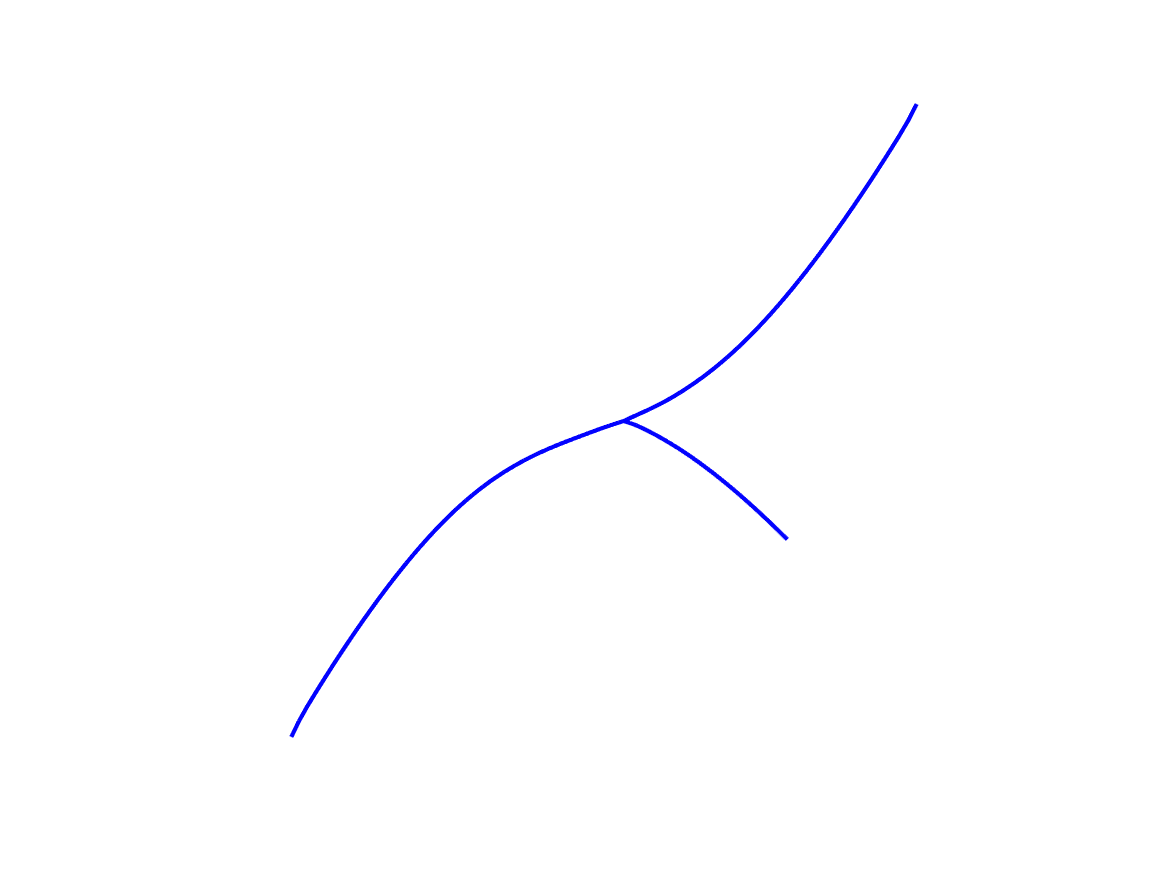}&
\includegraphics[trim = 45mm 15mm 45mm 10mm ,clip,width=2.5cm]{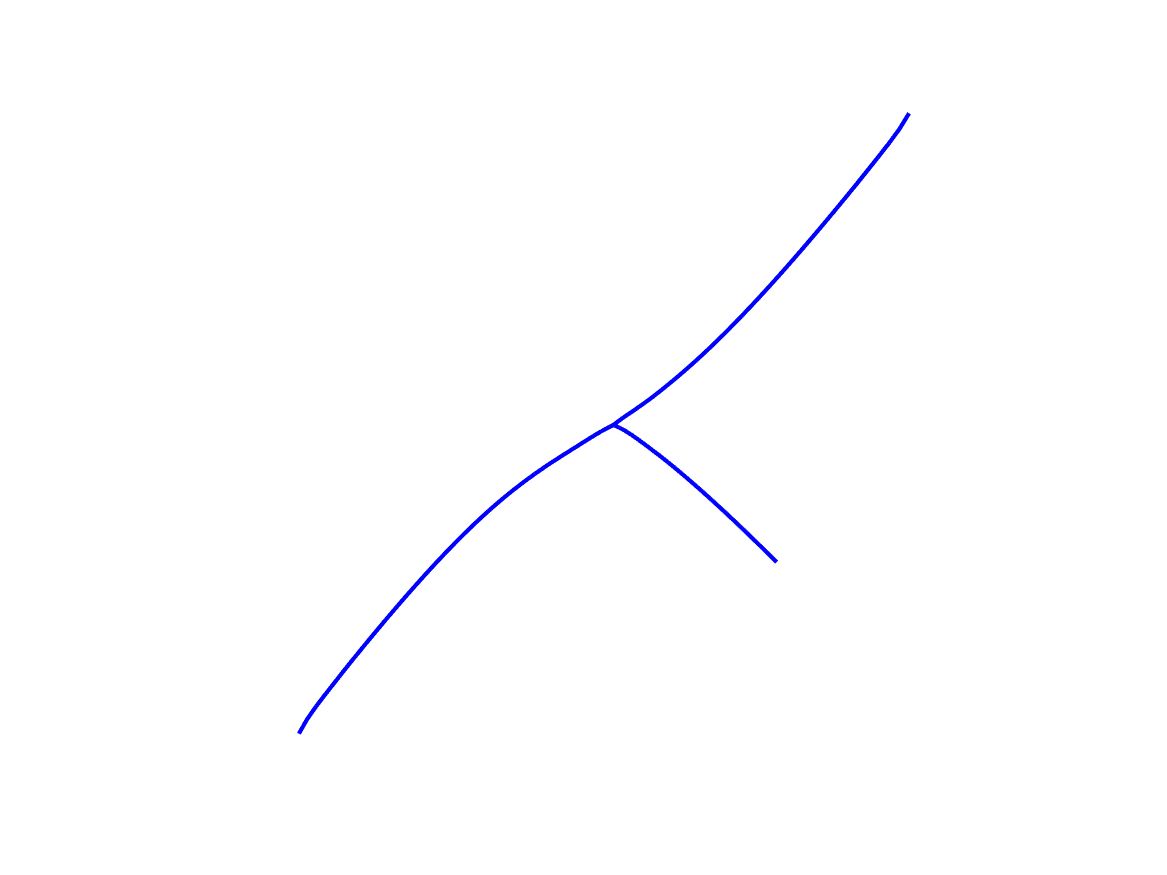}&
\includegraphics[trim = 45mm 15mm 45mm 10mm ,clip,width=2.5cm]{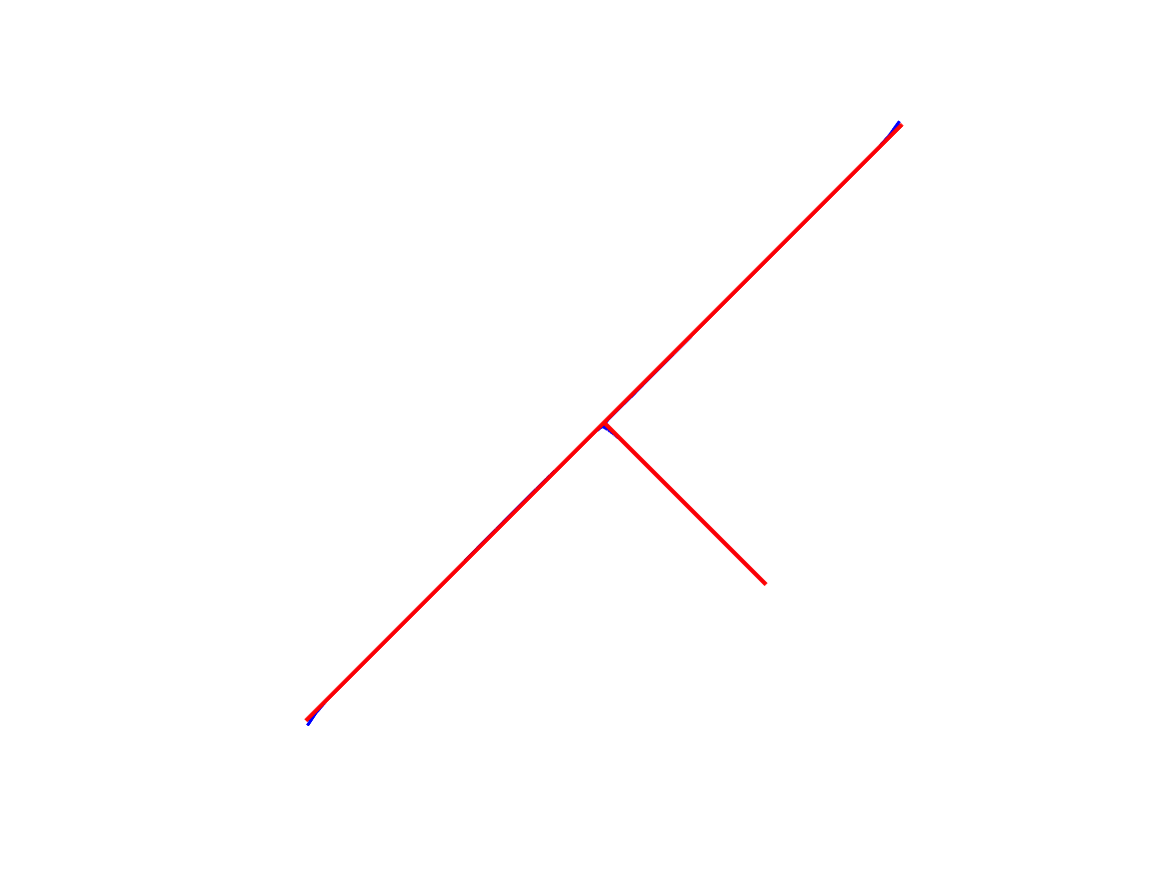}\\
$t=0$ & $t=0.25$ & $t=0.5$ & $t=0.75$ & $t=1$ 
\end{tabular}
\end{center}
\caption{Geodesic between two shape graphs with the same topology: the source $c_0$ (left) and target $c_1$ (in red on the right). The target is overlayed on the transformed source $c(1)$ at $t=1$. The estimated geodesic distance is $\overline{\operatorname{dist}}(c_0,c_1)=0.83$.} \label{fig:match_2branches_fixed_weights}
\end{figure}

\begin{figure}[htbp]
\begin{center}
\begin{tabular}{ccccc}
\includegraphics[trim = 45mm 15mm 45mm 10mm ,clip,width=2.5cm]{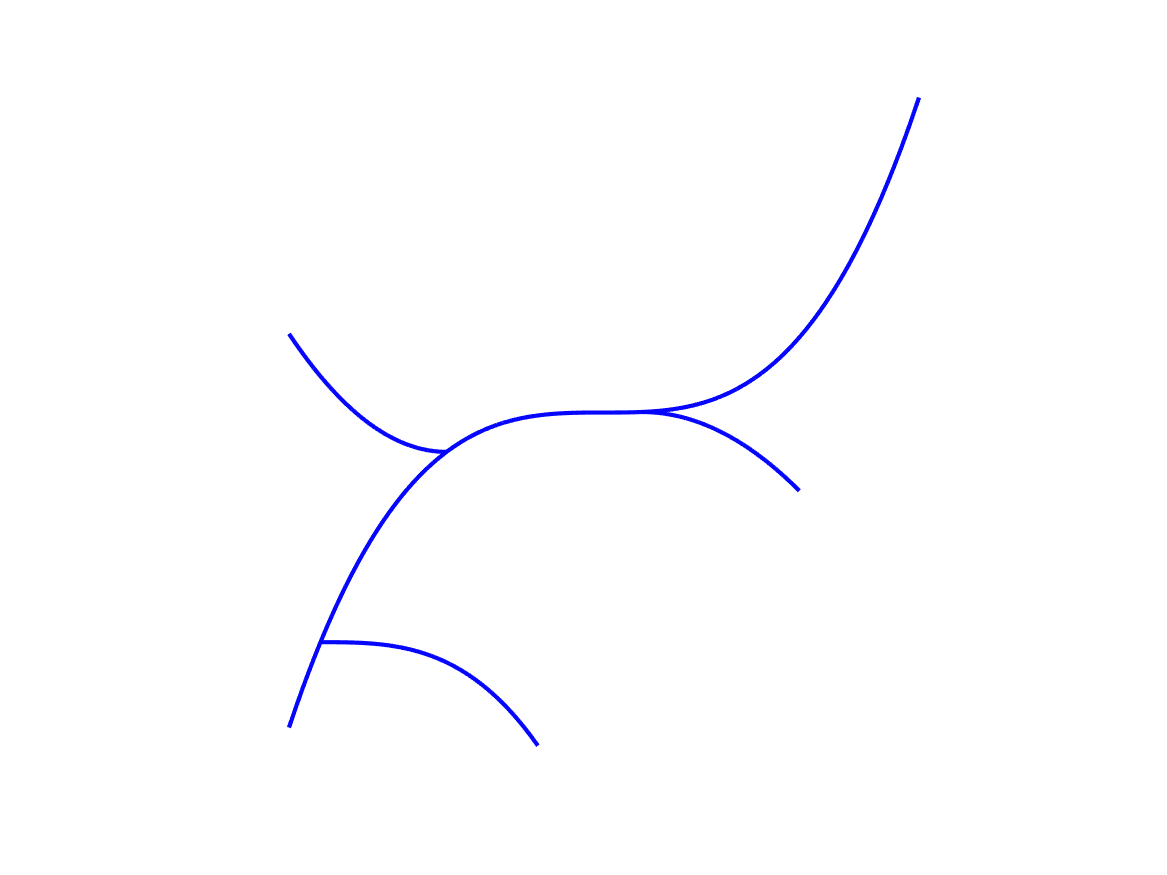}&
\includegraphics[trim = 45mm 15mm 45mm 10mm ,clip,width=2.5cm]{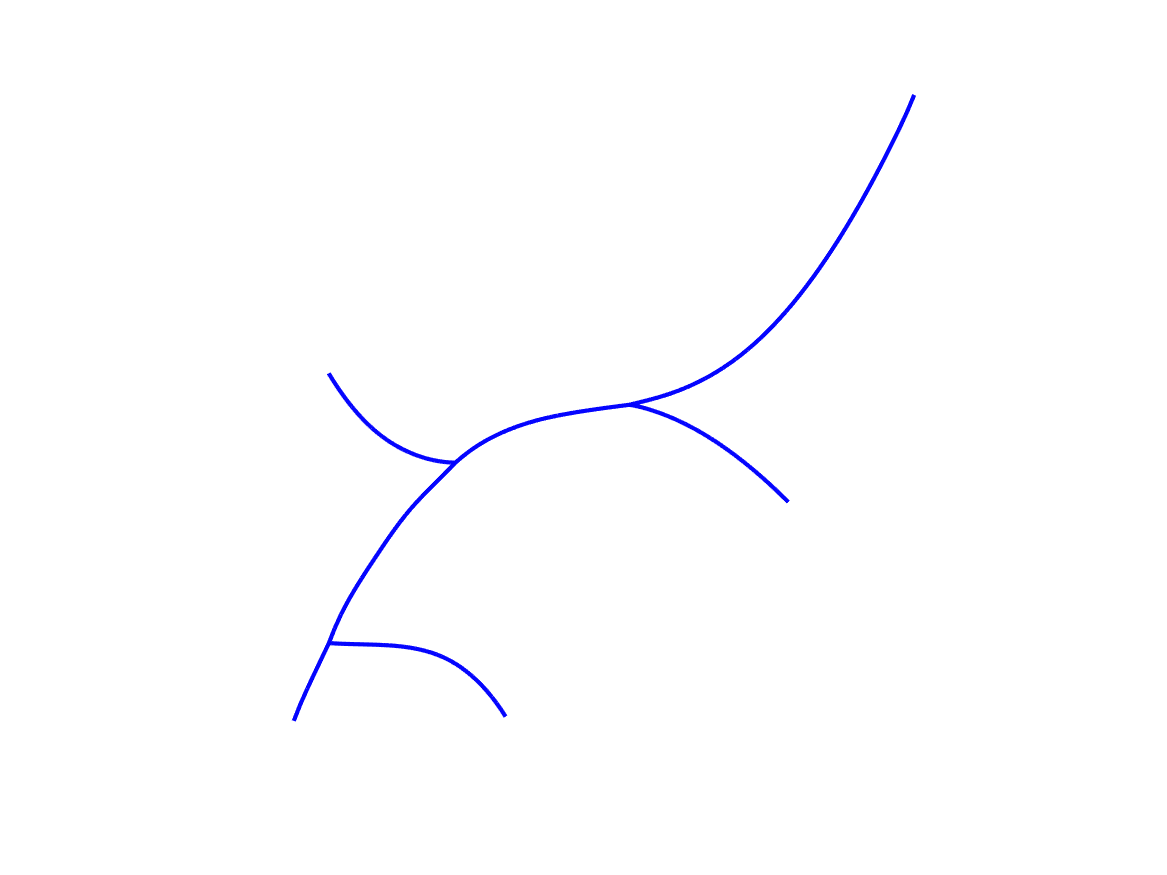}&
\includegraphics[trim = 45mm 15mm 45mm 10mm ,clip,width=2.5cm]{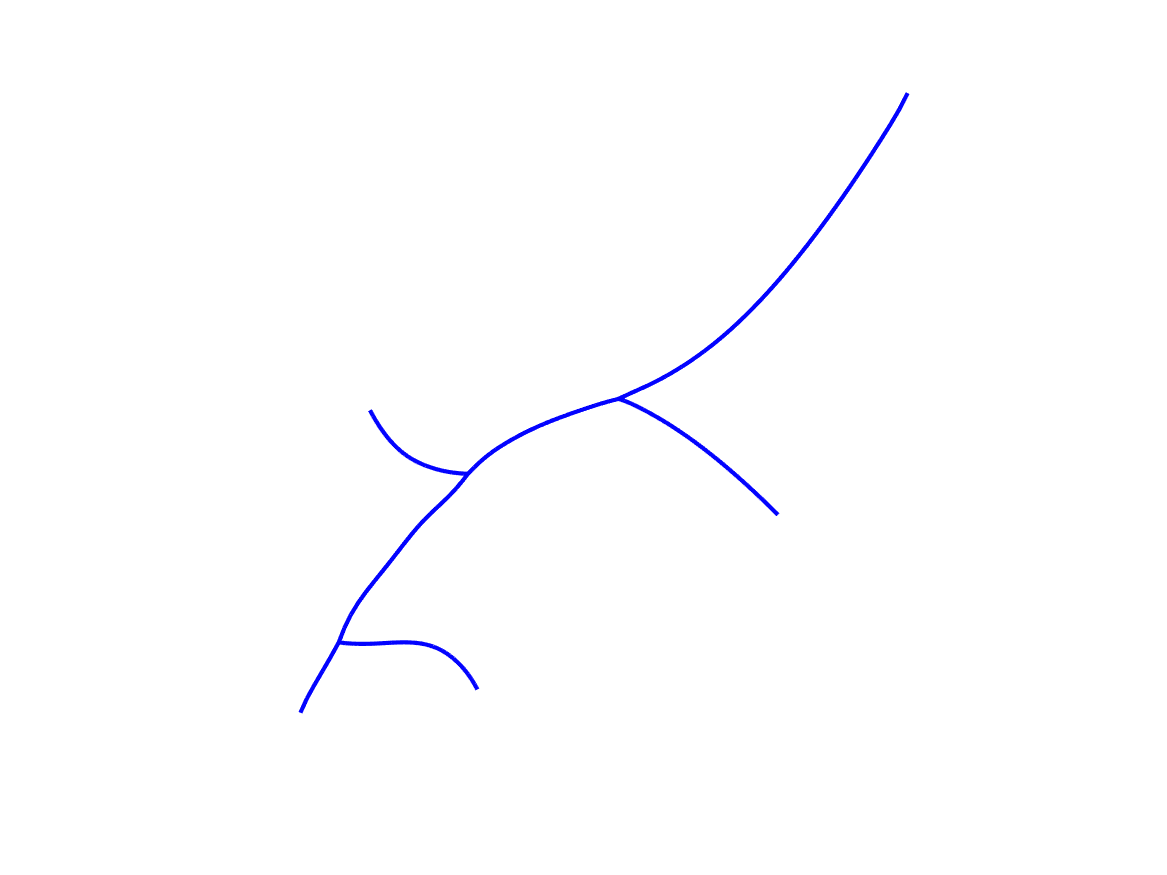}&
\includegraphics[trim = 45mm 15mm 45mm 10mm ,clip,width=2.5cm]{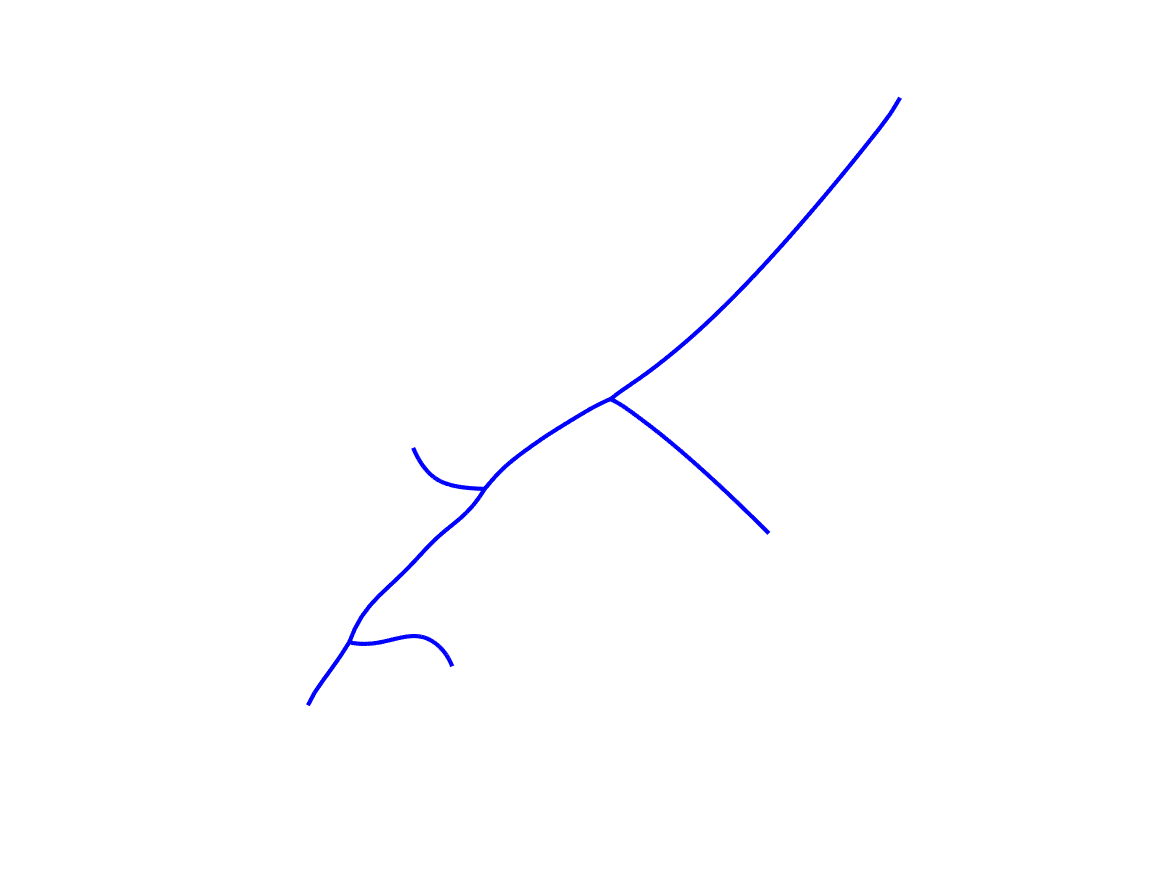}&
\includegraphics[trim = 45mm 15mm 45mm 10mm ,clip,width=2.5cm]{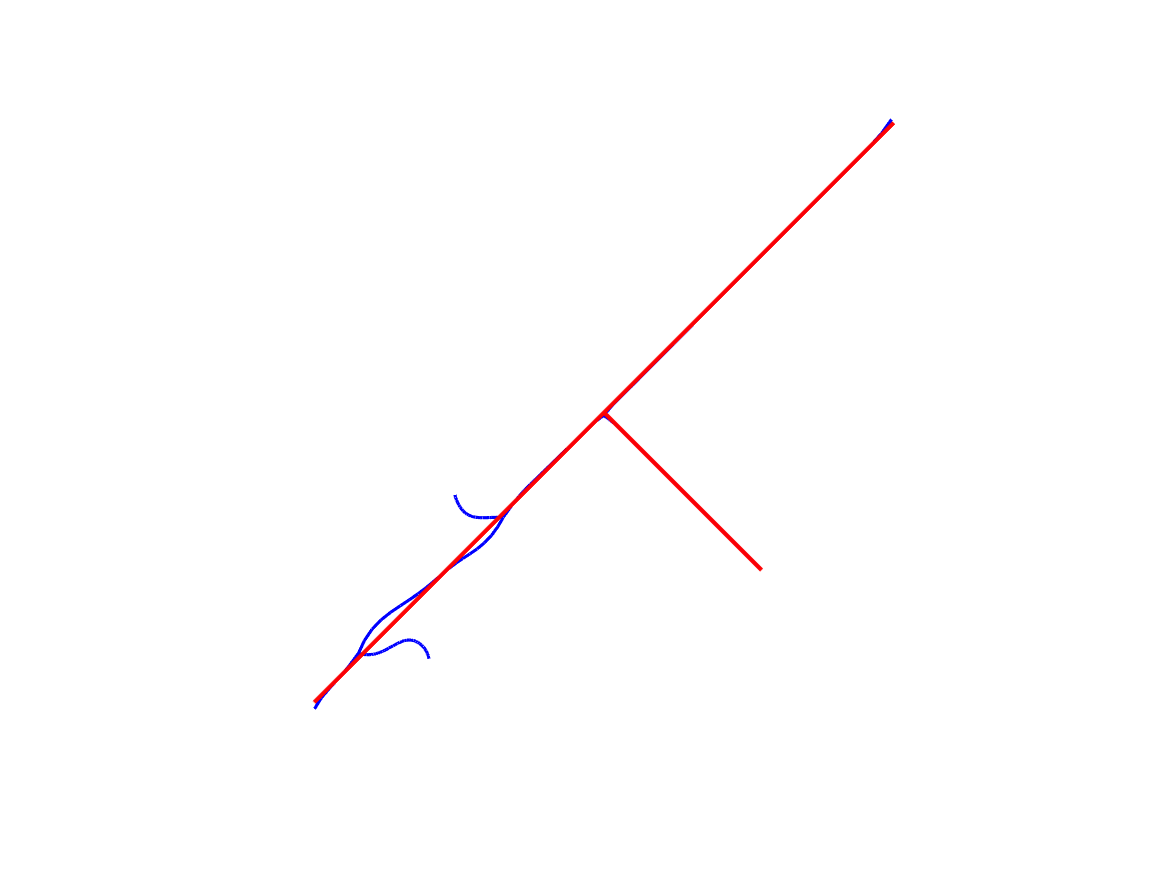}\\
\includegraphics[trim = 45mm 15mm 45mm 10mm ,clip,width=2.5cm]{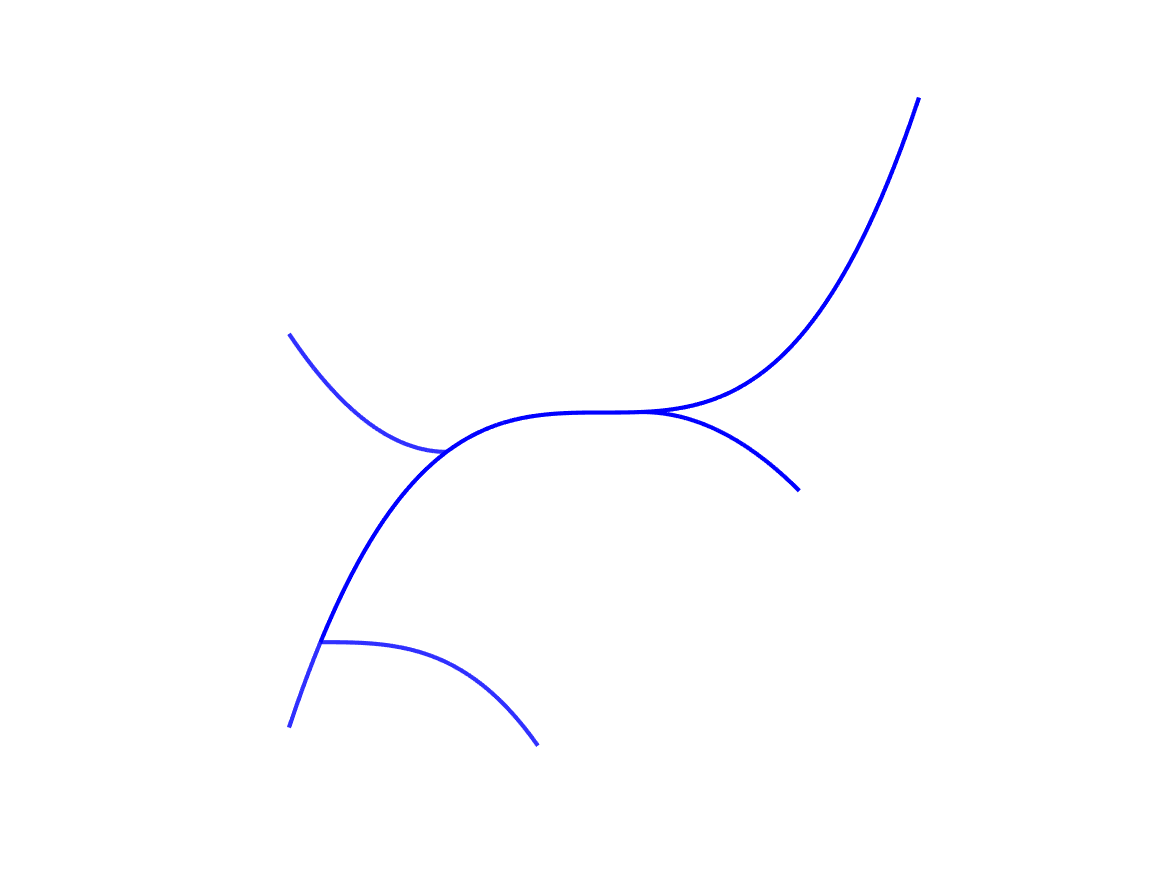}&
\includegraphics[trim = 45mm 15mm 45mm 10mm ,clip,width=2.5cm]{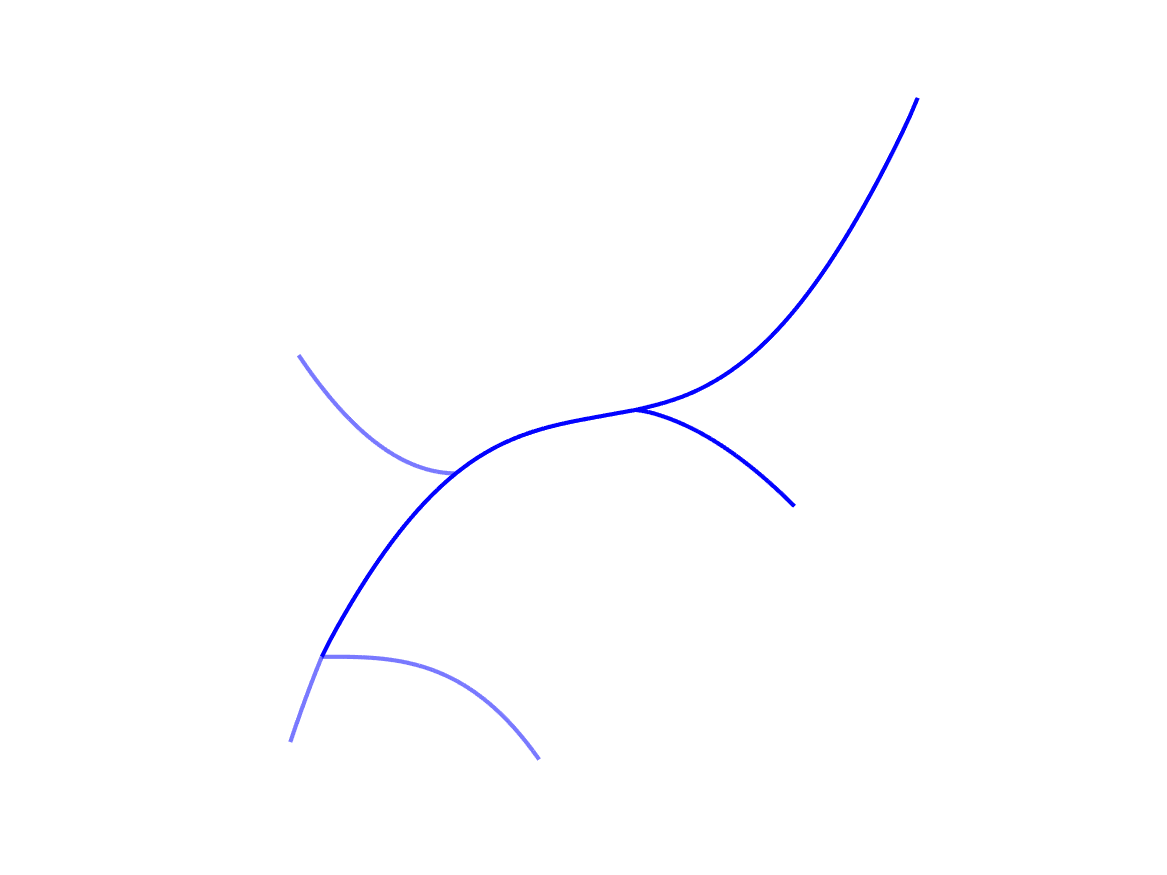}&
\includegraphics[trim = 45mm 15mm 45mm 10mm ,clip,width=2.5cm]{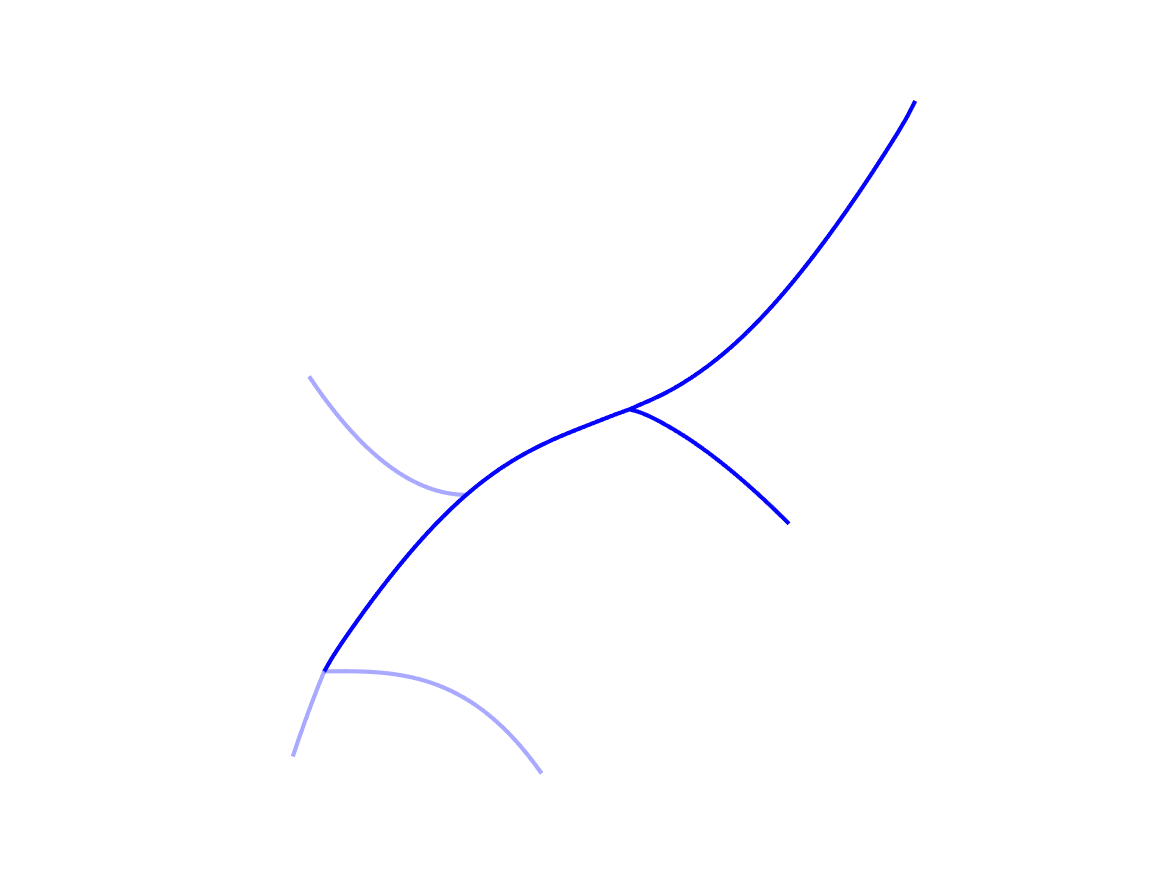}&
\includegraphics[trim = 45mm 15mm 45mm 10mm ,clip,width=2.5cm]{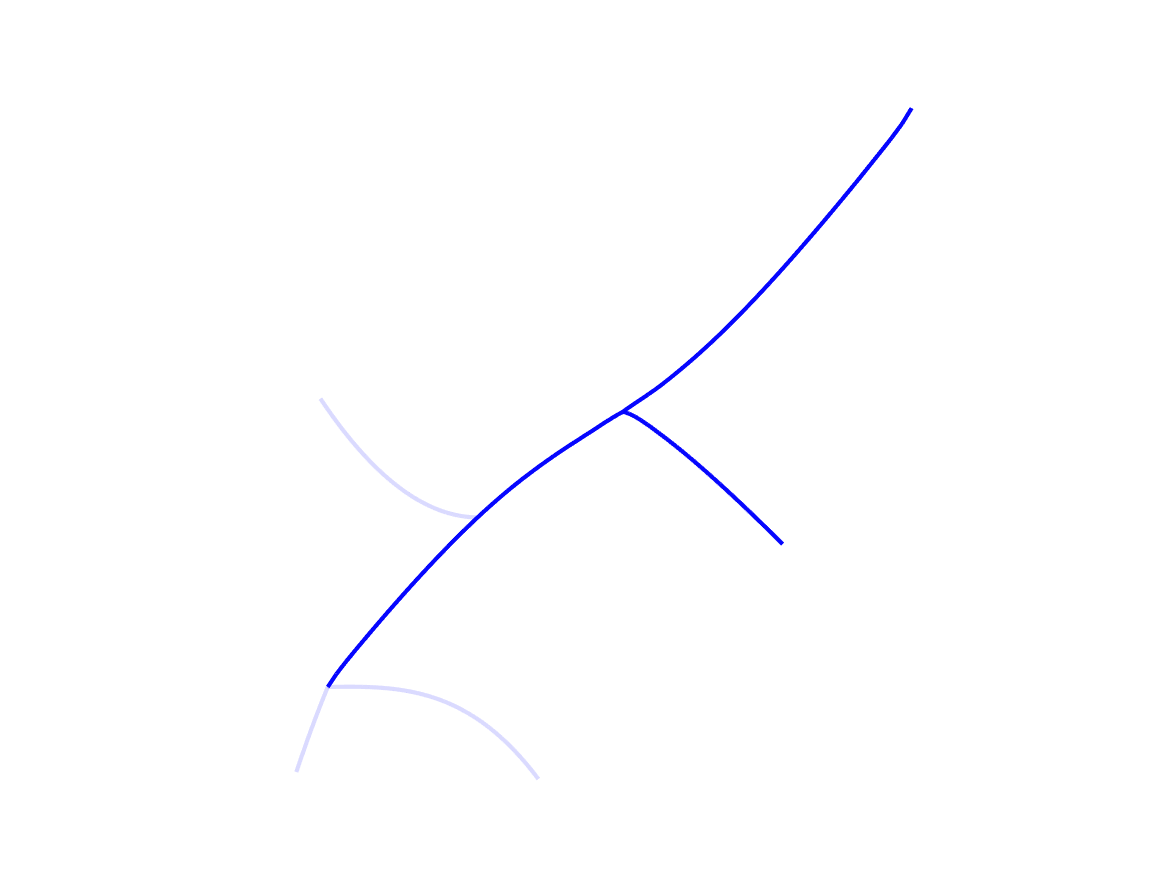}&
\includegraphics[trim = 45mm 15mm 45mm 10mm ,clip,width=2.5cm]{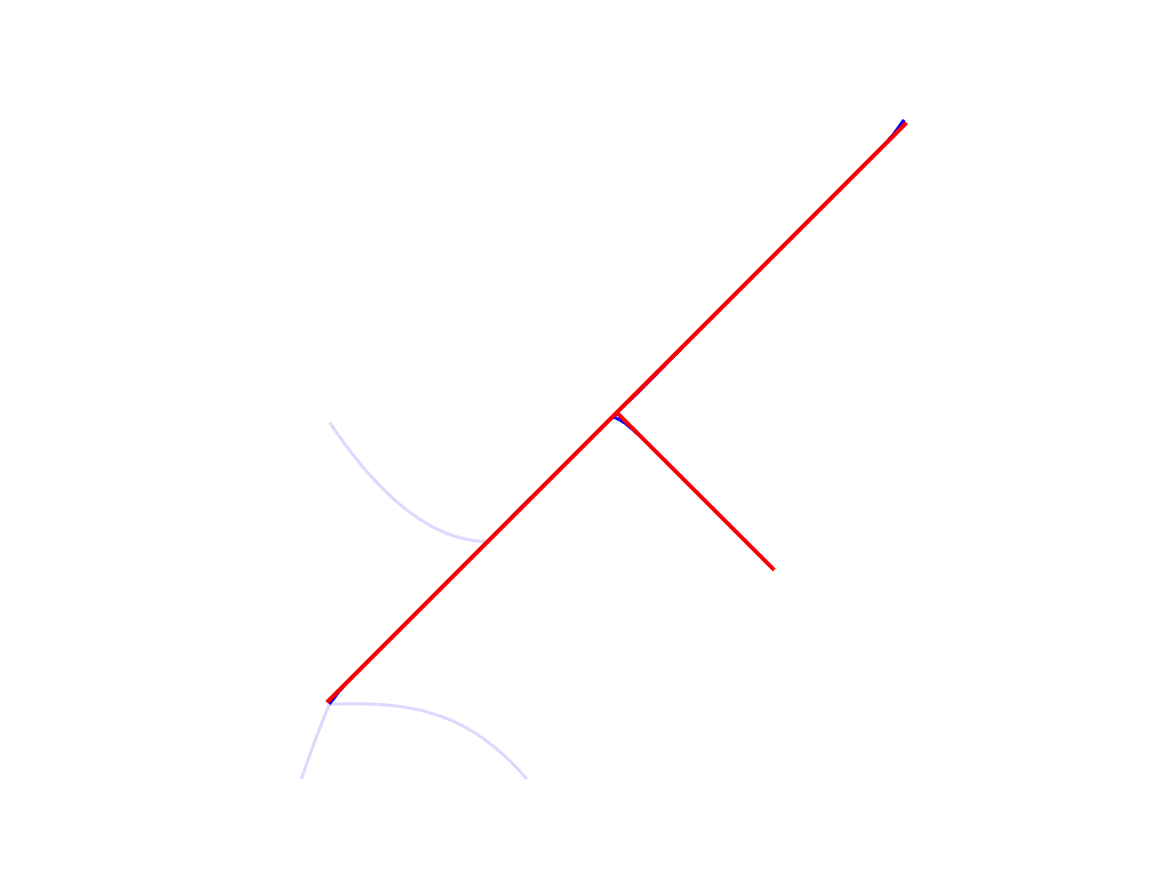}\\
$t=0$ & $t=0.25$ & $t=0.5$ & $t=0.75$ & $t=1$ 
\end{tabular}
\end{center}
\caption{Geodesic between source (blue at $t=0$) and target (red at $t=1$) shape graphs having different topologies. (Top row) We use the relaxed shape graph matching framework described in \eqref{eq:relaxed_shape_graph_match}, which only allows for a geometric deformation of the source. The estimated geodesic distance is $\overline{\operatorname{dist}}(c_0,c_1)=1.44$, around 1.7 times higher than in Fig. \ref{fig:match_2branches_fixed_weights}. (Bottom row) Result obtained from the weighted shape graph matching framework described in \eqref{eq:relaxed_weighted_shape_graph_match} that jointly estimates a deformation and weight changes on the source. Components of the source which get ``erased" are colored in progressively transparent shades of blue. The estimated geodesic distance here is $\overline{\operatorname{dist}}(c_0,c_1)=0.77$, now fairly comparable to Fig. \ref{fig:match_2branches_fixed_weights}.} \label{fig:match_branches_compare_fixed_vs_source_weights}
\end{figure}

In what follows, we propose to indirectly (and only partially) address this difficult challenge by augmenting the previous notion of shape graphs with a weight function defined on the shape, leveraging the flexibility of the varifold representation for that purpose. This leads to a new matching formulation where, in combination to the geometric matching process, one can vary the mass of different components of the source or target shape graphs. In particular, this allows us to remove certain parts of a shape graph when no corresponding components can be found in the other shape, as shown in Fig. \ref{fig:match_branches_compare_fixed_vs_source_weights}.   

\subsection{A new variational problem} 
Let $\Adj$ be a fixed adjacency matrix. We first define the space of parametrized weighted shape graphs of regularity $r>3/2$ having adjacency matrix $\Adj$ as the space of couples $(c,\rho)$ where $c\in \GraphrA$ and $\rho$ is a non-negative function of bounded variation on the shape graph. By this, we mean specifically that $\rho=\prod_{k=1}^{K} \rho^k$ is a Cartesian product of functions $\rho^k \in BV([0,1],\R_+)$, where $BV([0,1],\R_+)$ is the space of non-negative functions of bounded variation on the interval $[0,1]$, c.f. \cite{ambrosio2000functions} for a complete presentation of those spaces and their properties. For each $k=1,\ldots,K$ and $\theta \in [0,1]$, we can interpret $\rho^k(\theta)$ as the weight (or mass) assigned to the point $c^k(\theta)$ of the curve $c^k$. We will denote by $\WGraphrA$ the space of all such weighted shape graphs, and by $\WGraphr = \bigcup_{\mathcal{A}} \WGraphrA$  the space of all weighted shape graphs of regularity $r$ for all adjacency matrices $\Adj$. 
We can now formulate a generalized matching problem between two given weighted shape graphs as follows:
\begin{framed}
Given $(c_0,\rho_0),(c_1,\rho_1) \in \WGraphn$ we consider the minimization problem 
\begin{equation}
    \label{eq:relaxed_weighted_shape_graph_match}
    \inf \left\{\int_0^1 \bar{G}_{c(t)}^{n}(\partial_t{c}(t),\partial_t{c}(t)) dt  + \alpha F_{\rho_0}(c(1),\deltarho) + \lambda \|\mu_{c(1),\rho_0+\deltarho} - \mu_{c_1,\rho_1}\|_{\V}^2 \right\},
\end{equation}
where the infimum is taken over all paths $c(\cdot) \in H^1([0,1],\GraphnA)$ satisfying the initial constraint $c(0)=c_0$, and also over all weight change functions $\deltarho \in (BV([0,1],\R))^K$. In this framework, we refer to $(c_0, \rho_0)$ as the source shape graph, $(c(1), \rho_0 + \deltarho)$ as the transformed source, and $(c_1, \rho_1)$ as the target.

Moreover, $\alpha,\lambda>0$ are balancing parameters between the different terms,  $F_{\rho_0}(c(1), \deltarho)$ is a regularizing term for the weight change function defined on the transformed source, see Section~\ref{ssec:weigh_regularizer}, and the last term denotes the extension of the varifold norm for weighted shape graphs, see Section~\ref{ssec:varifold_weightedshapegraphs}.
\end{framed}

We reemphasize that the minimum value in \eqref{eq:relaxed_weighted_shape_graph_match} is independent of the ordering of the components of both shape graphs $c_0$ and $c_1$ as well as the parametrizations of each component despite the fact that we are not explicitly optimizing over those groups. This follows from the invariances of all terms in the matching functional and the independence of the varifold data attachment term to reparametrizations and permutations of the components of $c_0$ and $c_1$.

\begin{remark}
Note that the model formulated in \eqref{eq:relaxed_weighted_shape_graph_match} is asymmetric as it accounts for weight changes on the transformed source only. However, it can be adapted almost straightforwardly to weight changes on the target by instead optimizing over a function $\deltarho$ that transforms $(c_1,\rho_1)$ as $(c_1,\rho_1+\deltarho)$. The existence result of Theorem \ref{thm:existence_weighted_graph_match} below also holds in this case through a similar reasoning. Even more generally, one could consider the problem in which both the source and target weights are estimated by jointly optimizing over two weight change functions $\deltarho_0$ and $\deltarho_1$. The existence of solutions in this case is again easy to prove, although the design of the constraint function $\tilde{F}$, which will be introduced in \ref{ssec:weigh_regularizer}, becomes critical in order to avoid the trivial solution in which both $\rho_0 +\deltarho_0$ and $\rho_1 +\deltarho_1$ are equal to $0$ everywhere on the shape graphs.
\end{remark}

\subsubsection{The weight regularization term $F$}\label{ssec:weigh_regularizer}
We now introduce a specific choice for the regularization term $F$ that allows us to obtain an existence result for this new minimization problem: For a general weighted shape graph $(c,\rho) \in \WGraphnA$, we define the total variation ($TV$) norm of $\rho$ as the sum of the $TV$ norms on $[0,1]$ of each individual weight function:
\begin{equation}
 \label{eq:def_TV_shape_graph}
 \|\rho\|_{TV} = \sum_{k=1}^{K} \|\rho^k\|_{TV,[0,1]}. 
\end{equation}
In the context of \eqref{eq:relaxed_weighted_shape_graph_match}, we can now define the weight regularizer for the weight change function on the transformed source as follows:
\begin{equation}\label{eq:weight_reg_TV}
F_{\rho_0}(c(1), \deltarho)=\|\deltarho\|_{TV}+ \tilde \beta \tilde F_{\rho_0}(c(1), \deltarho),
\end{equation}
where $\tilde F$ is a  generic additional term that can be used to impose further constraints on the weight change function, with $\tilde \beta > 0$ being a balancing parameter. For instance, one could design these constraints to ensure the non-negativity of $\rho_0+\deltarho$, or as we shall see in the experiments section, one could push the values of $\rho_0+\deltarho$ to stay close to either $0$ or $1$ by choosing $\tilde F$ to be a double-well potential function. 

\begin{remark}
We point out that the $TV$ norm defined above does not depend on parametrization. Indeed, for any $BV$ function $\rho$ on $[0,1]$ and reparametrization $\varphi \in \Diff([0,1])$, by the usual properties of $BV$ functions and of the $TV$ norm, one has that $\rho \circ \varphi$ is also in $BV$ and that $\|\rho \circ \varphi\|_{TV,[0,1]} = \|\rho\|_{TV,[0,1]}$. It is also worth noting that if $c: [0,1]\rightarrow \R^d$ is an embedding, then the function $\rho \circ c^{-1}$ is a $BV$ function on the manifold curve $c([0,1])$ as defined in e.g. \cite{ambrosio2000functions}, and that $\|\rho\|_{TV,[0,1]}$ coincides with the total variation norm of $\rho \circ c^{-1}$ along that curve. 
\end{remark}
The following classical compactness property in the space $BV$ (that follows from Theorem 3.23 in \cite{ambrosio2000functions}) is the main reason for choosing to define $F$ using the $TV$ norm of $\deltarho$:
\begin{lemma}
\label{lemma:compactness_BV}
 If $(\rho_m)$ is a sequence in $BV([0,1],\R)$ with $\sup_{m \in \mathbb{N}} \|\rho_m\|_{TV,[0,1]} <+\infty$, then there exists $\rho \in BV([0,1],\R)$ such that, up to the extraction of a subsequence, one has $\rho_{m} \rightarrow \rho$ in $L^1([0,1])$, and for almost all $\theta \in [0,1]$, $\rho_{m}(\theta) \rightarrow \rho(\theta)$.  
\end{lemma}

\subsubsection{The varifold norm on the space of weighted shape graphs}\label{ssec:varifold_weightedshapegraphs}
Next, we discuss the extension of the varifold norm to the space of weighted shape graphs. For $n\geq 2$, we  represent any weighted shape graph $(c,\rho) \in \WGraphn$ as the varifold $\mu_{c,\rho} \in \V$ given by $\mu_{c,\rho} = \sum_{k=1}^{K} \rho^k \cdot \mu_{c^k}$, where for each $k$, we define $\rho^k \cdot \mu_{c^k}$ for any given test function $\omega \in C_0(\R^d \times S^{d-1},\R)$ by:
\begin{equation*}
    (\rho^k \cdot \mu_{c^k} | \omega) = \int_{0}^1 \rho^k(\theta) \omega\left(c^k(\theta),\frac{\partial_{\theta} c^{k }(\theta)}{|\partial_{\theta}c^{k}(\theta)|}\right) |\partial_{\theta}c^{k}(\theta)| d\theta.
\end{equation*}
Thus we can still rely on $\|\cdot\|_{\V}$ to compare weighted shape graphs, even when they have different numbers of component curves or adjacency matrices. We have the following technical lemma, which will be of importance in the proof of our existence result:
\begin{lemma}
\label{lemma:var_norm_continuity2}
 Let $(c_m,\rho_m)$ be a sequence of weighted shape graphs in $\WGraphnA$ for some fixed adjacency matrix $\Adj \in \R^{2K \times 2K}$ such that for all $m \in \N$ and all $k=1,\ldots,K$, $c_m^k$ is a $C^1$ immersion and $(c_m^k)$ converges uniformly to a $C^1$ immersion $c^k$, with the same holding true for its derivatives. Assume also that for all $k$, $(\rho^k_m)$ converges in $L^1$ to some $\rho^k \in BV([0,1],\R_+)$. Then, for $c=\prod_{k=1}^{K} c^k$ and $\rho=\prod_{k=1}^{K} \rho^k$, we have $\mu_{c_m,\rho_m} \xrightarrow[m\rightarrow \infty]{\|\cdot\|_\V} \mu_{c,\rho}$. 
\end{lemma}
The proof of this lemma is postponed to the appendix.

\subsection{Existence of solutions}
We are now able to show the well-posedness of the weighted shape graph matching problem~\eqref{eq:relaxed_weighted_shape_graph_match}, thereby generalizing Theorem \ref{thm:existence_graph_match}:
\begin{theorem}
\label{thm:existence_weighted_graph_match}
 Let $(c_0,\rho_0)$ and $(c_1,\rho_1)$ in $\WGraphn$ and $\Adj$ the adjacency matrix of $c_0$. Assume that $n\geq 2$ and that $\bar{G}^n$ is a scale invariant metric on $\GraphnA$. We further assume that $\H$ is continuously embedded into $C_0^1(\R^{d}\times \Sp^{d-1})$, that the weight regularization term $F$ is given by~\eqref{eq:weight_reg_TV}, and that $\deltarho \mapsto \tilde F_{\rho_0}(\cdot, \deltarho)$ is continuous with respect to the convergence a.e. of $\deltarho$. 
 
 Then there exist $c(\cdot) \in  H^1([0,1],\GraphnA)$ and $\deltarho \in (BV([0,1],\R))^K$ achieving the infimum in \eqref{eq:relaxed_weighted_shape_graph_match}.
\end{theorem}
\begin{proof}
Let us denote by $E(c(\cdot),\deltarho)$ the energy of \eqref{eq:relaxed_weighted_shape_graph_match} and let $(\tilde{c}_m(\cdot),\deltarho_{m})$ be a minimizing sequence of $E$ with $\tilde{c}_m(\cdot) \in H^1([0,1],\GraphnA)$, with $\Adj \in \R^{2K\times 2K}$ being the adjacency matrix of $c_0$. With the same arguments as in the proof of Theorem \ref{thm:existence_graph_match}, we deduce that up to a subsequence, we have that $(\tilde{c}_m(\cdot))$ converges weakly to some $c(\cdot) \in H^1([0,1],\GraphnA)$, and that for all $k=1,\ldots,K$ and all $t\in [0,1]$, we have that $\tilde{c}^k_m(t,\cdot)$ converges to $c^k(t,\cdot)$ with respect to the $\|\cdot\|_{1,\infty}$ norm on $[0,1]$. Furthermore, $\|\deltarho_m\|_{TV}$ is uniformly bounded in $m$, so by Lemma \ref{lemma:compactness_BV}, up to the extraction of another subsequence, we can also assume that there exists $\deltarho \in (BV([0,1],\R))^K$ such that for all $k$, we have $\deltarho^k_{m} \rightarrow \deltarho^k$ in $L^1([0,1])$ and $\deltarho^k_m(\theta) \rightarrow \deltarho^k(\theta)$ for almost all $\theta \in [0,1]$. 

This last statement implies, thanks to the assumption on $F$, that $F_{\rho_0}(c(1),\deltarho_m) \rightarrow F_{\rho_0}(c(1), \deltarho)$ as $m\rightarrow +\infty$. In addition, by Lemma \ref{lemma:var_norm_continuity2}, we deduce that $\|\mu_{\tilde{c}_m(1),\rho_0+\deltarho_m} - \mu_{c_1,\rho_1}\|_{\V}^2 \rightarrow \|\mu_{c(1),\rho_0+\deltarho} - \mu_{c_1,\rho_1}\|_{\V}^2$. Also, given the weak lower semi-continuity of the Riemannian distance already noted above and the fact that $\|\cdot\|_{TV,[0,1]}$ is lower semicontinuous with respect to the $L^1$--convergence on $[0,1]$, we see that:
\begin{equation*}
 E(c(\cdot),\deltarho) \leq \liminf\limits_{m \rightarrow +\infty} E(\tilde{c}_m(\cdot),\deltarho_m),
\end{equation*}
and consequently, we obtain that $(c(\cdot),\deltarho)$ is a minimizer of \eqref{eq:relaxed_weighted_shape_graph_match}.
\end{proof}

%---------------------- OPTIMIZATION APPROACH ------------------------------%
\section{Optimization approach}
\label{sec:optim_algo}
In this section we propose a numerical optimization approach to solve the generalized weighted shape graph matching problem introduced in \eqref{eq:relaxed_weighted_shape_graph_match}. We will focus exclusively on shape graph matching using the class of second-order elastic Sobolev metrics. As a result, we denote the space of parametrized shape graphs of regularity $r=2$ by $\Graphtwo$, the second-order Sobolev metrics of order $n = r = 2$ on this space as $\bar{G}^2$, and the corresponding space of weighted shape graphs as $\WGraph^2$. Our approach to solve \eqref{eq:relaxed_weighted_shape_graph_match} consists of discretizing the relaxed matching energy, before using appropriate optimization algorithms to minimize the discretized energy. We provide an open source implementation of our approach on \href{https://github.com/charoncode/ShapeGraph_H2match}{github}. 

\subsection{Discretizing the energy}
\label{ssec:discretizing_energy}
We first recall that the energy for the matching problem is given by:
\begin{equation*}
    E(c(\cdot), \deltarho) := \int_0^1 \bar{G}_{c(t)}^{2}(\partial_t{c}(t),\partial_t{c}(t)) dt  + \alpha F_{\rho_0}(c(1), \deltarho) + \lambda \|\mu_{c(1),\rho_0+\deltarho} - \mu_{c_1,\rho_1}\|_{\V}^2.
\end{equation*}
Thus, we see that the energy is a weighted sum of three terms, namely the Riemannian energy of the path of shape graphs satisfying the initial condition $c(0) = c_0$, the weight regularizer for the weight changes defined on the transformed source, and the varifold distance between the transformed source and target. In what follows, we describe how to discretize and compute each of these three terms separately.

\subsubsection{Riemannian energy of the path}
\label{ssec:riemannian_energy_path_computation}
To evaluate the Riemannian energy in the matching functional, we need to discretize the paths of shape graphs $c(\cdot) \in H^1([0,1], \GraphtwoA)$ satisfying the initial condition $c(0) = c_0$, and then find an expression for the energy in terms of these discretized paths.

We first note that all shape graphs in the path have the same topology as the source shape graph $c_0 = \prod_{k=1}^K c_0^k$, i.e., they consist of $K$ component curves whose connectivity is described by the adjacency matrix $\mathcal{A}$. Therefore, we construct a discretized path of immersions for each component curve, which we denote by $c^k(\cdot) \in H^1([0,1], \Imm^2([0,1], \R^d))$ for $k = 1,...,K$. We use the same procedure as defined in~\cite[Section 3]{bauer2017numerical} to construct these discretized paths of immersions, i.e., we discretize each $c^k(\cdot)$ using tensor product B-splines on knot sequences of orders $n_t$ in time and $n_{\theta}$ in space. Typically, we choose $n_t = 1$ and $n_{\theta} = 2$ in our numerical experiments. This produces $N_t \times N_{\theta}$ basis splines, with $N_t$ and $N_{\theta}$ being the number of control points in time and space respectively. Typical values used in experiments are $N_t = 10$ and $N_{\theta} = \mathcal{O}(10^2)$. We can thus write the discretized path of immersions for each component curve as follows:
\begin{equation*}
    c^k(t, \theta) = \sum_{i=1}^{N_t} \sum_{j=1}^{N_{\theta}} c_{i,j}^k B_i(t) C_j(\theta), \qquad (t, \theta) \in [0,1] \times [0,1]
\end{equation*}
where $c_{i,j}^k \in \R^d$ are the control points for $c^k(\cdot)$, and $B_i(t)$ and $C_j(\theta)$ are B-splines defined by an equidistant simple knot sequence on $[0, 1]$ with full multiplicity at the boundary knots. The full multiplicity of boundary knots in $t$ implies that the discretized initial curve $c^k(0)$ for each $k=1,...,K$ depends only on the control points $c_{1,j}^k$. This allows us to easily  enforce the initial constraint by keeping these control points fixed. We point out that if the source shape graph $c_0 = \prod_{k=1}^K c_0^k$ with adjacency matrix $\mathcal{A}$ is not already split into its individual component curves, this can be done using e.g., Tarjan's algorithm \cite{tarjan1972depth}.

Note that we can analytically differentiate our spline paths in both time and space. Thus we get an explicit expression for the Riemannian energy of each path of immersions by the formulas in Section~\ref{ssec:shape_graphs}. 
Since we are using a B-spline discretization for the paths, we evaluate the integrals in the resulting expression using Gaussian quadrature with quadrature sites placed between knots where the curves are smooth, c.f.~\cite{bauer2019relaxed}. This gives us a numerical approximation of the energy for the paths of immersions corresponding to each component curve of the shape graph, which we then sum up to obtain the Riemannian energy of the full path of shape graphs, as the linearity of integrals combined with the formula for the Riemannian metric on $\GraphtwoA$ in \eqref{eq:def_bar_G} allows us to write:
\begin{align*}
    \int_0^1 \bar{G}_{c(t)}^{2}(\partial_t{c}(t),\partial_t{c}(t)) dt = \sum_{k=1}^K \left( \int_0^1 G_{c^k(t)}^2(\partial_t{c}^k(t), \partial_t{c}^k(t)) dt \right).
\end{align*}
This results in a fast and robust way to evaluate the Riemannian energy of the path of shape graphs satisfying $c(0) = c_0$. The same is true for the evaluation of the derivatives of this energy. The formulas for the derivatives for each component curve can be found in~\cite[Appendix A]{bauer2019relaxed}.

\subsubsection{Weight regularizer}
\label{ssec:weight_regularizer_computation}
We next discuss the discretization and computation of the weight regularizer for the weight change function $\delta\hspace{-.1em} \rho \in (BV([0,1], \R))^K$ defined on the transformed source. 

First, we express the discretized transformed source in a convenient format for discretizing the weight function defined on it. This is done by evaluating each spline component $c^k(1)$ of the transformed source at the uniform samples $\theta_i= \frac{i}{N_k}$, leading to an ordered list of vertices $\{v_0^k, v_1^k,\ldots,v_{N_k}^k\}$, together with corresponding oriented edge vectors $e_i^k = v_{i+1}^{k} - v_{i}^{k}$ and edge centers $x_i^k = (v_{i}^k+v_{i+1}^k)/2$. This allows us to discretize the transformed weight function $\rho^k := \rho^k_0 + \deltarho^k \in BV([0,1],\R)$ as weights defined over the edges of the discretized component curve $c^k(1)$, which we represent by a list $\rho^k_i := \rho^k_{0,i} + \deltarho^k_i$, for $i=0,\ldots,N_k-1$ for each $k=1,\ldots,K$. We thus have a simpler representation for the discretized transformed source as a weighted piecewise linear (i.e. polygonal) curve given by the full list of edge vectors $e_1,e_2,\ldots,e_N$, edge centers $x_1,x_2,\ldots,x_N$, and edge weights $\rho_1,\rho_2,\ldots,\rho_N$, where $N=N_1 N_2\ldots N_K$.

Next, we see that the weight regularizer $F_{\rho_0}(c(1), \deltarho)=\|\delta\rho\|_{TV} + \tilde \beta \tilde F_{\rho_0}(c(1),\deltarho)$ consists of the $TV$ norm of the weight change function, plus a flexible choice for the additional term $\tilde{F}_{\rho_0}$. To compute the $TV$ norm, we use the aforementioned discretization of the transformed source as a weighted polygonal curve, and write:
\begin{align*}
    \| \deltarho \|_{TV} = \sum_{k=1}^K \| \deltarho^k \|_{TV,[0,1]} \approx \sum_{k=1}^K \sum_{i=1}^{N_k-1} | \deltarho_{i}^k - \deltarho_{i-1}^k | = \sum_{k=1}^K \| D^k \deltarho^k \|_1,
\end{align*} 
where for each $k=1,\ldots,K$, we denote the discretized weight change function on each component with a slight abuse of notation by $ \deltarho^k = (\deltarho_0^k,\ldots, \deltarho_{N_k-1}^k) \in \R^{N_k}$, with $D^k \in \R^{(N_k-1) \times N_k}$ being the appropriate difference operator such that $\| D^k \deltarho^k \|_1 = \sum_{i=1}^{N_k-1} | \deltarho_{i}^k -  \deltarho_{i-1}^k |$, namely:
\begin{align*}
    D^k =
    \begin{bmatrix}
    -1 & 1 & 0 & \hdots & 0 & 0 \\
    0 & -1 & 1 & \hdots & 0 & 0 \\
     \vdots & \vdots & \vdots  & \ddots & \vdots & \vdots \\
    0 & 0 & 0 & \hdots & -1 & 1
    \end{bmatrix}_{(N_k-1) \times N_k}.
\end{align*}
In what follows, it will be more convenient to express the numerical approximation for the $TV$ norm of the weight change function as follows:
\begin{equation}\label{eq:tv_norm_formula_l1}
    \| \deltarho \|_{TV} \approx \| D \deltarho \|_1,
\end{equation}
where $\deltarho = (\deltarho^1,\ldots,\deltarho^K) \in \R^N$ represents the discretized weight change function on the full transformed source shape graph, and $D = \operatorname{diag}(D^1,\ldots,D^K) \in \R^{(N-K) \times N}$ is a non-square block diagonal matrix of the difference operators. We note that since the $TV$ norm is non-differentiable with respect to $\deltarho$, we will require carefully selected optimization techniques to solve~\eqref{eq:relaxed_weighted_shape_graph_match}, as will be outlined in Section~\ref{ssec:minimization_shape_graph_matching}.

Meanwhile, as mentioned in Section \ref{ssec:weigh_regularizer}, the flexible additional term $\tilde{F}_{\rho_0}$ can be chosen to impose further constraints on the weight change function, such as nonnegativity constraints, or constraints that binarize the value of $\rho := \rho_0 + \deltarho$. As long as $\deltarho \mapsto \tilde{F}_{\rho_0}(\cdot, \deltarho)$ is continuous, the existence result from Theorem \ref{thm:existence_weighted_graph_match} holds. In our numerical experiments, we will focus on the following particular choice for $\tilde{F}_{\rho_0}$, which is inspired from a similar penalty used in the context of clustering in \cite{wang2020efficient}:
\begin{equation*}
    \tilde{F}_{\rho_0}(c, \deltarho):= \sum_{k=1}^K \int_0^1 8(\rho^k(\theta)(\rho^k(\theta) - 1))^2 |\partial_{\theta} c^{k}(\theta)| d\theta.
\end{equation*}
We call this mapping $\tilde{F}_{\rho_0}$ the \textit{$\{0,1\}$-penalty}, as it is minimized precisely when a weight function $\rho \in (BV([0,1],\R)^K$ defined on a shape graph $c \in \Graphn$ takes on pointwise values of either 0 or 1. We focus on the $\{0,1\}$-penalty as it is particularly well-suited for comparing weighted shape graphs with different topologies or for shape registration problems with partial matching constraints, as we will illustrate in Section \ref{sec:numerical_results}.

Given the discretized weighted transformed source $(c(1), \rho)$, which is represented by its list of oriented edge vectors $e_1, e_2, \ldots, e_N$, edge centers $x_1, x_2, \ldots, x_N$ and edge weights $\rho_1, \rho_2, \ldots, \rho_N$ with $\rho_i = \rho_{0,i} + \delta\hspace{-.1em} \rho_i$ for $i = 1,\ldots,N$, one can approximate the $\{0,1\}$-penalty numerically as follows:
\begin{equation}\label{penalty_01}
    \tilde{F}_{\rho_0}(c(1), \deltarho) \approx \sum_{i=1}^N 8 (\rho_i (\rho_i - 1))^2 |e_i|.
\end{equation}
From the expression above, we see that the $\{0,1\}$-penalty is differentiable with respect to the edges $e_i^k$, and by the chain rule, with respect to the vertices $v_i^k$ of the discretized transformed source for each $k=1,\ldots,K$. By applying a second chain rule, one obtains the differentiability and the explicit expression of the derivatives of the $\{0,1\}$-penalty with respect to the final spline control points $c_{N_t,j}^k$, similar to what is done in \cite{bauer2019relaxed} and in the library \cite{h2metricsGIT}. Moreover, since the $\{0,1\}$-penalty is a quartic function of the edge weights $\rho_i := \rho_{0,i} + \deltarho_i$, its derivative with respect to each $\deltarho_i$ is also easy to compute. In our experiments, we control the high growth of the penalty's derivatives with respect to the weight changes $\deltarho_i$ outside the interval $[0,1]$ by using a clipped version of the penalty, which has piecewise linear segments outside the interval $(-\epsilon, 1+\epsilon)$ for some fixed value of $\epsilon > 0$, see our \href{https://github.com/charoncode/ShapeGraph_H2match}{source code} for details.

\subsubsection{Varifold norm}
\label{ssec:varifold_norm_computation}
Lastly, we address the discretization and numerical computation of the varifold norm $\| \cdot \|_{\V}$ that defines the discrepancy term between the transformed source and target. As explained above, we can represent the discretized transformed source $(c(1),\rho)$, where $\rho = \rho_0 + \deltarho$, by the lists of edge vectors $e_1,e_2,\ldots,e_N$, edge centers $x_1,x_2,\ldots,x_N$ and edge weights $\rho_1,\rho_2,\ldots,\rho_N$. Similarly, the discretized target $(c_1,\rho_1)$ will be represented by its list of edge vectors $\tilde{e}_1,\tilde{e}_2,\ldots,e_{\tilde{N}}$, edge centers $\tilde{x}_1,\tilde{x}_2,\ldots,\tilde{x}_{\tilde{N}}$ and edge weights $\tilde{\rho}_1,\tilde{\rho}_2,\ldots,\tilde{\rho}_{\tilde{N}}$. This allows us to write the following approximation for $\langle \mu_{c(1),\rho_0 + \deltarho} , \mu_{c_1,\rho_1} \rangle_{\V}$: 
\begin{equation}
\label{eq:inner_prod_var_curve_discrete}
    \langle \mu_{c(1),\rho_0 + \deltarho} , \mu_{c_1,\rho_1} \rangle_{\V} \approx \sum_{i=1}^{N} \sum_{j=1}^{\tilde{N}}  \mathcal{K}\left(x_i,\frac{e_i}{|e_i|},\tilde{x}_j,\frac{\tilde{e}_j}{|\tilde{e}_j|}\right) \rho_i \tilde{\rho}_j |e_i| |\tilde{e}_j|,
\end{equation}
and obtain a corresponding finite approximation of the squared varifold distance $\|\mu_{c(1),\rho_0+\delta\rho} - \mu_{c_1,\rho_1}\|_{\V}^2$. Note that \eqref{eq:inner_prod_var_curve_discrete} essentially amounts in approximating the varifold inner product between edge $e_i$ in the transformed source $c(1)$ and edge $e_j$ in the target $c_1$ by a single evaluation of the kernel $\mathcal{K}$ at the center of those respective edges. We point to the references \cite{kaltenmark2017general,charon2020fidelity,hsieh2021metrics} for more details on this discretization scheme and the resulting approximation bounds that one is able to recover.  

Regarding the choice of kernel $\mathcal{K}$ on $\R^d \times S^{d-1}$, a natural class that has been commonly used in the aforementioned papers are the positive definite separable radial kernels which take the form $\mathcal{K}(x,u,y,v) = \Psi(|x-y|^2) \Phi(u \cdot v)$, where $\Psi$ and $\Phi$ are two $C^1$ functions on $\R_{+}$ and $[-1,1]$ respectively. These kernels do satisfy the assumptions required for the existence of solutions in Theorem \ref{thm:existence_weighted_graph_match}, while also leading to translation and rotation invariant distances between shape graphs. In our experiments, we choose $\Psi(t) = e^{-\frac{t}{\sigma^2}}$, i.e., a Gaussian kernel of width $\sigma>0$.
The scale parameter $\sigma$ is typically selected depending on the spatial size of the shape graphs to be matched. As for the kernel on $S^{d-1}$, we typically choose either $\Phi(u\cdot v) = (u \cdot v)^2$, which leads to distances independent of the orientation of the shape graph components, or alternatively $\Phi(u\cdot v) = e^{-\frac{2}{\tau^2}(1-u\cdot v)}$, in which case the varifold metric does take orientation into account. A more thorough discussion on the properties of such kernel metrics can be found in \cite{kaltenmark2017general} and \cite{charon2020fidelity}.   

With this regularity on the kernel $\mathcal{K}$, the discrete varifold inner product in \eqref{eq:inner_prod_var_curve_discrete}, and by extension, the data attachment term $\|\mu_{c(1),\rho_0+\delta\rho} - \mu_{c_1,\rho_1}\|_{\V}^2$ are clearly differentiable with respect to the $x_i$'s and $e_i$'s, and by a simple chain rule, with respect to the vertices $v_i^k$ of the discretized transformed source. Then by applying a second chain rule, one obtains the differentiability and the explicit expression of the derivatives of the varifold data attachment term with respect to the final spline control points $c_{N_t,j}^k$, similarly to what is done in \cite{bauer2019relaxed} and in the library \cite{h2metricsGIT}. Furthermore, since the varifold norm is a quadratic function of the edge weights $\rho_i := \rho_{0,i} + \deltarho_i$, its derivative with respect to each $\deltarho_i$ is straightforward to compute.

\subsection{Minimizing the energy}
\label{ssec:minimization_shape_graph_matching}
With the discretization introduced in Section \ref{ssec:discretizing_energy}, the generalized weighted shape graph matching problem becomes a standard finite dimensional optimization problem over the spline control points of the discretized path of shape graphs $c(\cdot) \in \R^{N_t \times N_{\theta}}$, and the changes in edge weights $\deltarho \in \R^N$ defined on the transformed source. The main technicality arises from the presence of the non-smooth $TV$ norm of $\deltarho$ as a regularizer, which prevents us from directly applying standard gradient-descent based optimization techniques. To tackle this issue, we propose an approach which broadly speaking, involves iteratively minimizing a smoothed version of the matching energy, which we obtain by carefully rewriting the $TV$ norm and introducing appropriate auxiliary variables. This is adapted from the smoothed fast iterative shrinkage-thresholding (SFISTA) algorithm proposed in \cite{tan2014smoothing} for convex non-smooth minimization. For the problem we consider here, the advantage of this approach over alternative splitting methods for $TV$ norm minimization, such as the split Bregman algorithm, is that it will allow us to tackle the minimization over both control points and weight changes as a series of unconstrained smooth optimization problems that can be solved efficiently using a limited memory BFGS procedure. We detail our full approach below.

First, by using the expression for the $TV$ norm introduced in \eqref{eq:tv_norm_formula_l1}, namely $\|\deltarho\|_{TV} = \|D \deltarho\|_{1}$, and by introducing the auxiliary variable $\eta = D \deltarho \in \R^{N-K}$, we can rewrite \eqref{eq:relaxed_weighted_shape_graph_match} as a constrained minimization problem:
\begin{align*}
    \inf_{c(\cdot), \deltarho, \eta} \int_0^1 \bar{G}_{c(t)}^{2}(\partial_t{c}(t),\partial_t{c}(t)) dt  + \lambda \|\mu_{c(1),\rho_0+\deltarho} - \mu_{c_1,\rho_1}\|_{\V}^2 + \beta \tilde F_{\rho_0}(c(1), \deltarho) + \alpha \| \eta \|_{1} 
\end{align*}
such that  $\eta = D \deltarho$. By relaxing this constrained problem  using a Lagrange multiplier $\gamma \in \R$, we obtain the  unconstrained problem:
\begin{multline}
    \label{eq:unconstrained_relaxed_matching_energy}
    \inf_{c(\cdot), \deltarho, \eta} \int_0^1 \bar{G}_{c(t)}^{2}(\partial_t{c}(t),\partial_t{c}(t)) dt  + \lambda \|\mu_{c(1),\rho_0+\delta\rho} - \mu_{c_1,\rho_1}\|_{\V}^2 \\+ \beta \tilde F_{\rho_0}(c(1),\deltarho) + \alpha \| \eta \|_{1} + \frac{\gamma}{2} \|\eta - D \deltarho \|_2^2.
\end{multline}
We denote the energy above as $E_{\gamma}(c(\cdot), \deltarho, \eta)$, and note that as $\gamma \rightarrow +\infty$, the unconstrained minimization of $E_{\gamma}(c(\cdot), \deltarho, \eta)$ becomes equivalent to its constrained counterpart. The minimization for \eqref{eq:unconstrained_relaxed_matching_energy} now involves an additional variable $\eta$, but fortunately, the objective function is strictly convex in $\eta$. For fixed $c(\cdot)$ and $\deltarho$, the unique global minimizer $\eta^*$ for $\eqref{eq:unconstrained_relaxed_matching_energy}$ is obtained from the proximal operator of $\|\cdot\|_1$ as:
\begin{align*}
    \eta^*
    &=  \underset{\eta}{\argmin} \hspace{.5em} E_{\gamma}(c(\cdot), \deltarho, \eta) = \underset{\eta}{\argmin} \hspace{.5em} \| \eta \|_{1} + \frac{1}{2 \alpha / \gamma} \|\eta - D \deltarho \|_2^2
    = \operatorname{prox}_{\frac{\alpha}{\gamma} \| \cdot \|_1} (D \deltarho) .
\end{align*}
Specifically, $\eta^* = (\eta_1^*,\ldots,\eta_{N-K}^*)$ can be expressed element-wise as follows:
\begin{equation}
    \label{eqn:shrink_operator}
    \eta_i^* = 
    \begin{cases}
    (D\deltarho)_i - \frac{\alpha}{\gamma} &\mbox{if } (D\deltarho)_i > \frac{\alpha}{\gamma} , \\
    (D\deltarho)_i + \frac{\alpha}{\gamma} &\mbox{if } (D\deltarho)_i < -\frac{\alpha}{\gamma} , \\
    0 &\mbox{otherwise.}
    \end{cases}
\end{equation}
By plugging this explicit expression for $\eta^*$ back into \eqref{eq:unconstrained_relaxed_matching_energy}, we recover a minimization problem over $c(\cdot)$ and $\deltarho$ only, namely:
\begin{equation}
    \label{eq:unconstrained_relaxed_energy_match_explicit}
    \inf_{c(\cdot), \deltarho} \int_0^1 \bar{G}_{c(t)}^{2}(\partial_t{c}(t),\partial_t{c}(t)) dt  + \lambda \|\mu_{c(1),\rho_0+\deltarho} - \mu_{c_1,\rho_1}\|_{\V}^2 + \beta \tilde F_{\rho_0}(c(1),\deltarho) + \sum_{i=1}^{N-K} H_{\alpha, \gamma}((D \deltarho)_i),
\end{equation}
where
$H_{\alpha, \gamma}: \R \rightarrow \R$ is the so-called Huber function, which is defined as:
\begin{equation}
    \label{eq:huber_function}
    v \longmapsto H_{\alpha, \gamma}(v) := 
    \begin{cases}
    \frac{\gamma}{2} v^2 &\mbox{if } |v| \leq \alpha / \gamma, \\
    \alpha \Big[ |v| - \frac{\alpha}{2 \gamma} \Big] &\mbox{if } |v| > \alpha / \gamma.
    \end{cases}
\end{equation}
We notice that \eqref{eq:unconstrained_relaxed_energy_match_explicit} corresponds to the minimization of a continuously differentiable function since one can check the continuity of the Huber function and its derivative at the points $v = \pm \alpha / \gamma$.

As a result, for a given $\gamma > 0$, minimizing the energy in~\eqref{eq:unconstrained_relaxed_energy_match_explicit} can be done using standard gradient-descent based methods. In our experiments, we use an L-BFGS procedure, whose MATLAB implementation is available via the HANSO library\footnote{\url{https://cs.nyu.edu/overton/software/hanso/}}. Yet, the smoothed problem~\eqref{eq:unconstrained_relaxed_energy_match_explicit} only coincides with~\eqref{eq:relaxed_weighted_shape_graph_match} as $\gamma \rightarrow + \infty$. Therefore, to obtain a solution for our original problem, we solve the smoothed problem repeatedly with warm start initialization and values of $\gamma$ that are incrementally increased. The precise analysis of the validity and convergence of this sequential approach can be found in \cite{tan2014smoothing}, albeit for convex functions.  Moreover, we point out that the most computationally expensive operation at each iteration of the optimization process involves evaluating the varifold discrepancy term, which requires $\mathcal{O}(N \Tilde{N})$ kernel evaluations as described in~\eqref{eq:inner_prod_var_curve_discrete}.

%------------------------ NUMERICAL RESULTS ------------------------------%
\section{Numerical results}
\label{sec:numerical_results}
We now present results from several numerical experiments which demonstrate how the generalized weighted shape graph elastic matching framework can be useful in the context of partially observed and topologically varying data. Note that all of the following experiments were performed by numerically solving~\eqref{eq:relaxed_weighted_shape_graph_match} using a MATLAB implementation of the optimization approach outlined in Section~\ref{sec:optim_algo}. In our experiments, the optimization procedure is initialized by default with a constant path equal to the source shape graph. This is an added benefit of our variational formulation, as it gives us a more principled initialization compared to other gradient descent-based graph matching algorithms, where random initializations are usually employed. Our source code is publicly available on Github\footnote{\url{https://github.com/charoncode/ShapeGraph_H2match}}.

We point out that in order to solve problem~\eqref{eq:relaxed_weighted_shape_graph_match}, one needs to select a range of parameter values, such as the balancing parameters in the smoothed matching energy from~\eqref{eq:unconstrained_relaxed_energy_match_explicit}, namely $\lambda$ and $\beta$, the coefficients of the Riemannian metric for the energy of the path of shape graphs, parameters for the varifold data attachment term, and the Huber function parameters $\alpha$ and $\gamma$. While the latter is chosen according to the suggested scheme in \cite{tan2014smoothing}, the selection of the other parameters is highly data dependent as it typically depends on the size of the considered shape graphs, or the relative amount of expected geometric deformation versus weight change. Our current approach is to select parameter values by grid search combined with sequential parameter refinement, and we leave it as future work to devise a more principled or data-driven scheme for parameter selection. 
In all our experiments we scale the shape graphs to unit diameter and use the constant coefficient Sobolev metric from~\eqref{eq:def_Sobolev_metric} with coefficients $a_0 = 0.1, a_1 = 1, a_2 = 10^{-5}$ rather than the scale-invariant counterpart since this typically leads to faster numerics and, from our experience, very similar results when there are no major scale differences between the two shapes.

\begin{figure}[h!]
\hspace{-1cm}
\begin{center}
\begin{tabular}{ccccc}
\includegraphics[angle=90,trim = 60mm 15mm 60mm 10mm ,clip,width=2.3cm]{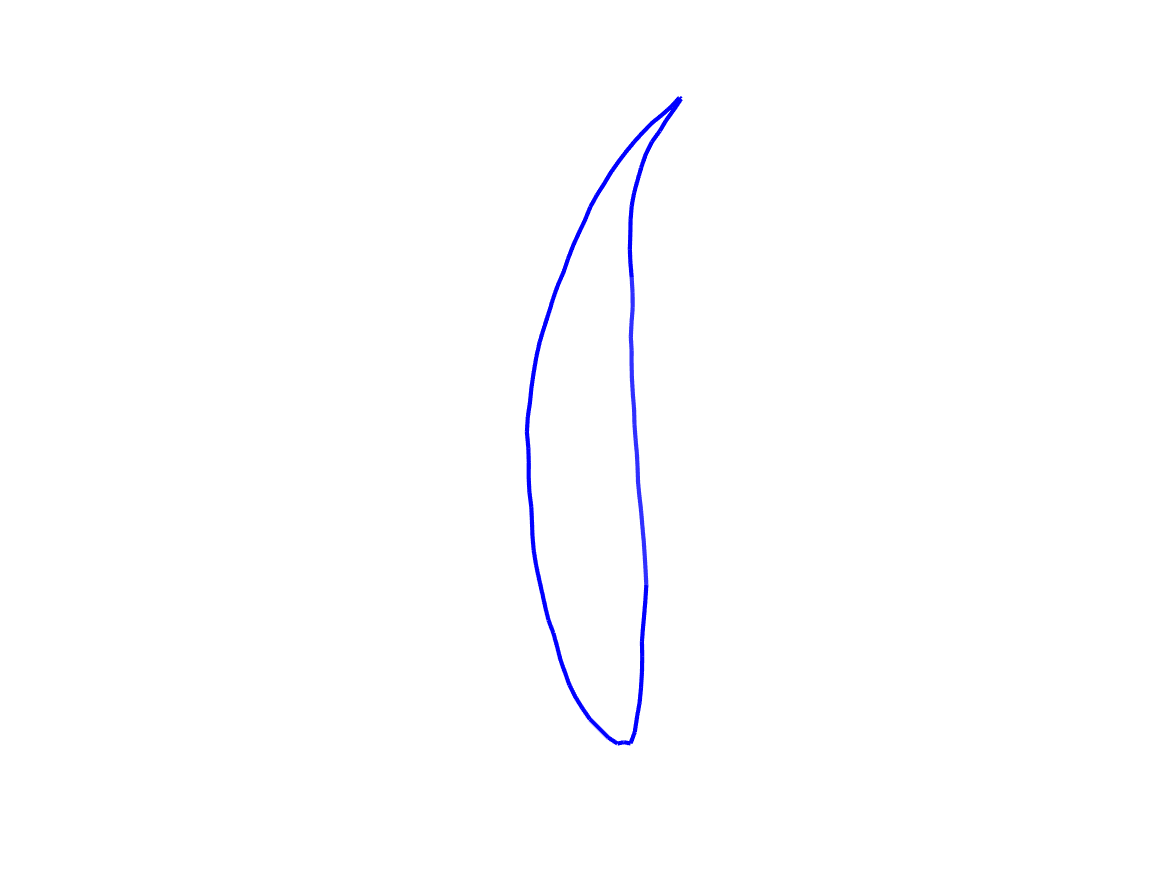}&
\includegraphics[angle=90,trim = 60mm 15mm 60mm 10mm ,clip,width=2.3cm]{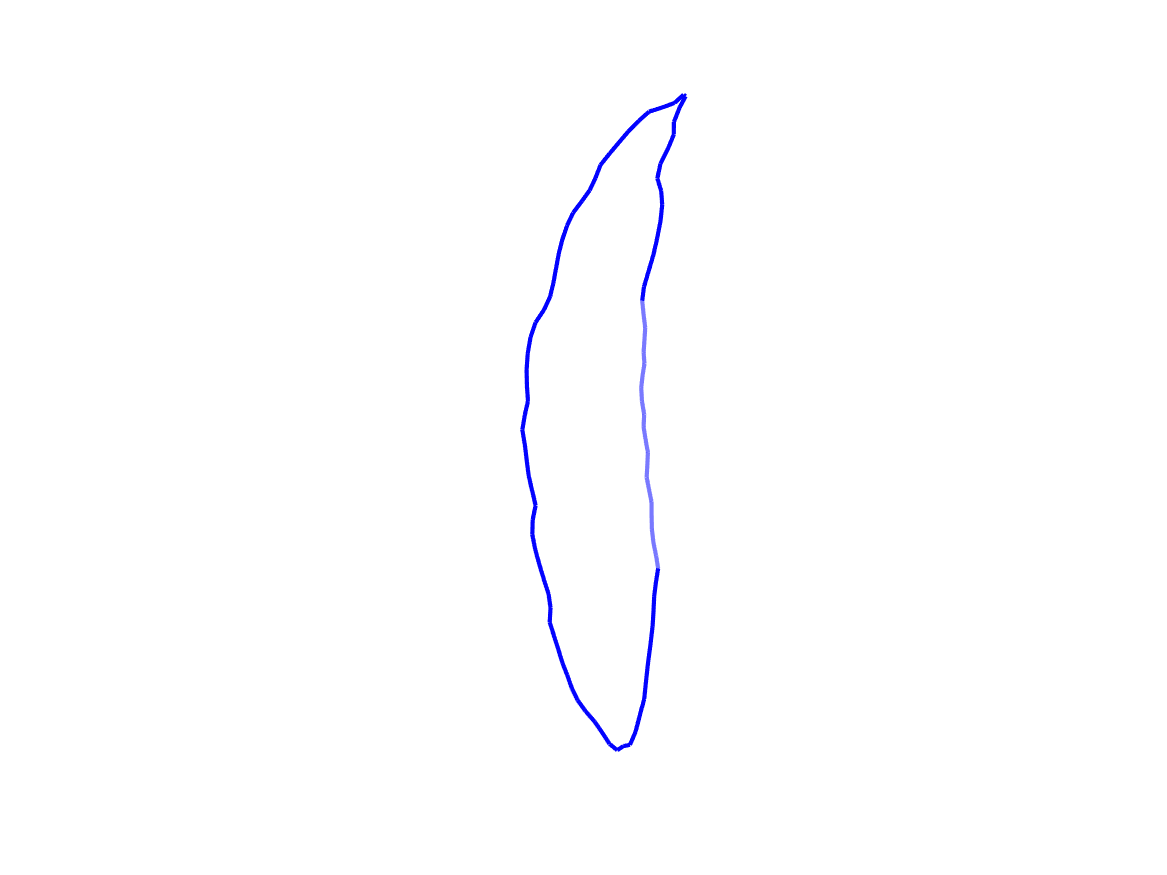}&
\includegraphics[angle=90,trim = 60mm 15mm 60mm 10mm ,clip,width=2.3cm]{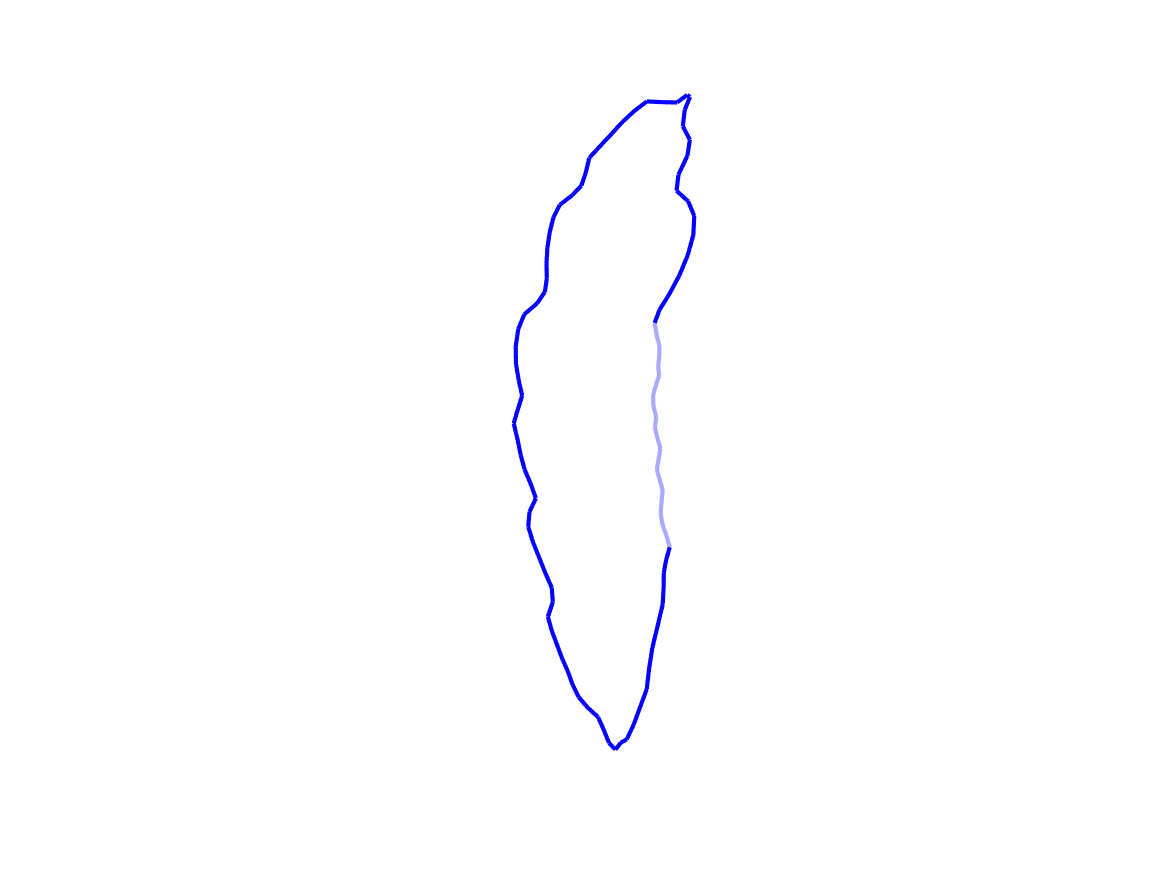}&
\includegraphics[angle=90,trim = 60mm 15mm 60mm 10mm ,clip,width=2.3cm]{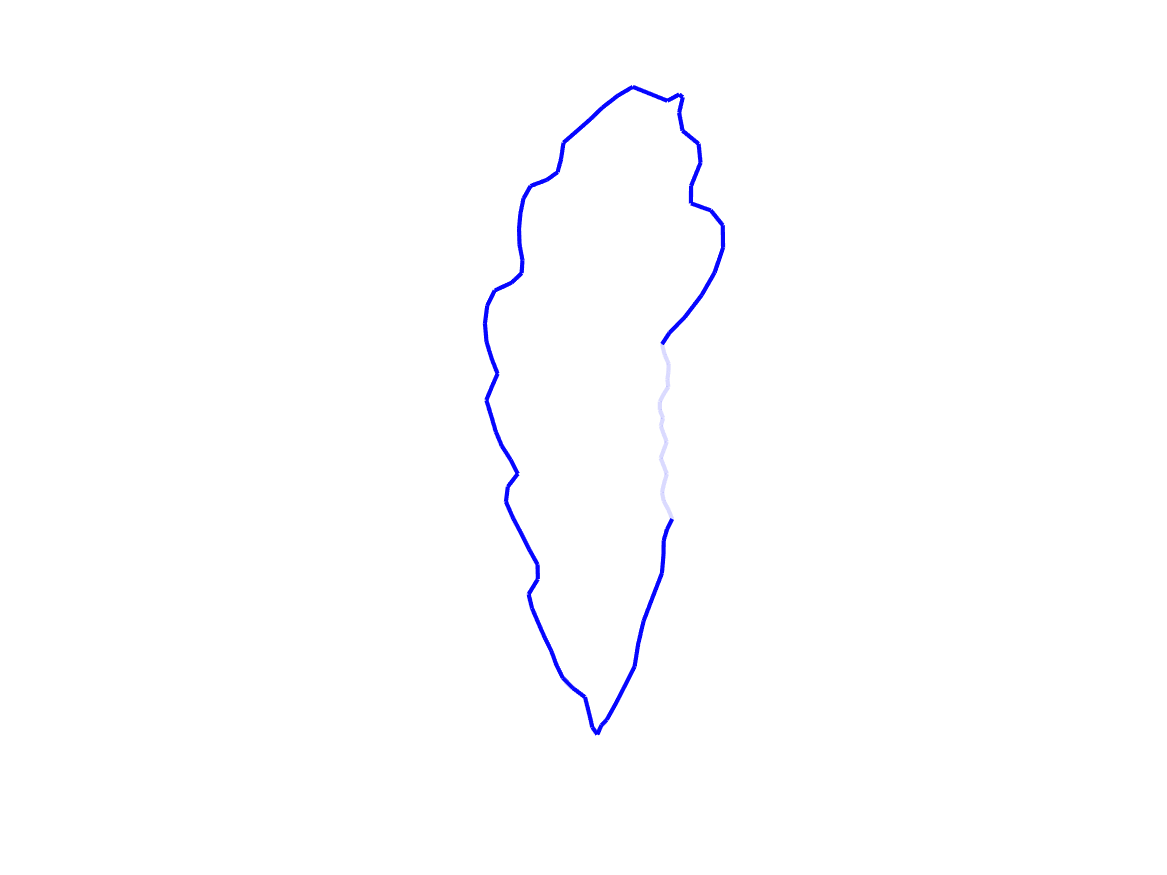}&
\includegraphics[angle=90,trim = 60mm 15mm 60mm 10mmm ,clip,width=2.3cm]{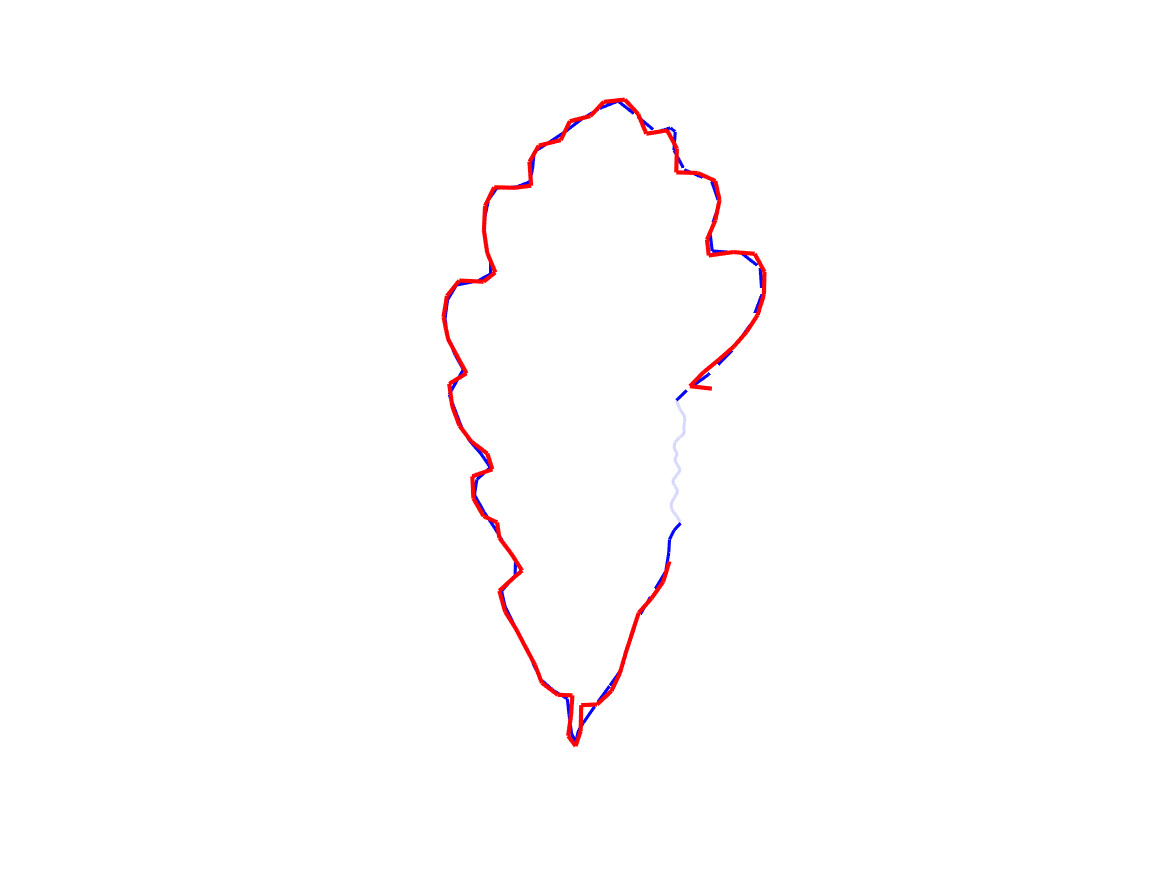}\\
\includegraphics[angle=90,trim = 50mm 15mm 50mm 10mm ,clip,width=2.3cm]{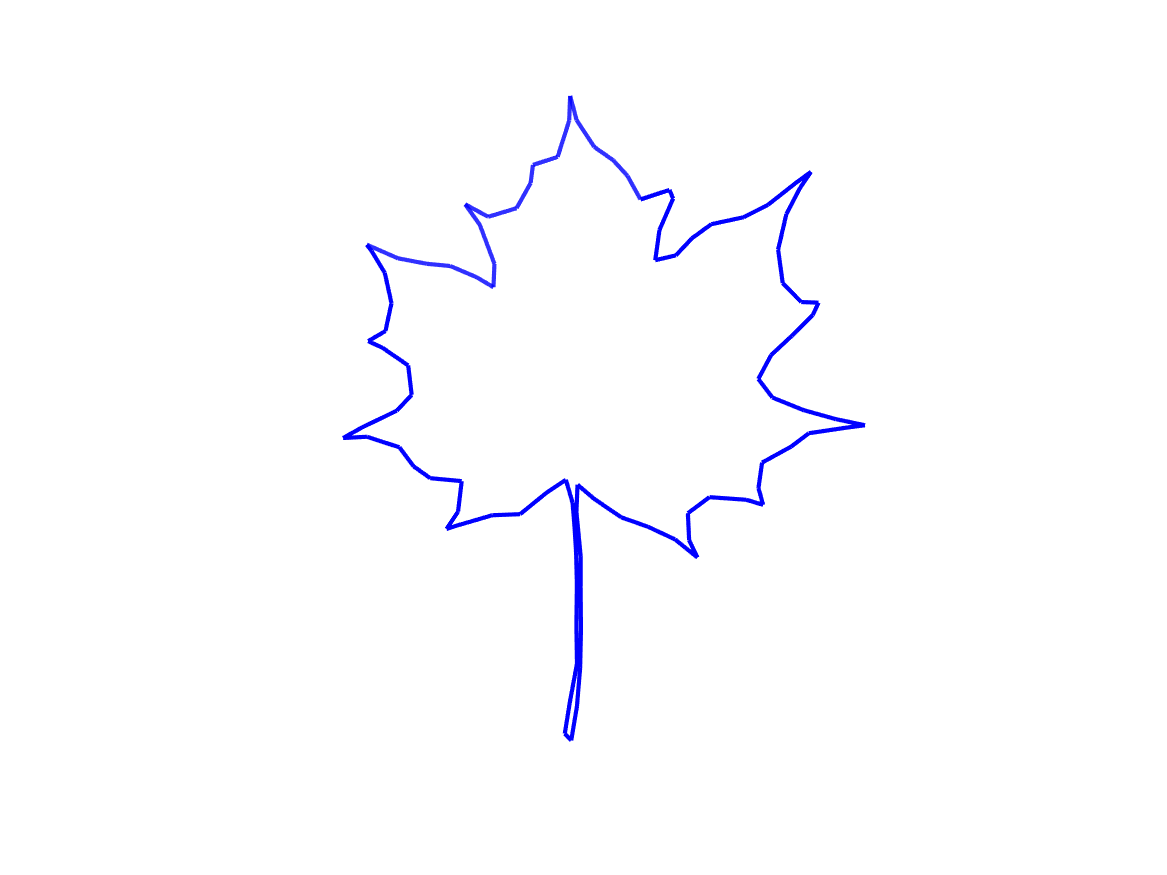}&
\includegraphics[angle=90,trim = 50mm 15mm 50mm 10mm ,clip,width=2.3cm]{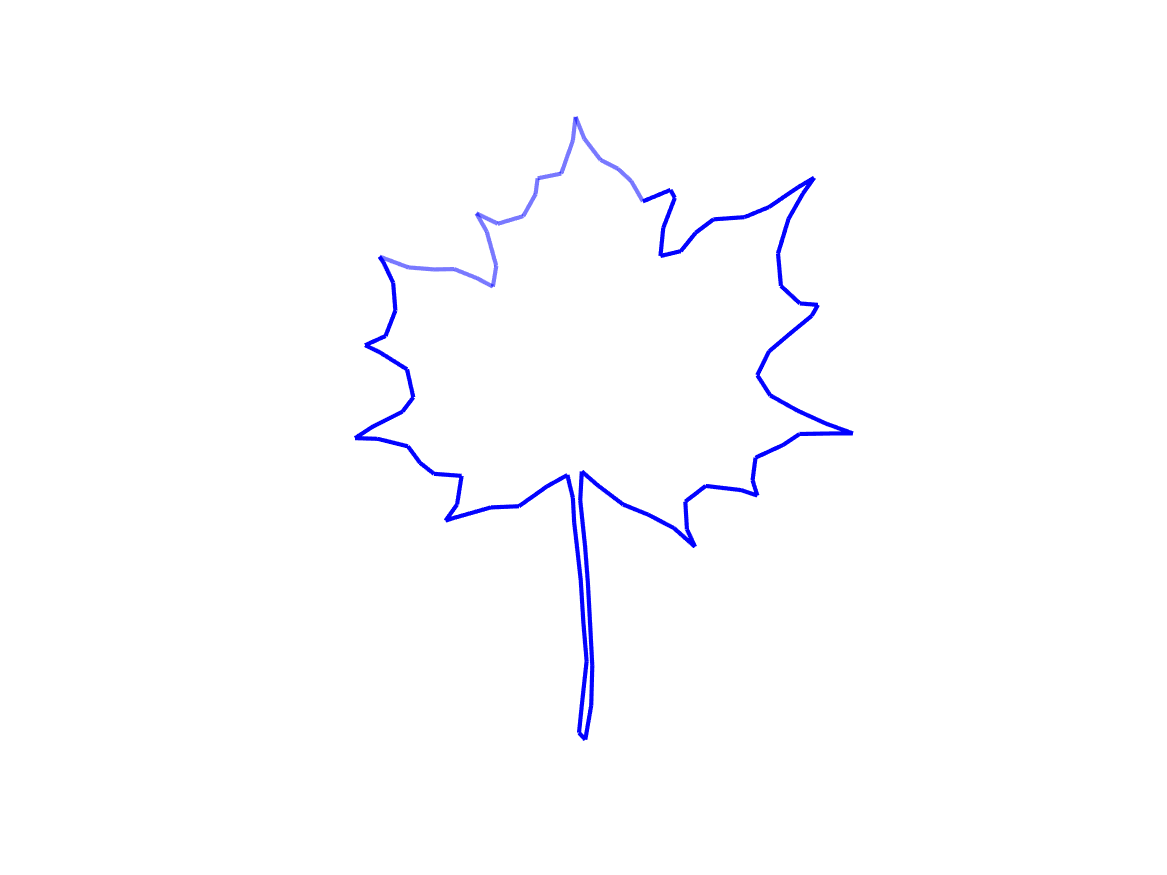}&
\includegraphics[angle=90,trim = 50mm 15mm 50mm 10mm ,clip,width=2.3cm]{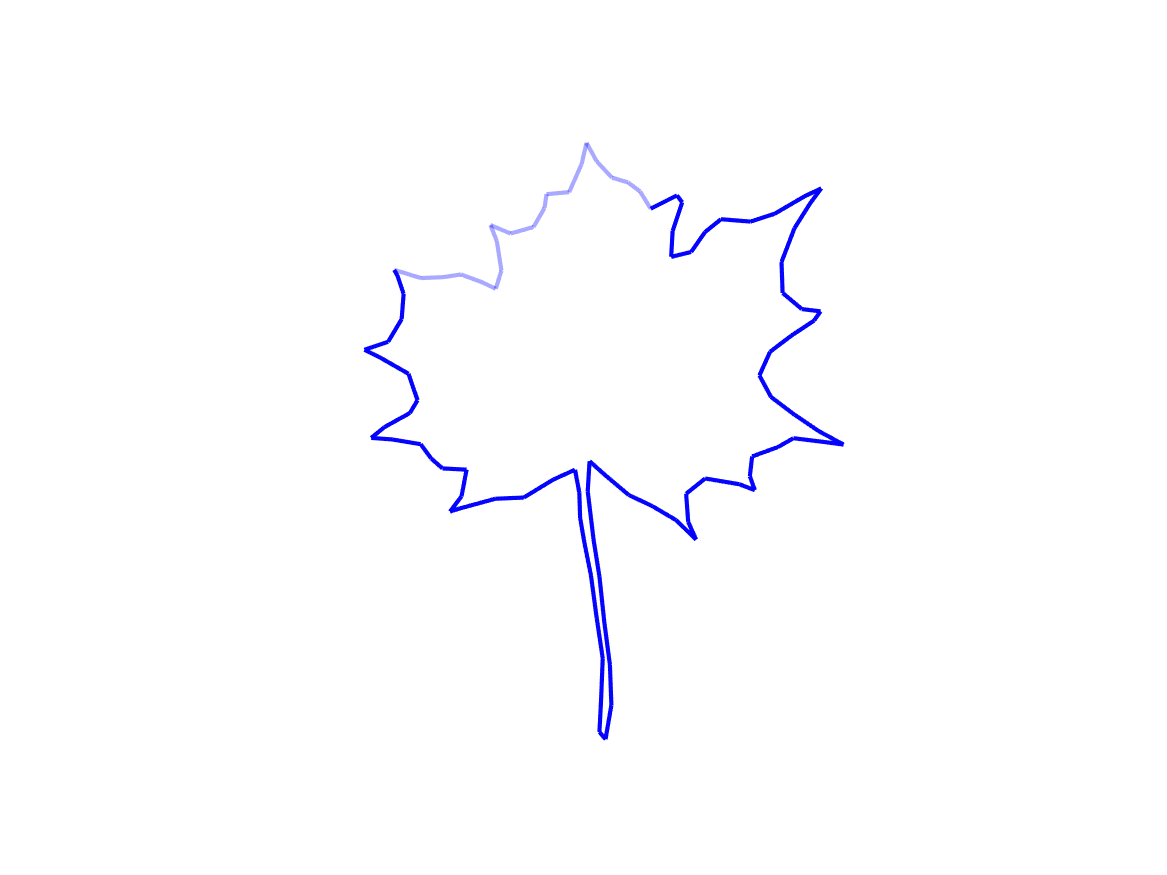}&
\includegraphics[angle=90,trim = 50mm 15mm 50mm 10mm ,clip,width=2.3cm]{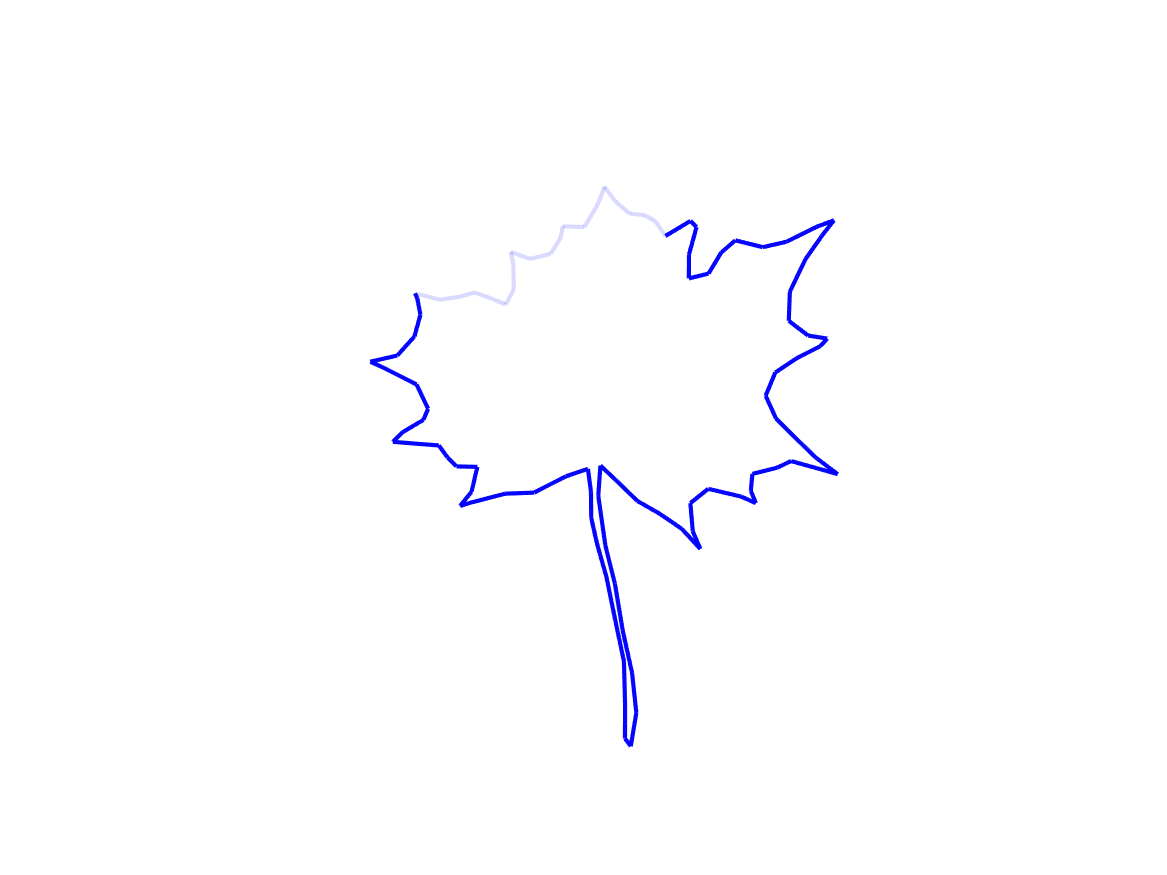}&
\includegraphics[angle=90,trim = 50mm 15mm 50mm 10mm ,clip,width=2.3cm]{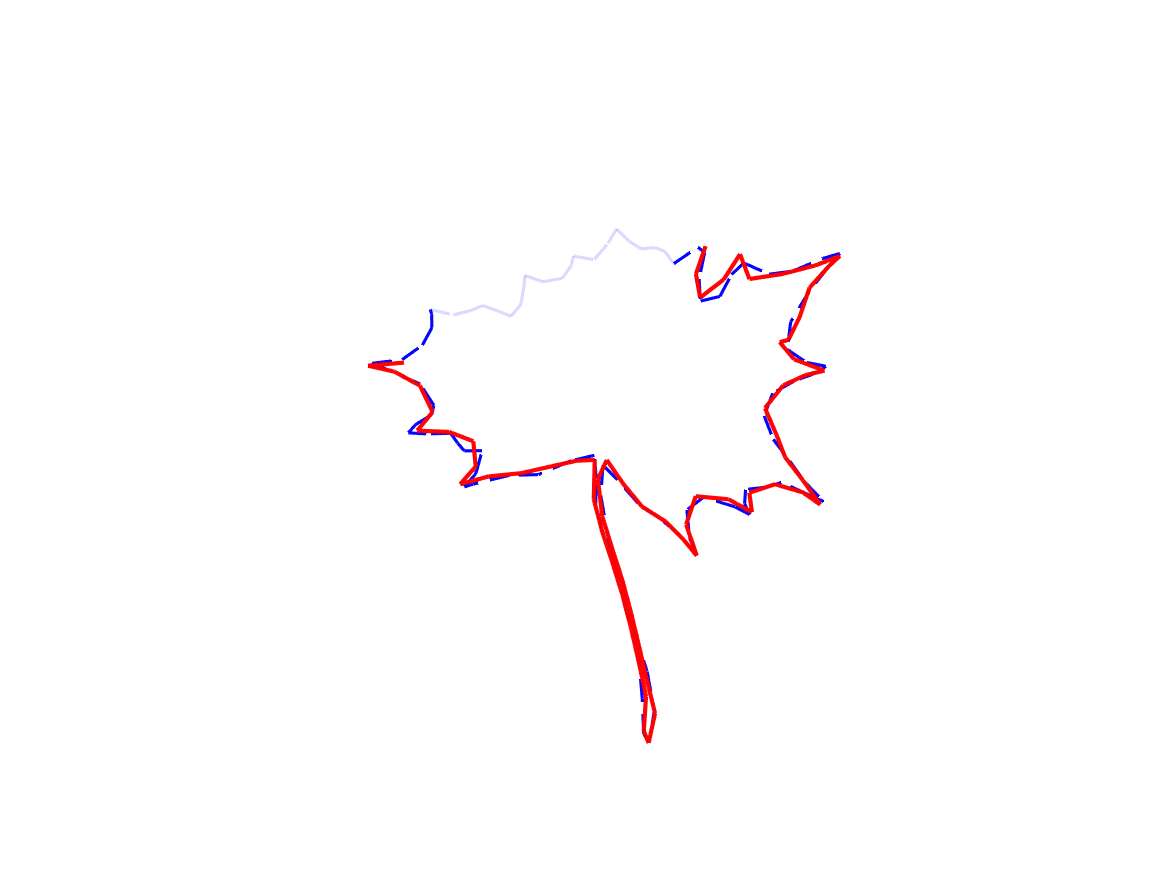}\\
\includegraphics[angle=90,trim = 60mm 15mm 60mm 10mm ,clip,width=2.3cm]{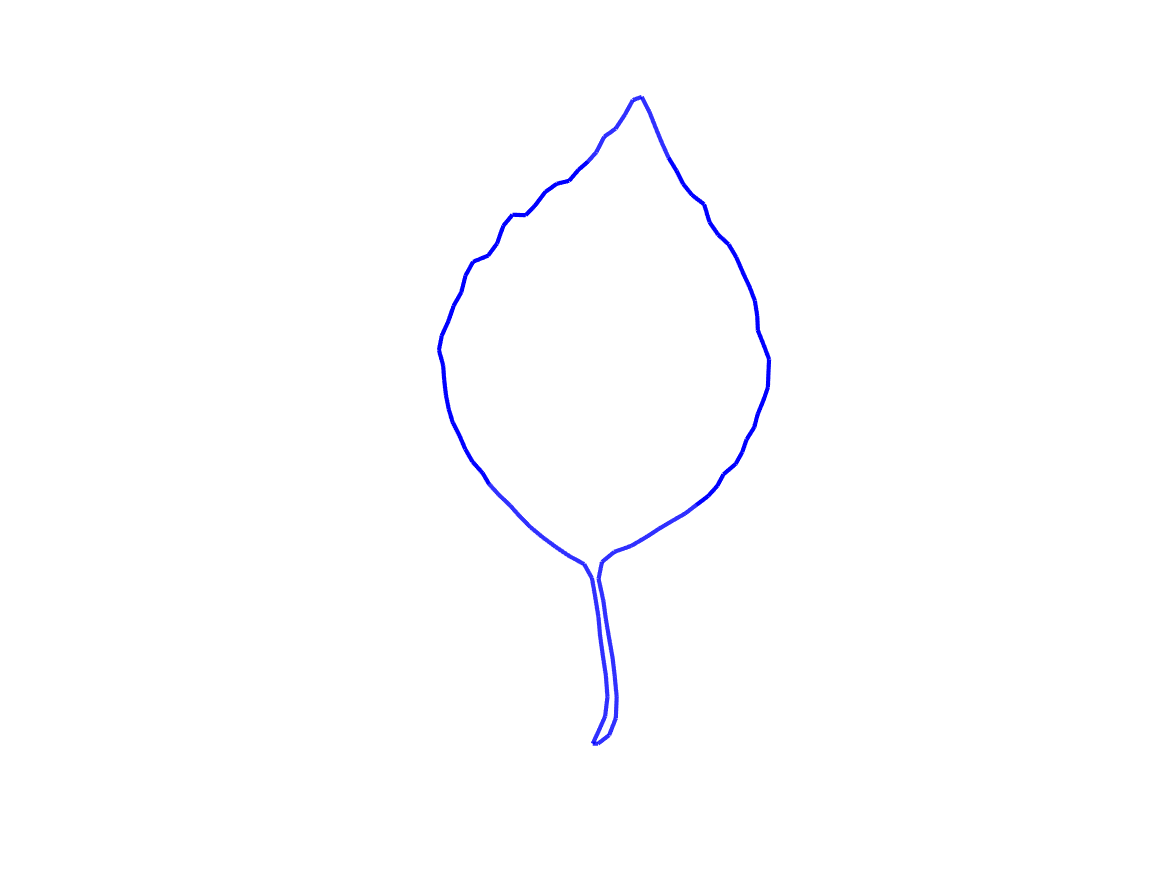}&
\includegraphics[angle=90,trim = 60mm 15mm 60mm 10mm ,clip,width=2.3cm]{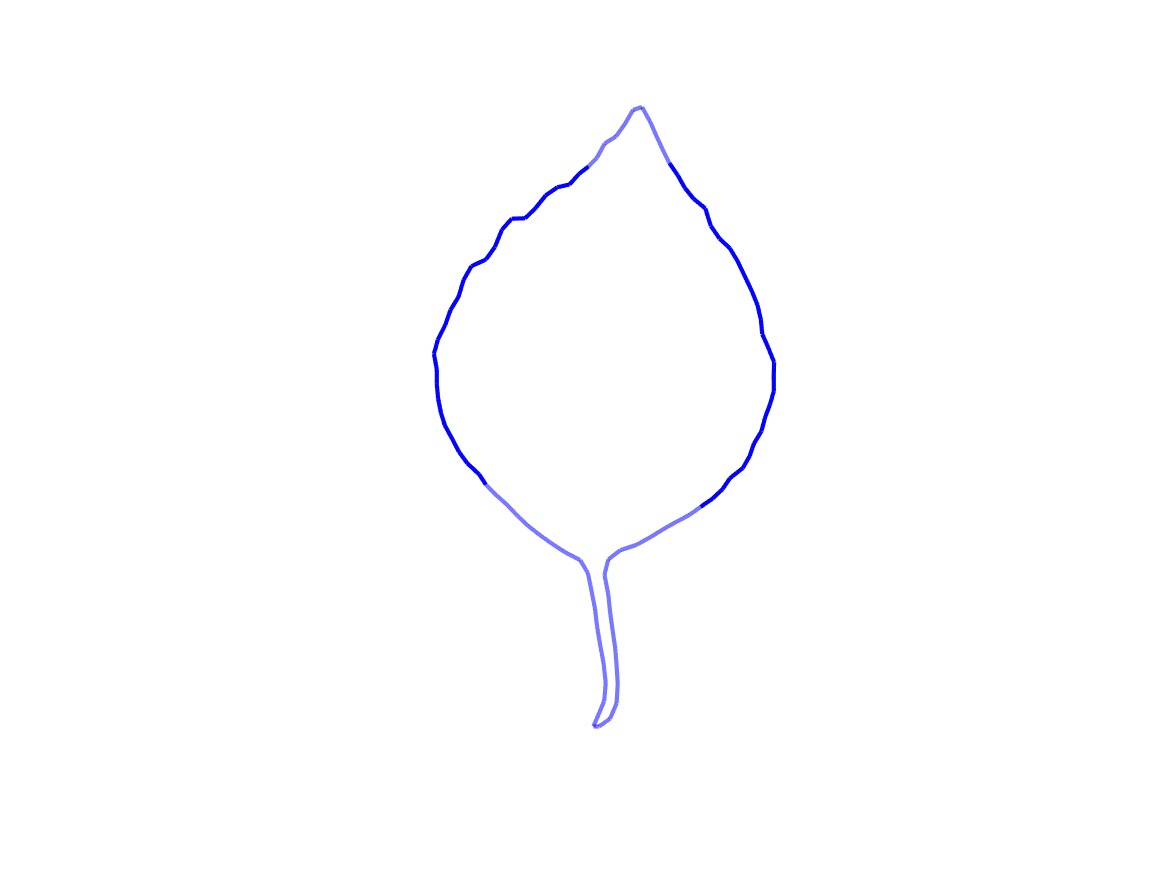}&
\includegraphics[angle=90,trim = 60mm 15mm 60mm 10mm ,clip,width=2.3cm]{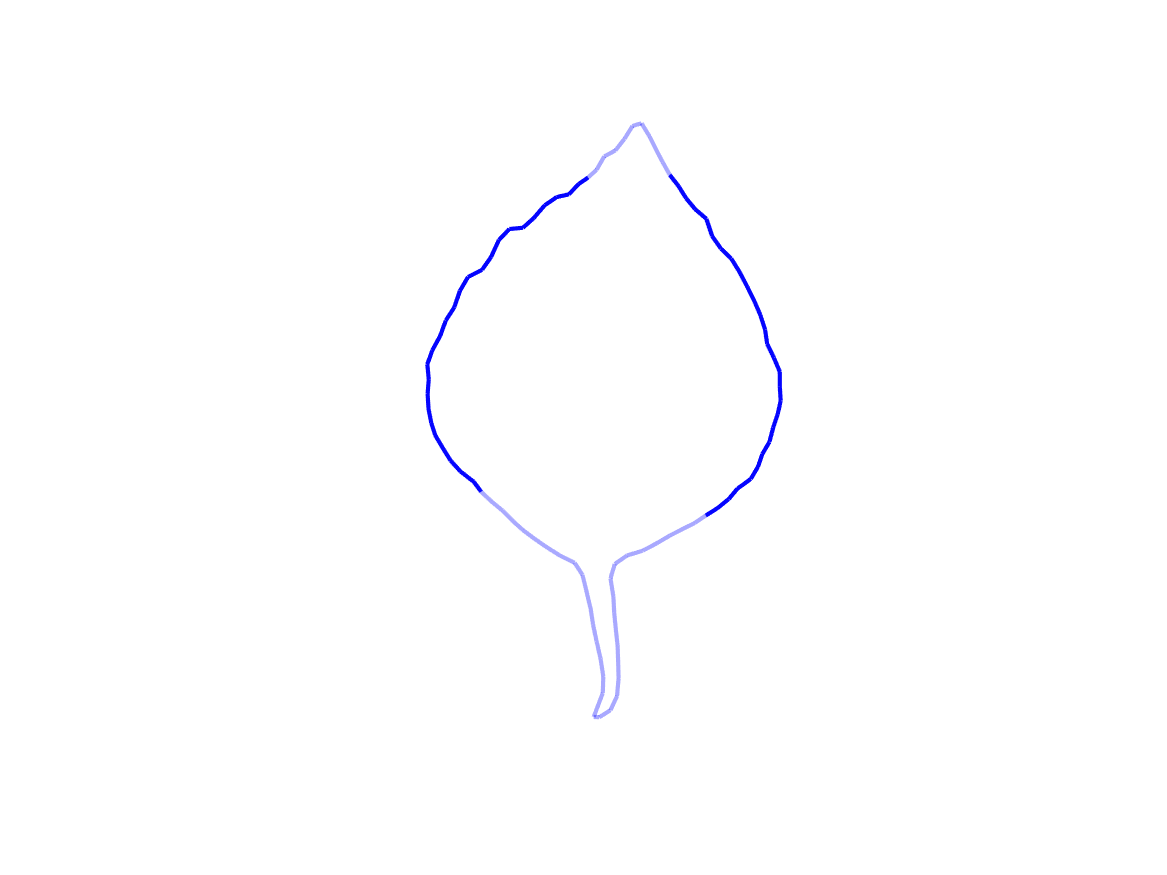}&
\includegraphics[angle=90,trim = 60mm 15mm 60mm 10mm ,clip,width=2.3cm]{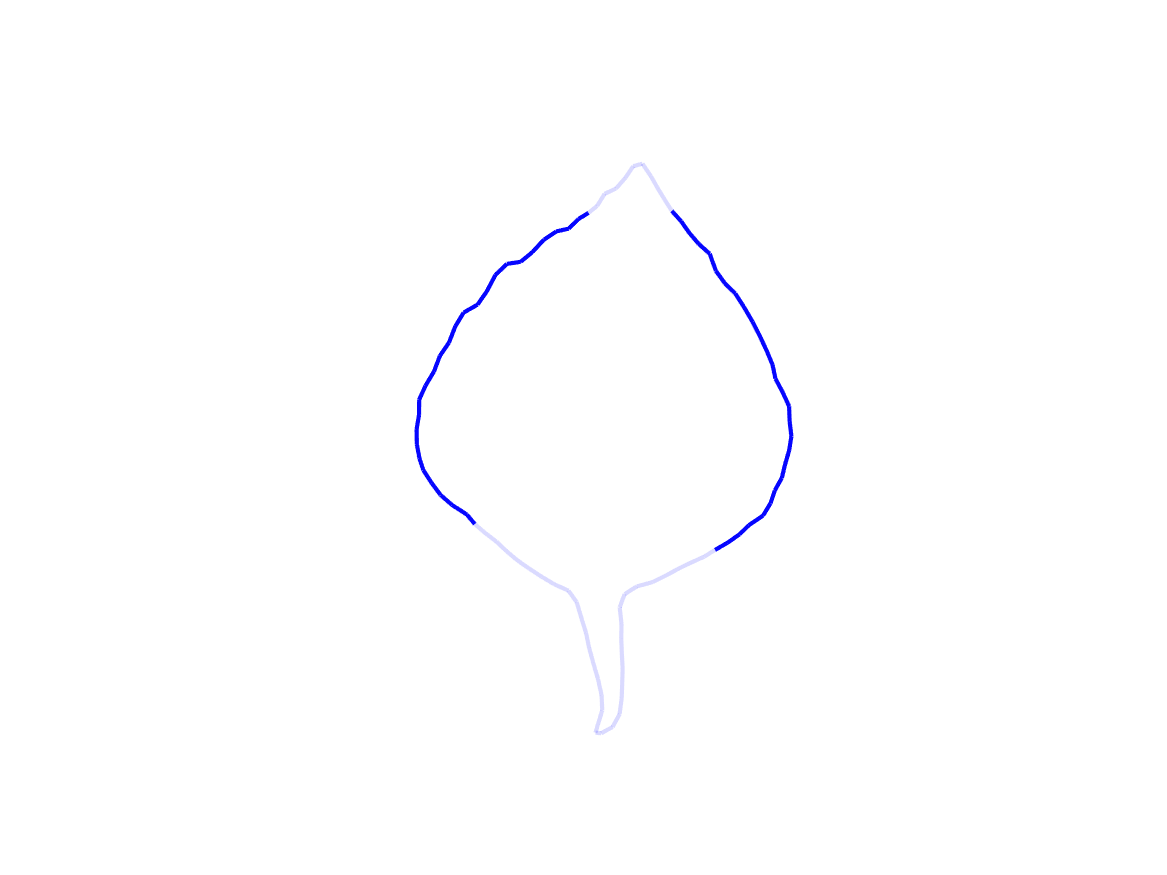}&
\includegraphics[angle=90,trim = 60mm 15mm 60mm 10mm ,clip,width=2.3cm]{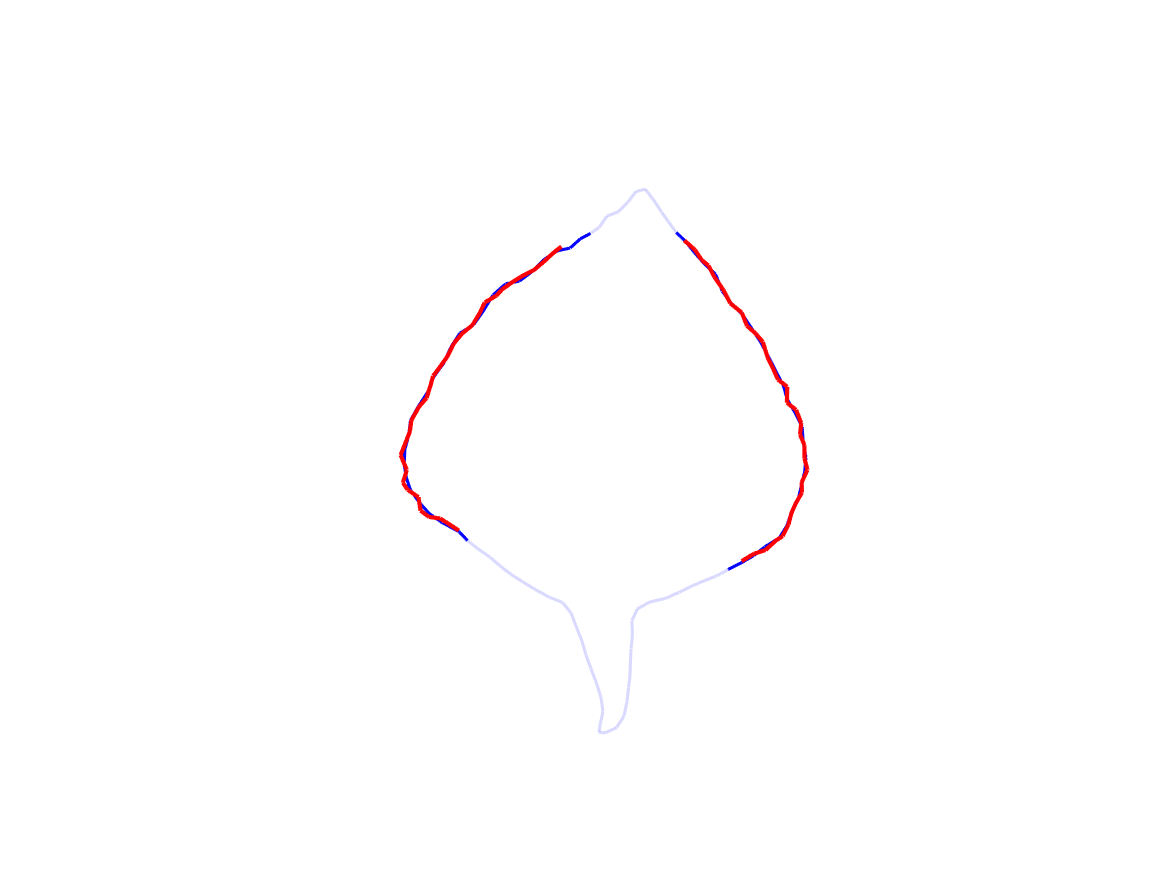}\\
\includegraphics[angle=90,trim = 40mm 15mm 40mm 10mm ,clip,width=2.3cm]{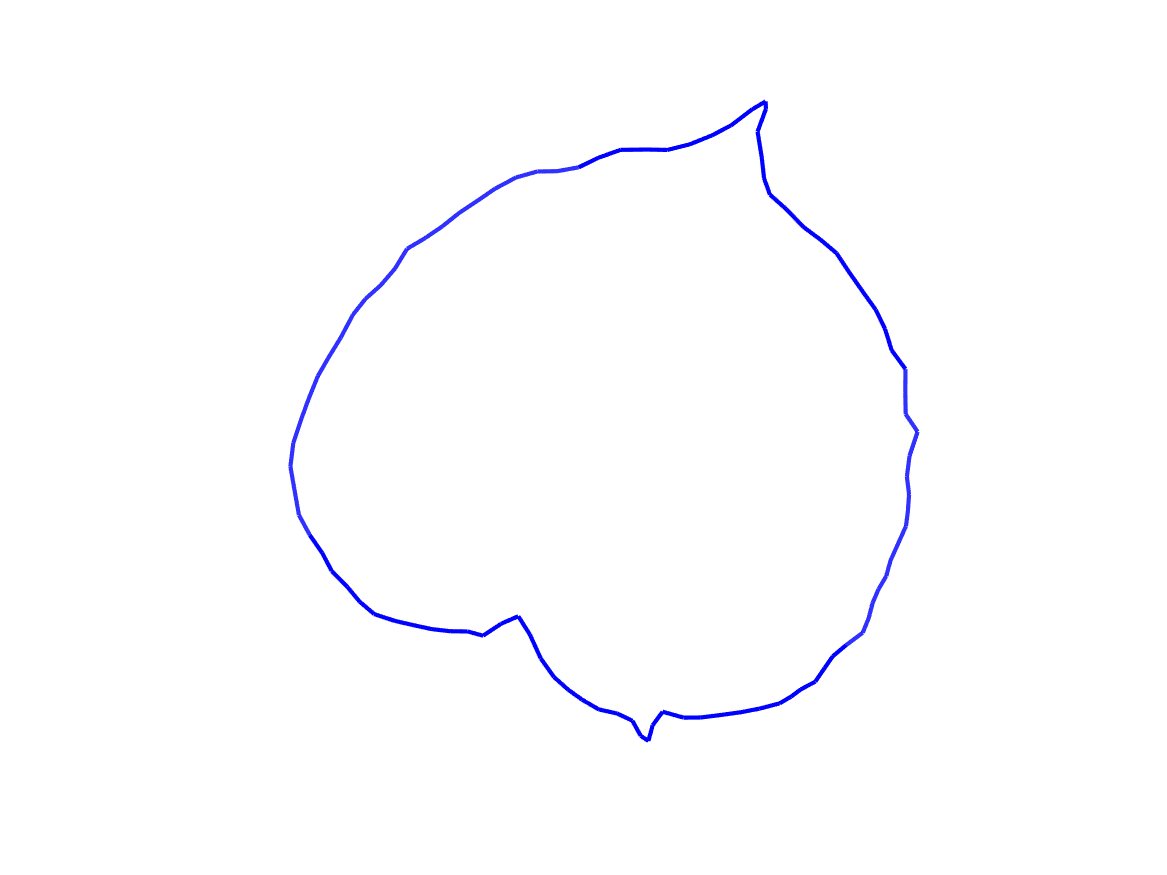}&
\includegraphics[angle=90,trim = 40mm 15mm 40mm 10mm ,clip,width=2.3cm]{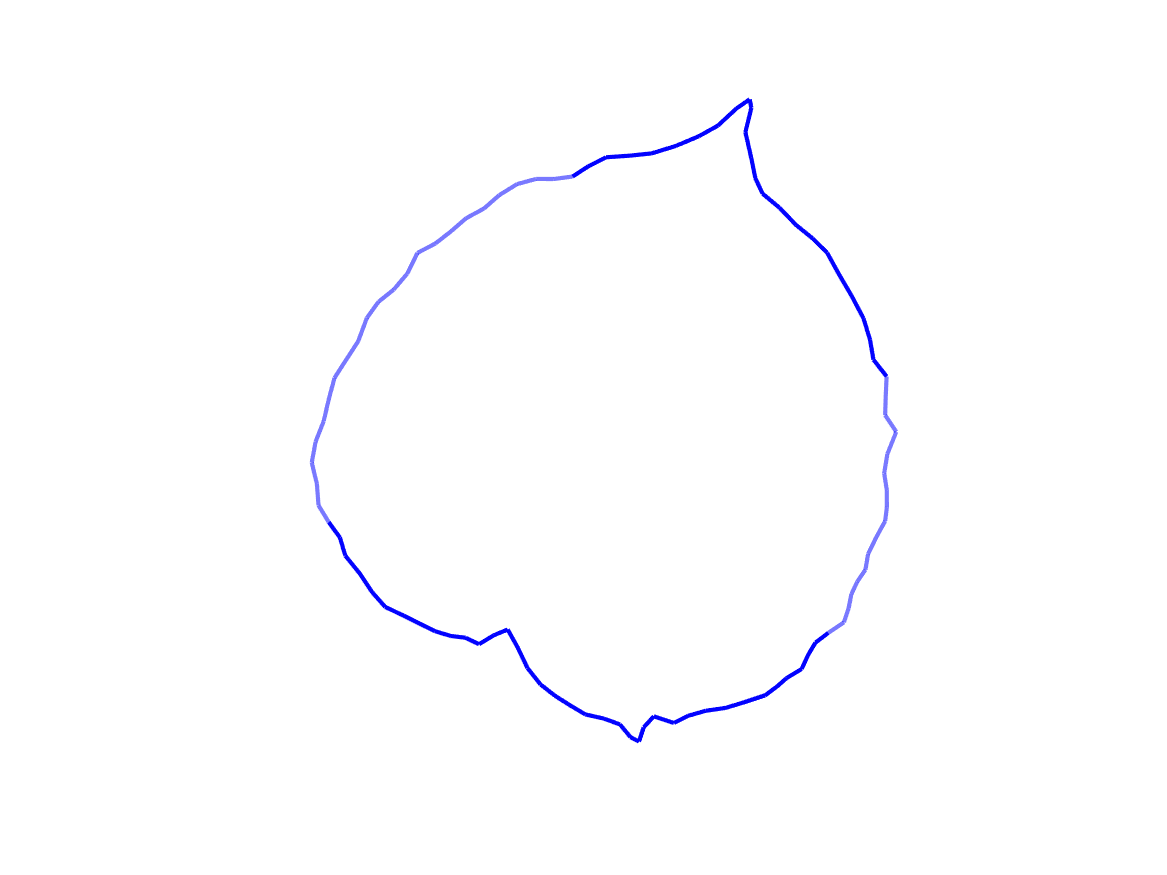}&
\includegraphics[angle=90,trim = 40mm 15mm 40mm 10mm ,clip,width=2.3cm]{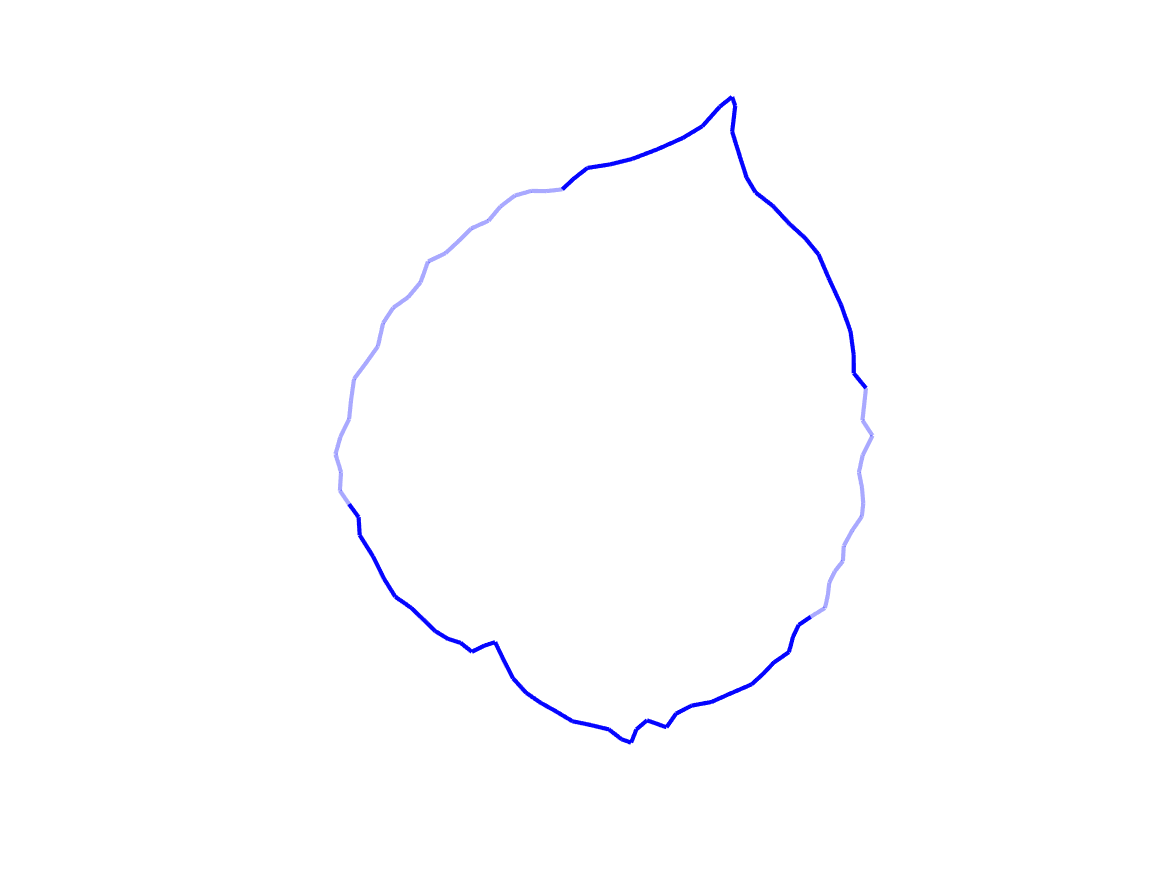}&
\includegraphics[angle=90,trim = 40mm 15mm 40mm 10mm ,clip,width=2.3cm]{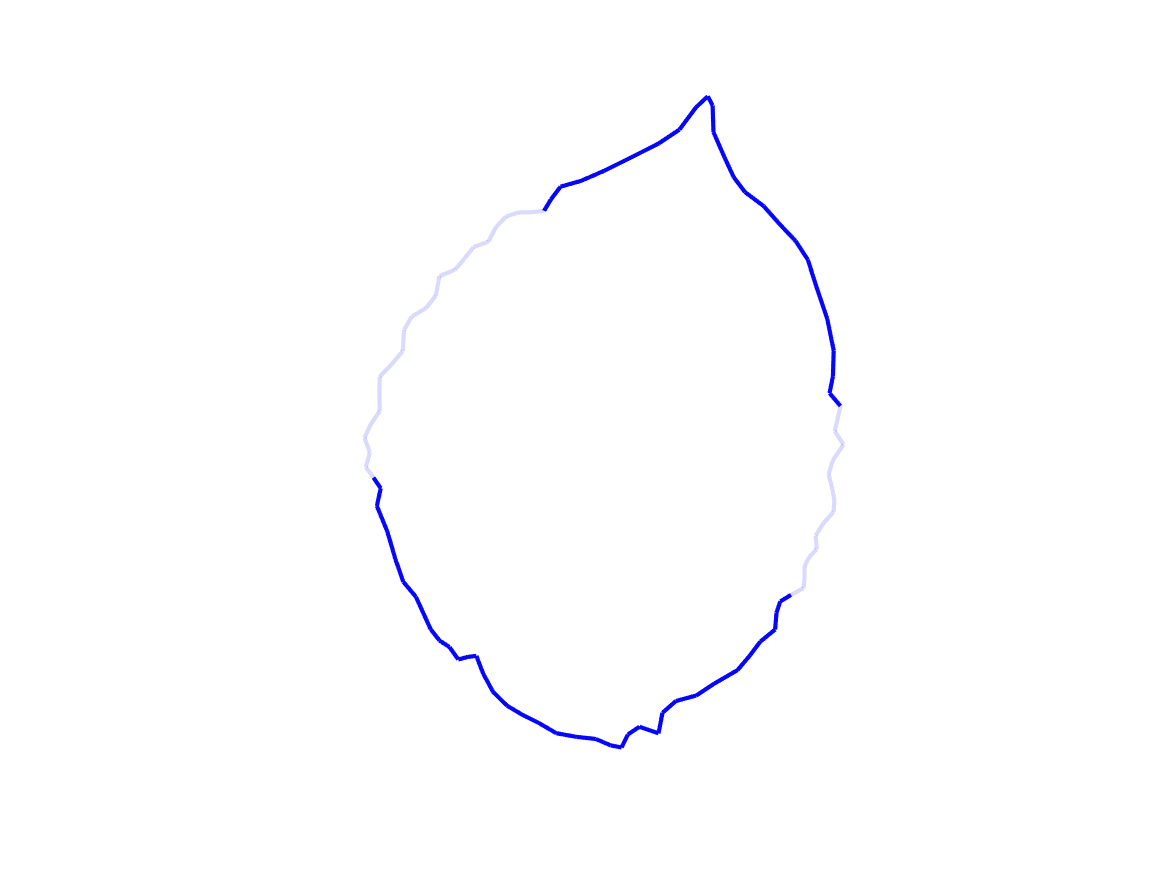}&
\includegraphics[angle=90,trim = 40mm 15mm 40mm 10mm ,clip,width=2.3cm]{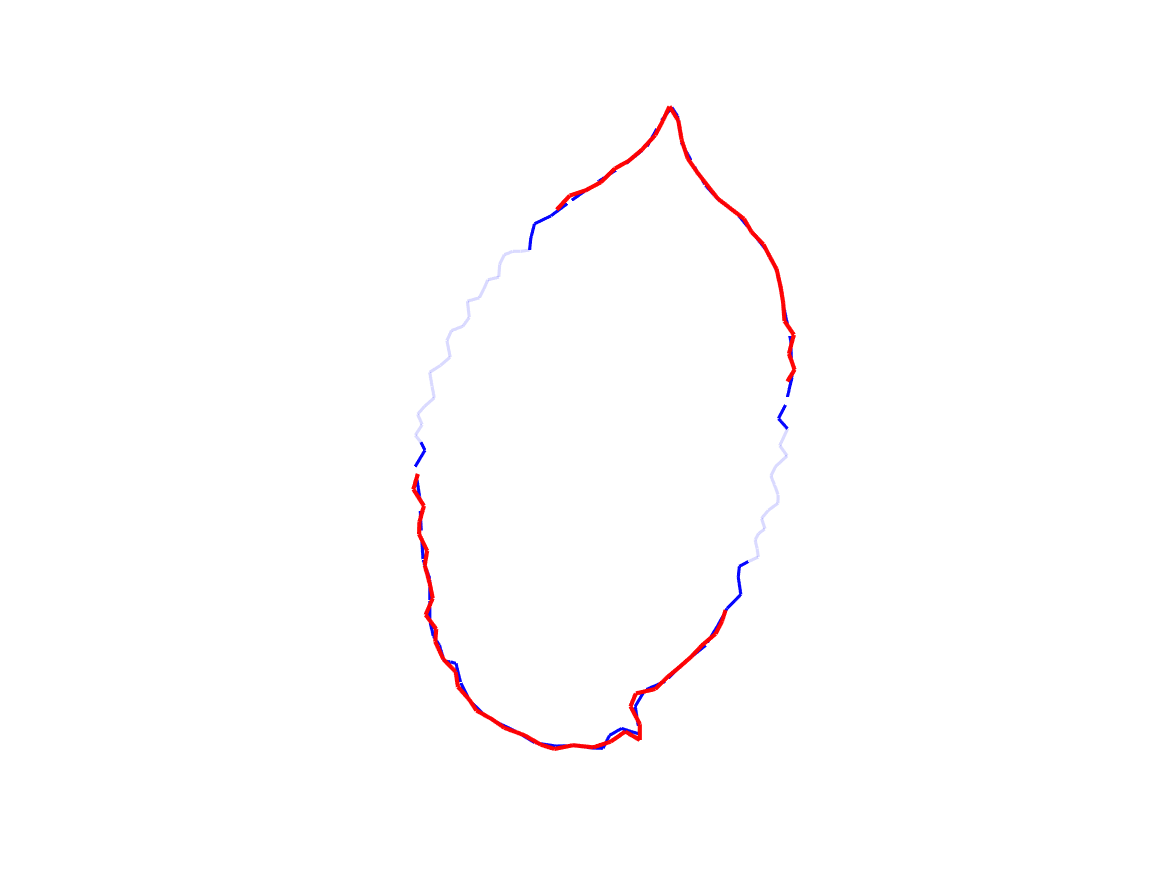}\\
$t=0$ & $t=0.25$ & $t=0.5$ & $t=0.75$ & $t=1$ 
\end{tabular}
\caption{\textbf{Matching incomplete leaves.} Geodesics between Swedish leaves with partial matching constraints. The source (blue at $t=0$) and target (red at $t=1$) have distinct topologies, with the source being a closed curve, and the target being e.g., an open curve or having multiple connected components. The target is overlayed on the transformed source $c(1)$ on the right, and parts of the transformed source which get ``erased'', i.e. where the estimated weight function vanishes, are colored in progressively transparent shades of blue. (Top to bottom) The estimated geodesic distances are (i) $\overline{\operatorname{dist}}(c_0,c_1)= 1.29$, (ii) $\overline{\operatorname{dist}}(c_0,c_1)= 0.95$, (iii) $\overline{\operatorname{dist}}(c_0,c_1)= 0.51$, (iv) $\overline{\operatorname{dist}}(c_0,c_1)= 1.35$.}
\label{fig:match_leaves_source_weight}
\end{center}
\end{figure}

\textbf{Matching partially observed data.}
In many practical applications, datasets frequently involve incomplete shapes. This can be due to multiple factors such as occlusions, segmentation issues or simply inadequate fields of view during the data acquisition process. In such cases, performing shape registration using usual elastic shape matching algorithms will typically bend, stretch or compress portions of the source in an attempt to e.g. fill in some of those missing parts. This results in unnatural deformations which lead to a significant overestimation of the Riemannian distance between the shape graphs, as we pointed out with Fig.~\ref{fig:match_branches_compare_fixed_vs_source_weights}. One way \red{to} overcome this issue is through partial matching, which broadly refers to the process of registering a source shape onto an (unknown) subset of a target shape. In our proposed weighted shape graph matching framework, this can be achieved by setting vanishing weights to portions of the transformed source that cannot be matched to the target. The automatic estimation of these vanishing parts is achieved precisely thanks to the optimality criterion of \eqref{eq:relaxed_weighted_shape_graph_match}, where the $\{0,1\}$ penalty $\tilde{F}_{\rho_0}$ discussed in Section \ref{ssec:weigh_regularizer} enforces the weight function to be essentially binary. We illustrate the efficiency of this approach with experiments on 2D curves extracted from the Swedish leaf dataset\footnote{\href{https://www.cvl.isy.liu.se/en/research/datasets/swedish-leaf/}{https://www.cvl.isy.liu.se/en/research/datasets/swedish-leaf/}}, for which different parts of the target curves were artificially removed, see Fig.~\ref{fig:match_leaves_source_weight}.

\textbf{Matching with topological variations}
As outlined in Section~\ref{ssec:limitations} with Fig.~\ref{fig:match_branches_compare_fixed_vs_source_weights}, beyond the situation of partial data observation, one advantage of our weighted shape graph matching framework is that it can account for certain topological changes between source and target shape graphs.

\begin{figure}[h!]
\begin{tabular}{ccccc}
\includegraphics[trim = 80mm 90mm 80mm 20mm ,clip,width=2.5cm]{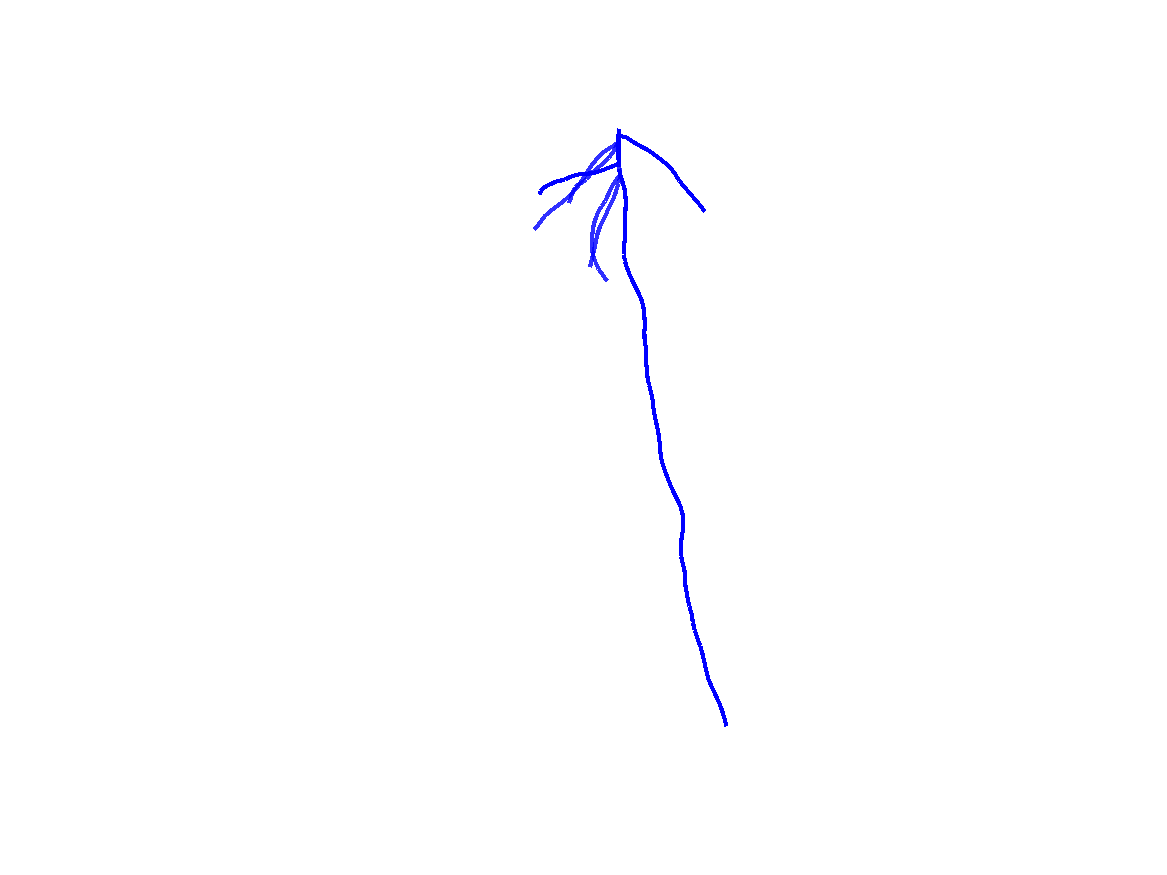}&
\includegraphics[trim = 80mm 90mm 80mm 20mm ,clip,width=2.5cm]{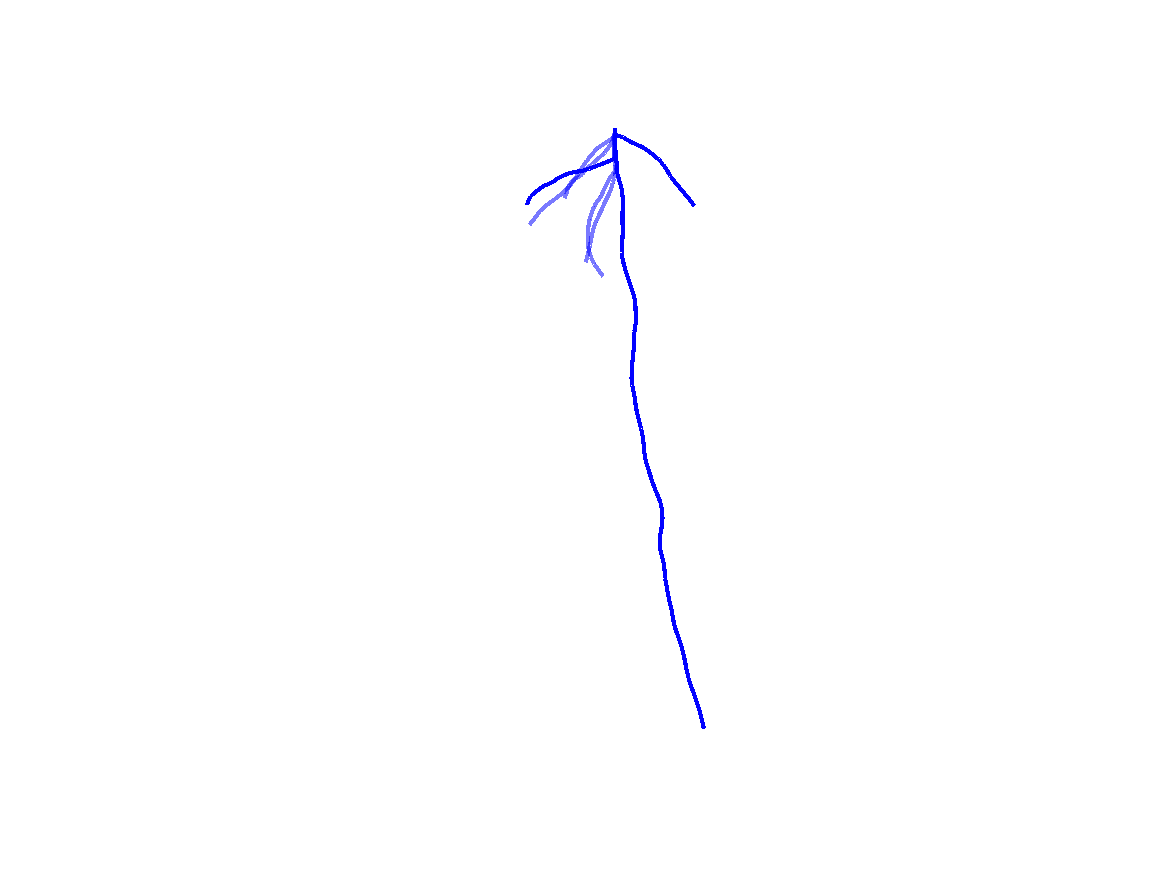}&
\includegraphics[trim = 80mm 90mm 80mm 20mm ,clip,width=2.5cm]{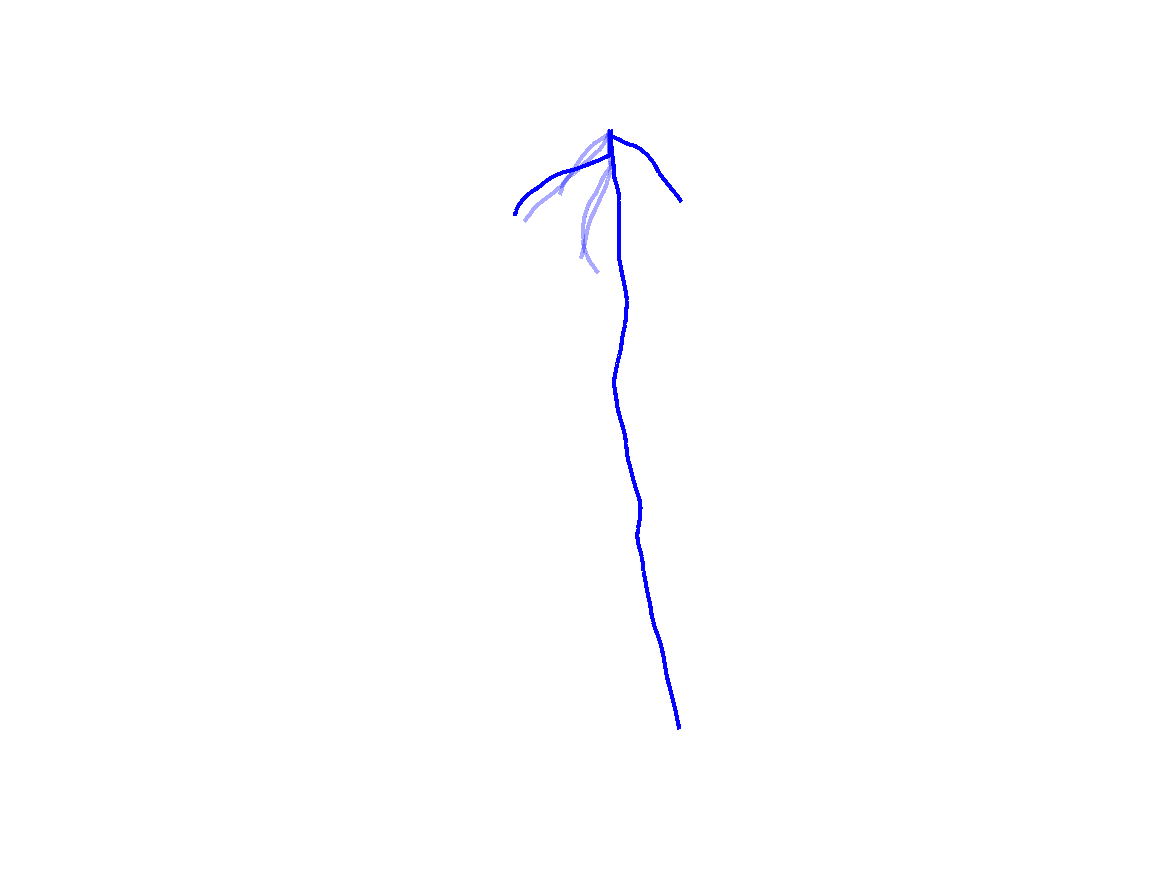}&
\includegraphics[trim = 80mm 90mm 80mm 20mm ,clip,width=2.5cm]{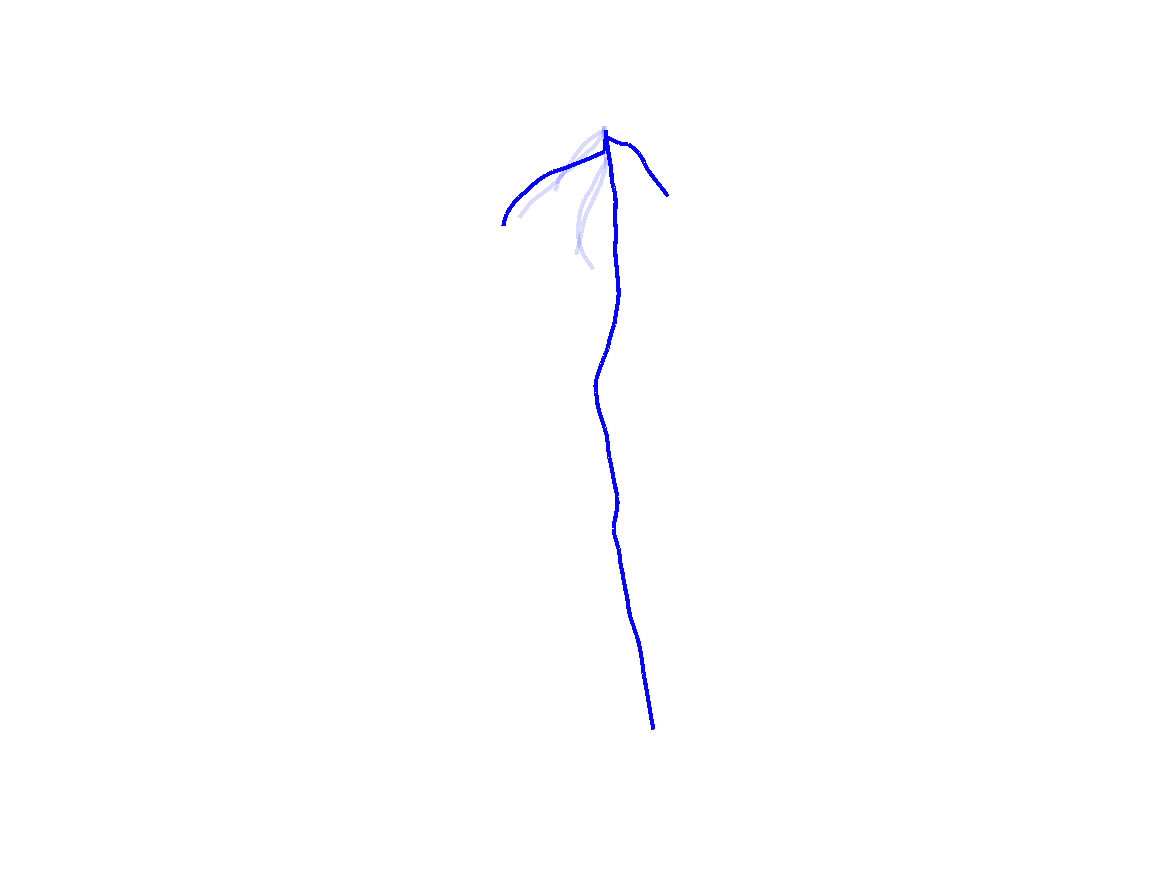}&
\includegraphics[trim = 80mm 90mm 80mm 20mm ,clip,width=2.5cm]{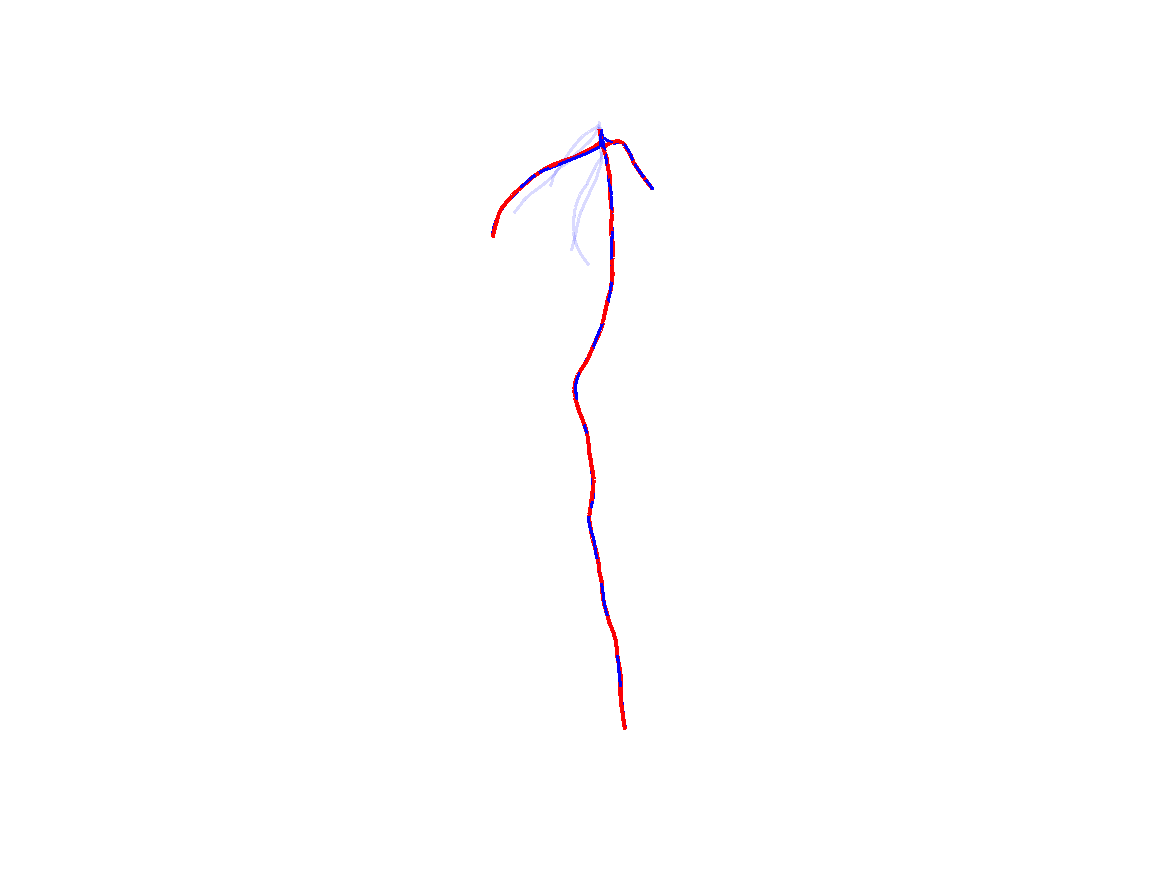}\\
$t=0$ & $t=0.25$ & $t=0.5$ & $t=0.75$ & $t=1$ 
\end{tabular}
\caption{\textbf{3D maize root systems.} The source (blue at $t=0$) is a 3D maize root system with multiple lateral roots, and the target (red at $t=1$) is another maize root system with only 2 lateral roots. The extra branches of the transformed source which get ``erased'' are colored in progressively transparent shades of blue. The estimated geodesic distance here is $\overline{\operatorname{dist}}(c_0,c_1)=0.64$. We note that the left branch of the target matches with the particular branch on the transformed source shown above because matching with this specific branch requires the least amount of geometric deformation in $\R^3$ (and hence the least amount of energy) when compared to any of the 4 other branches. The algorithm then decides to erase these other 4 branches.
\label{fig:match_3droots_source_weight}}
\end{figure}

\begin{figure}[h!]
\begin{center}
\begin{tabular}{ccccc}
\includegraphics[angle=0, trim = 69.5mm 45mm 64.5mm 40mm ,clip,height=1.5cm]{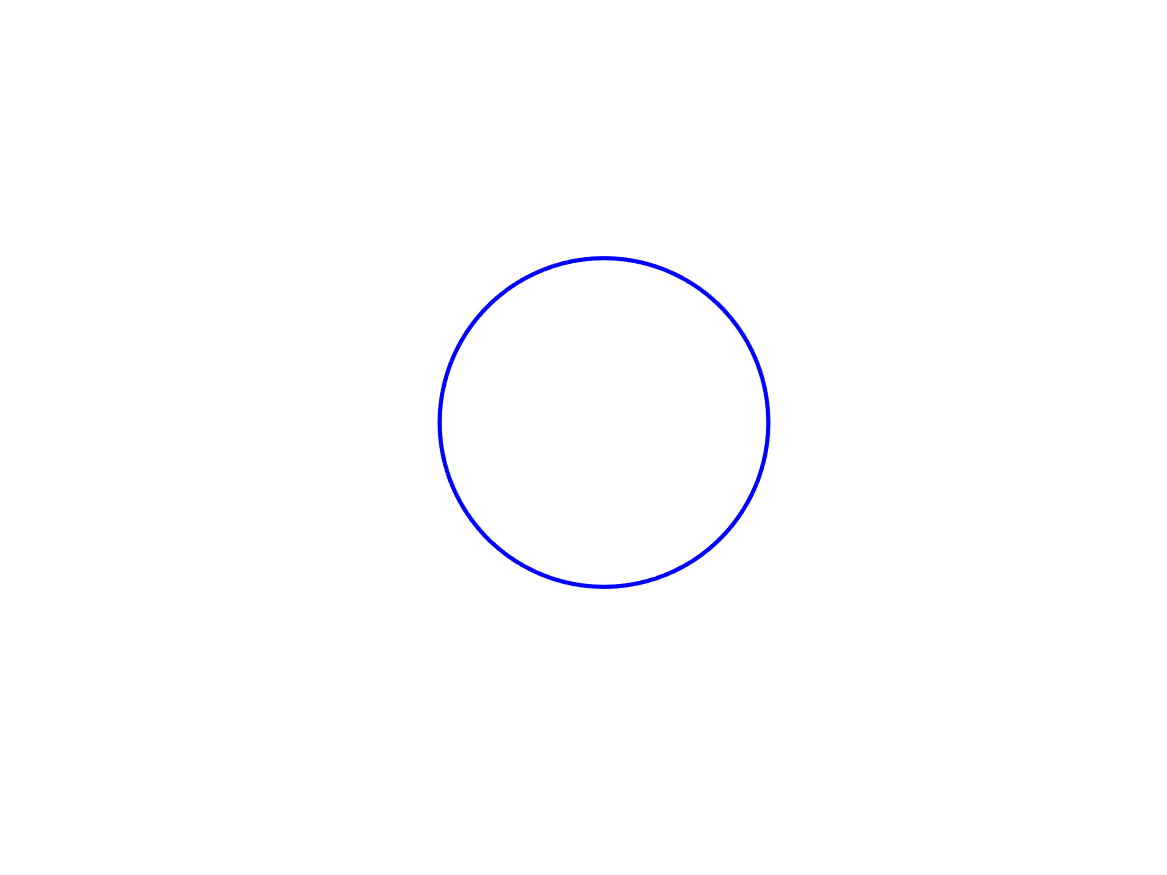}&
\includegraphics[angle=0, trim = 62mm 45mm 57mm 40mm ,clip,height=1.5cm]{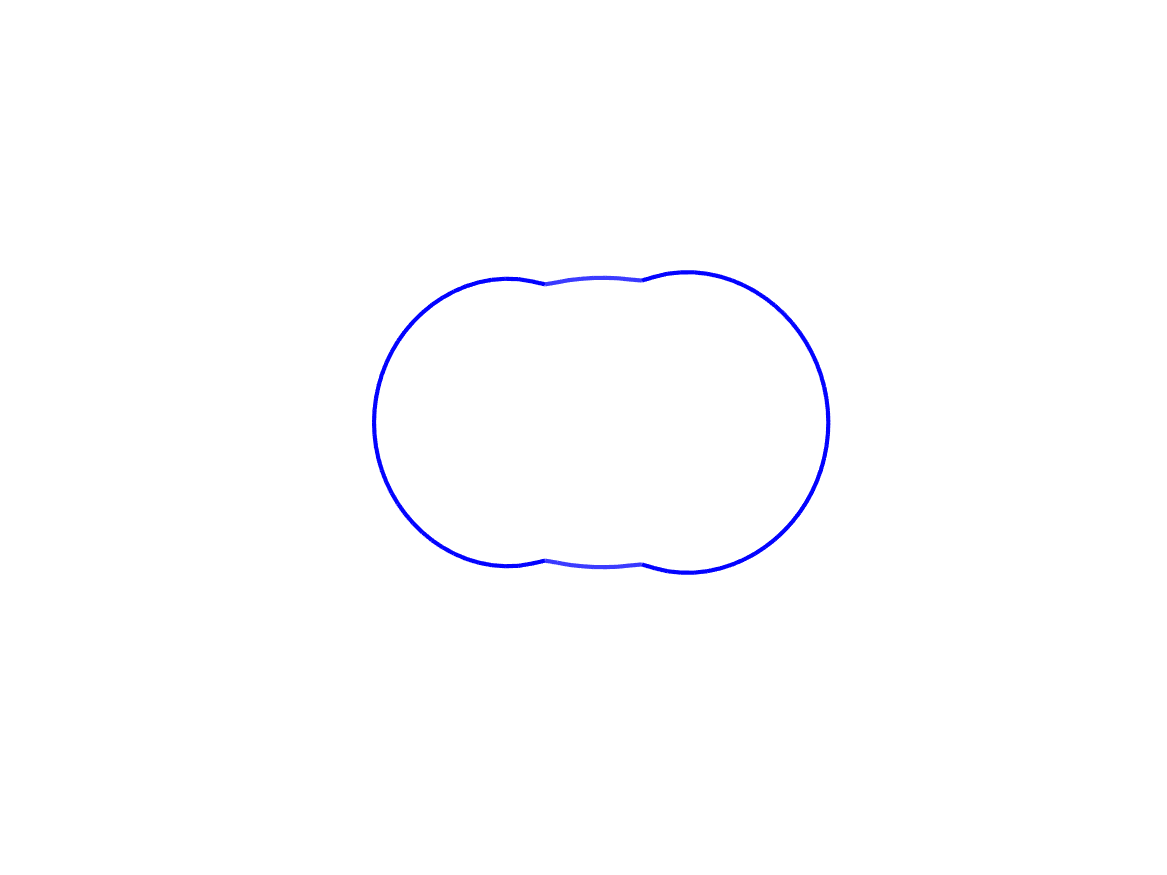}&
\includegraphics[angle=0, trim = 53mm 45mm 47.5mm 40mm ,clip,height=1.5cm]{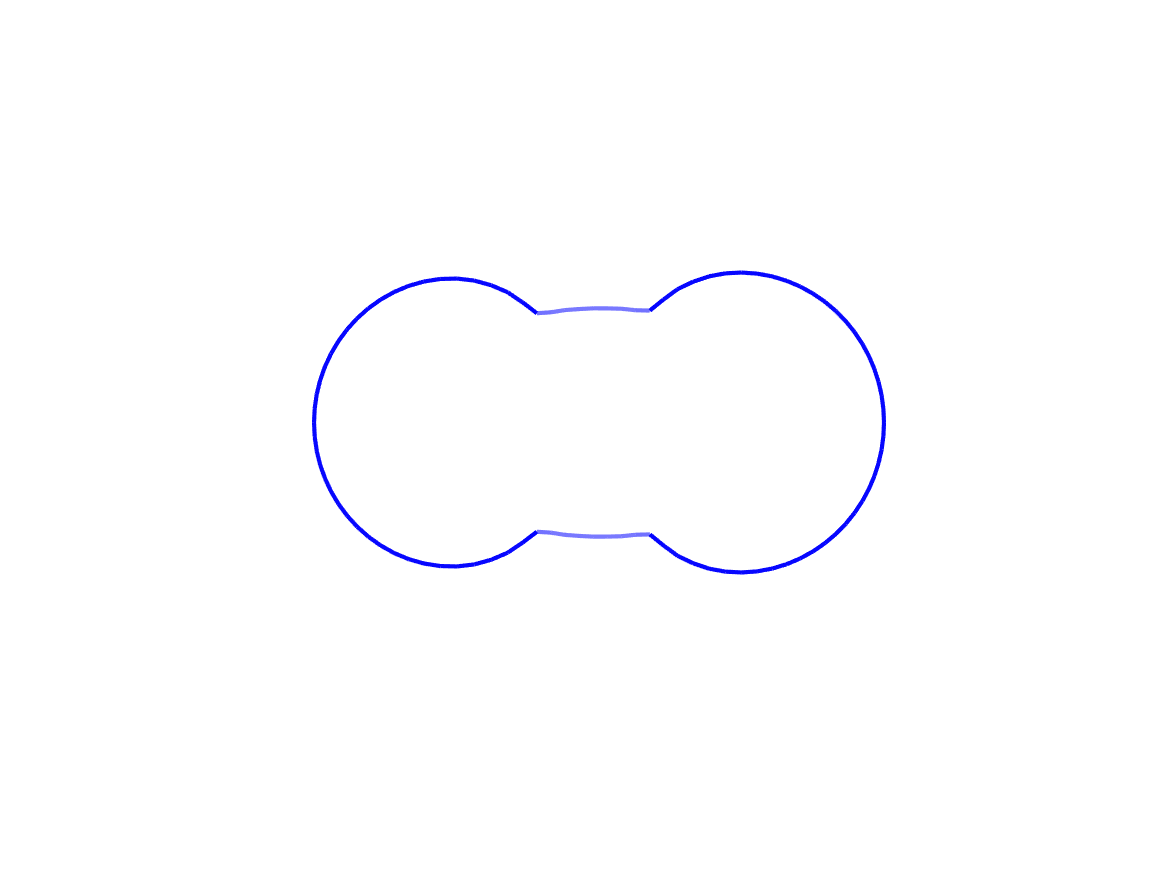}&
\includegraphics[angle=0, trim = 44mm 45mm 38mm 40mm ,clip,height=1.5cm]{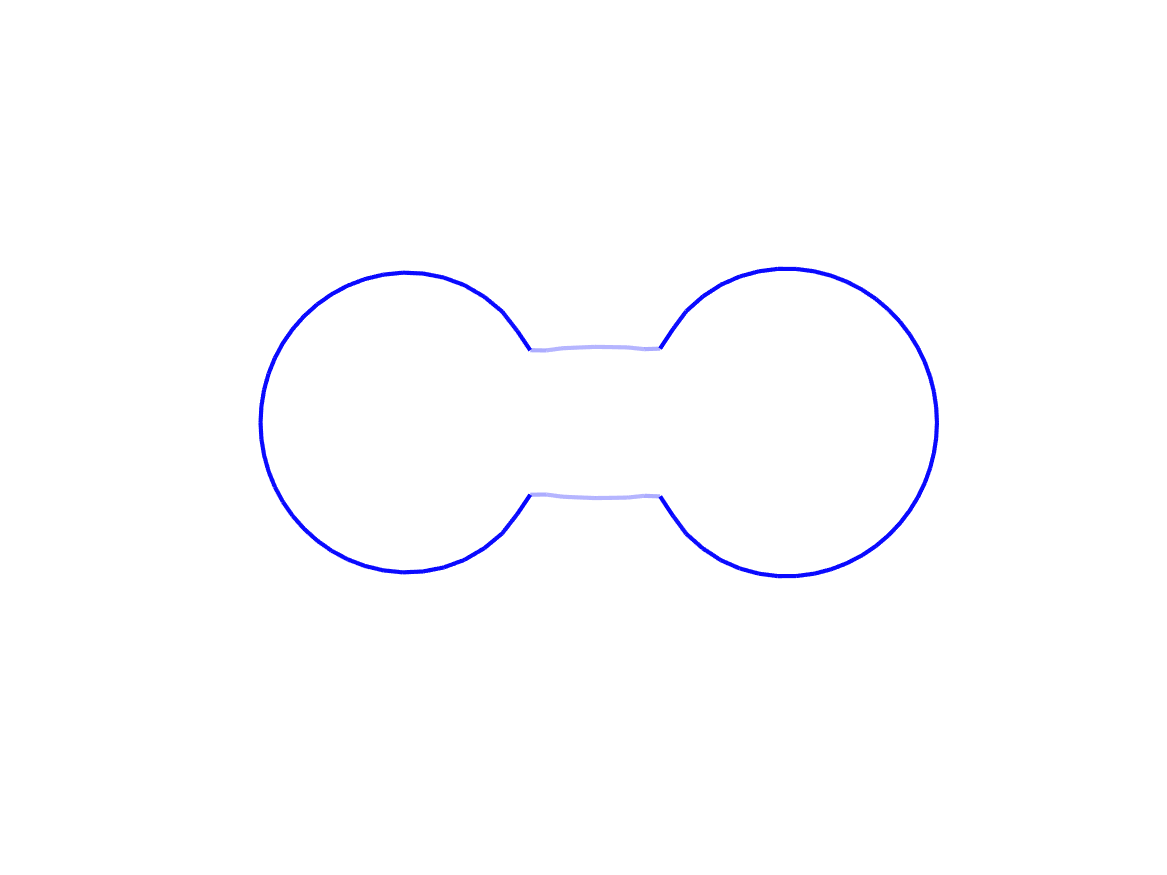}&
\includegraphics[angle=0, trim = 32mm 45mm 25mm 40mm ,clip,height=1.5cm]{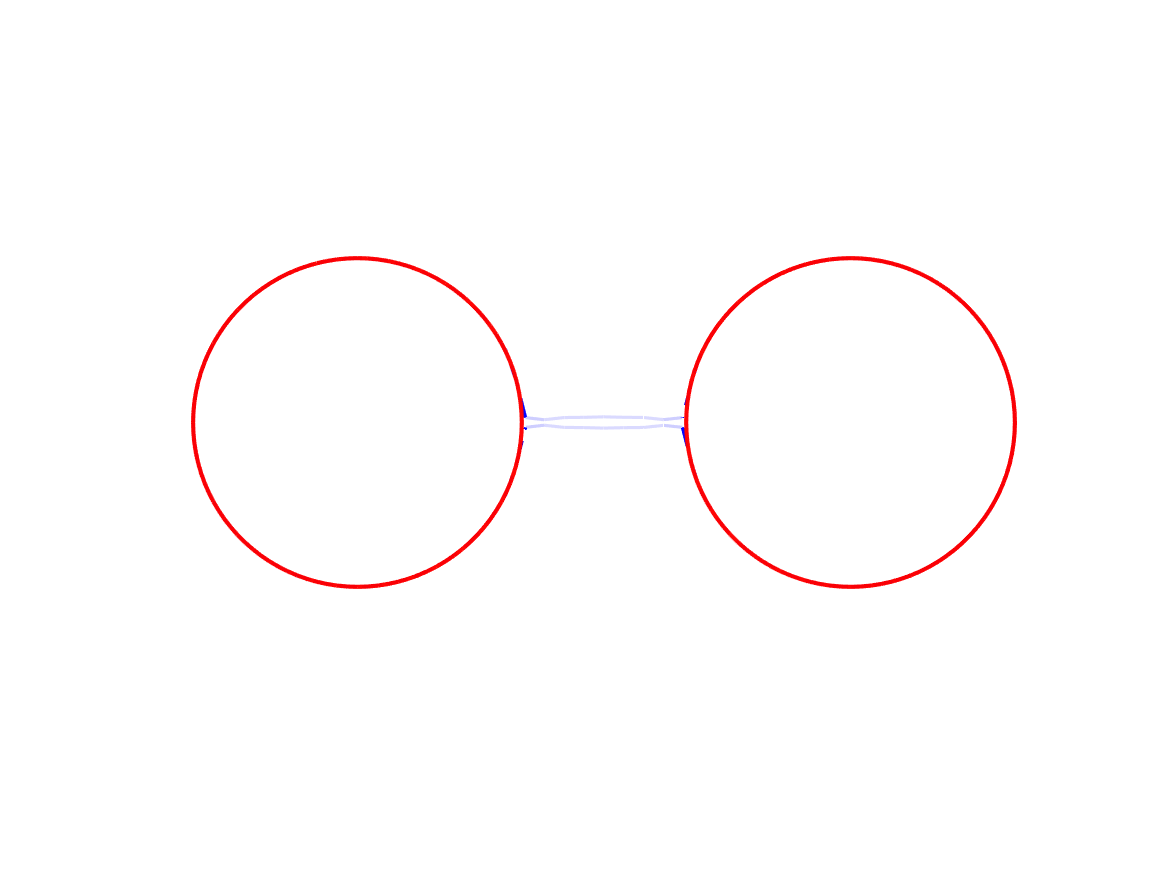}\\
$t=0$ & $t=0.25$ & $t=0.5$ & $t=0.75$ & $t=1$ 
\end{tabular}
\end{center}
\caption{\textbf{Splitting into multiple connected components.} In this example, the source shape graph consists of a single circle while the target is made of two disconnected circles. The target is overlayed on the transformed source on the rightmost image, where the estimated weights on the ``bridge" between the two connected components are equal to $0$, shaded in translucent blue. The method thus effectively erases this part. We find an estimated geodesic distance of $\overline{\operatorname{dist}}(c_0,c_1)= 2.64$.}
\label{fig:match_circles_source_weight}
\end{figure}

Here again, by allowing variations in the weight function in addition to the geometric deformation, we can effectively allow certain components on the transformed source shape graph to get ``erased'' if they have no corresponding matching part in the target. This is essential for instance in situations where the source shape graph has extra branches compared to the target, as we show with the examples of 3D maize root systems\footnote{\href{https://github.com/RSA-benchmarks/collaborative-comparison}{https://github.com/RSA-benchmarks/collaborative-comparison}} in Fig.~\ref{fig:match_3droots_source_weight}. It can also constitute an effective way to model a shape undergoing certain topological transformations such as a split into distinct connected components, as illustrated with the synthetic example of Fig.~\ref{fig:match_circles_source_weight}.

Note that so far, we have only focused on the asymmetric scenario in which the source shape has extra components that need to be erased to match the target. However, one could technically model mass creation as well in order to tackle the opposite situation. However, this case is less obvious since one needs prior knowledge of potential components that need to be added to the source shape, because in our model weights are only defined on existing components. Similarly to what was proposed with other frameworks, such as the supertree model of \cite{feragen2020statistics}, one idea is to assume that the source shape graph contains additional phantom components or branches whose weights are initially set to $0$. We then let our algorithm increase those weights if needed. This is illustrated with the example of Fig.~\ref{fig:match_masscreatedeleteroot_source_weight} in which one source branch is being created while another one is being erased. Here as well, this process yields a quite natural geodesic between the two root systems, and hence a robust estimate of the elastic distance between them, despite the presence of topological inconsistencies.

\begin{figure}[h!]
\begin{center}
\begin{tabular}{ccccc}
\includegraphics[trim = 80mm 80mm 80mm 20mm ,clip,width=2.5cm]{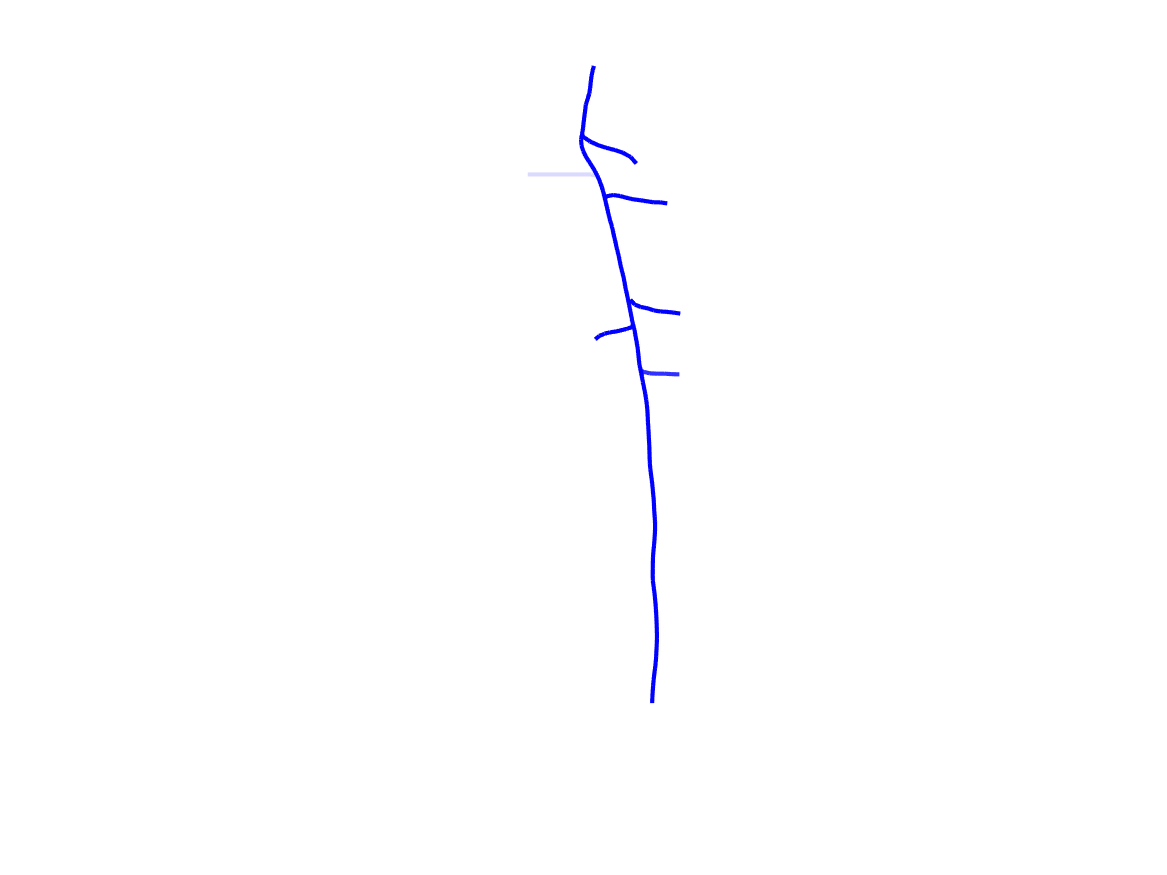}&
\includegraphics[trim = 80mm 80mm 80mm 20mm ,clip,width=2.5cm]{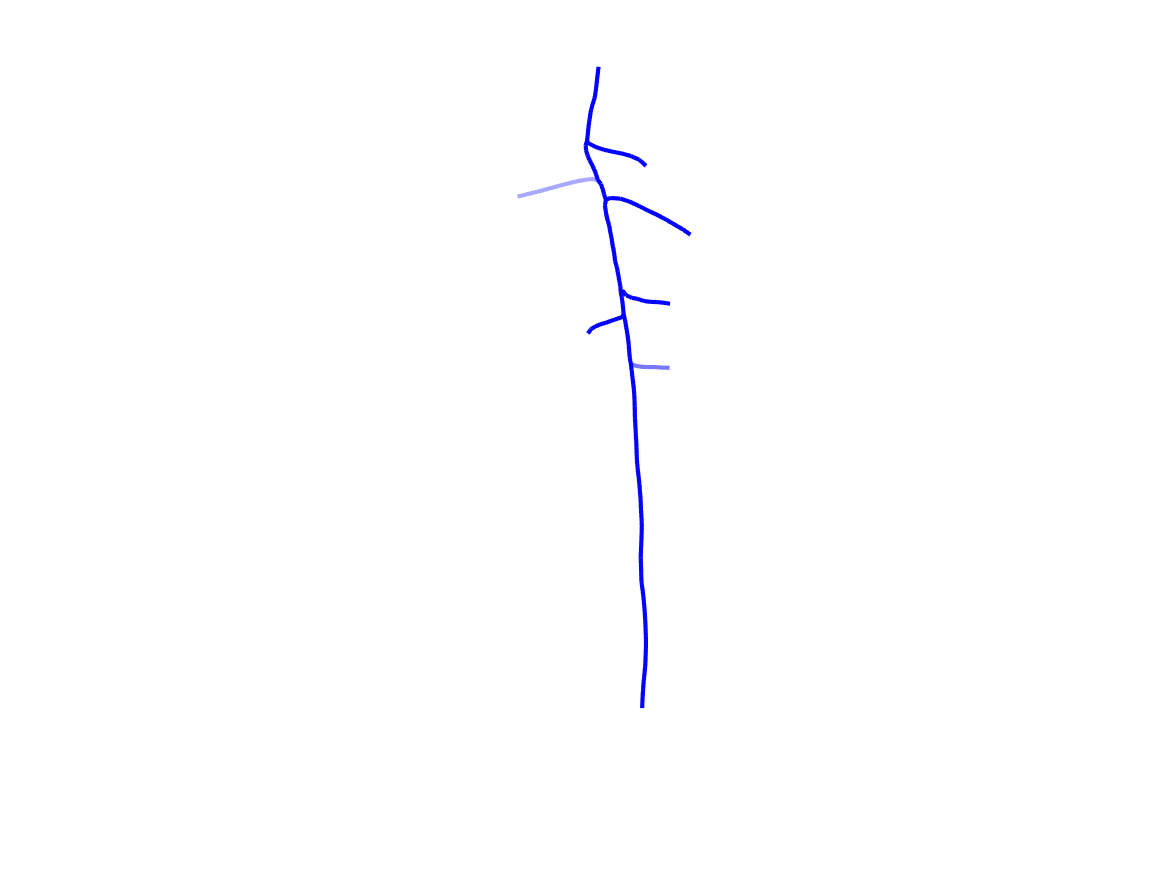}&
\includegraphics[trim = 80mm 80mm 80mm 20mm ,clip,width=2.5cm]{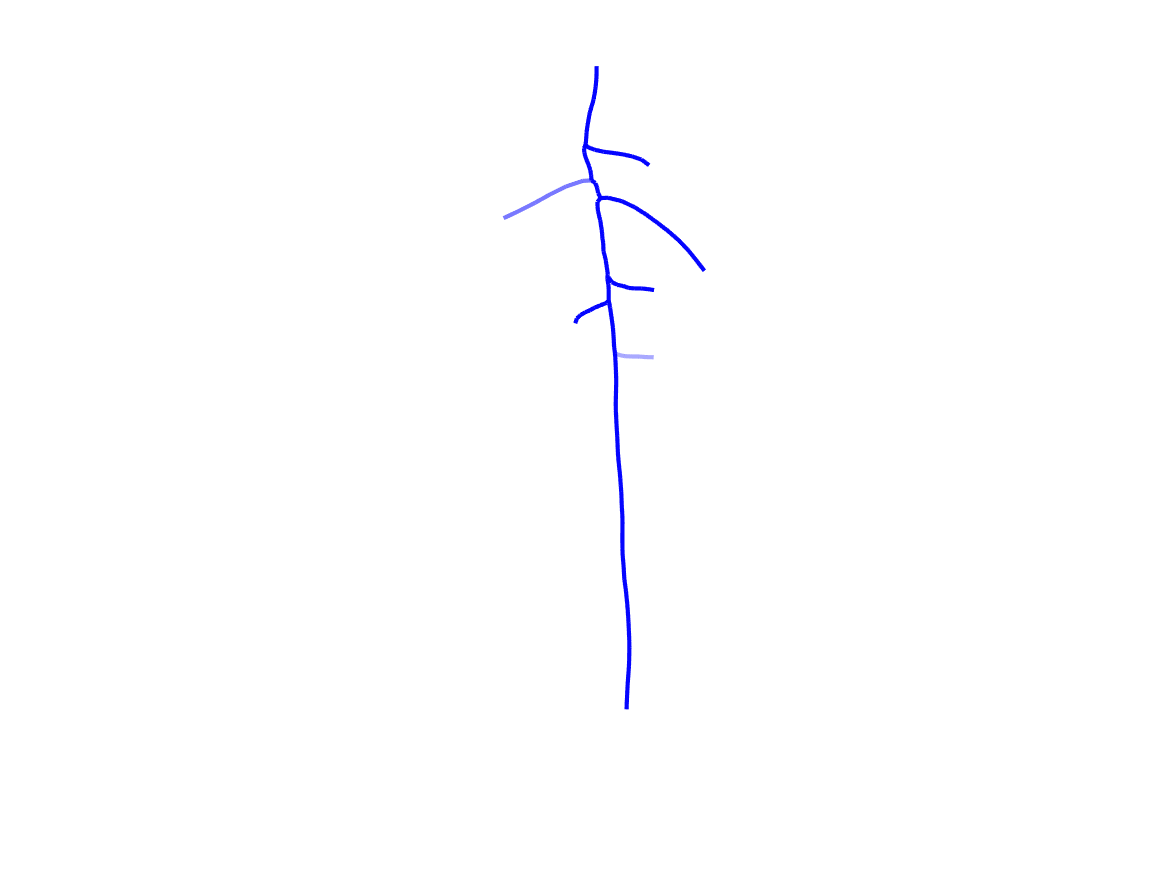}&
\includegraphics[trim = 80mm 80mm 80mm 20mm ,clip,width=2.5cm]{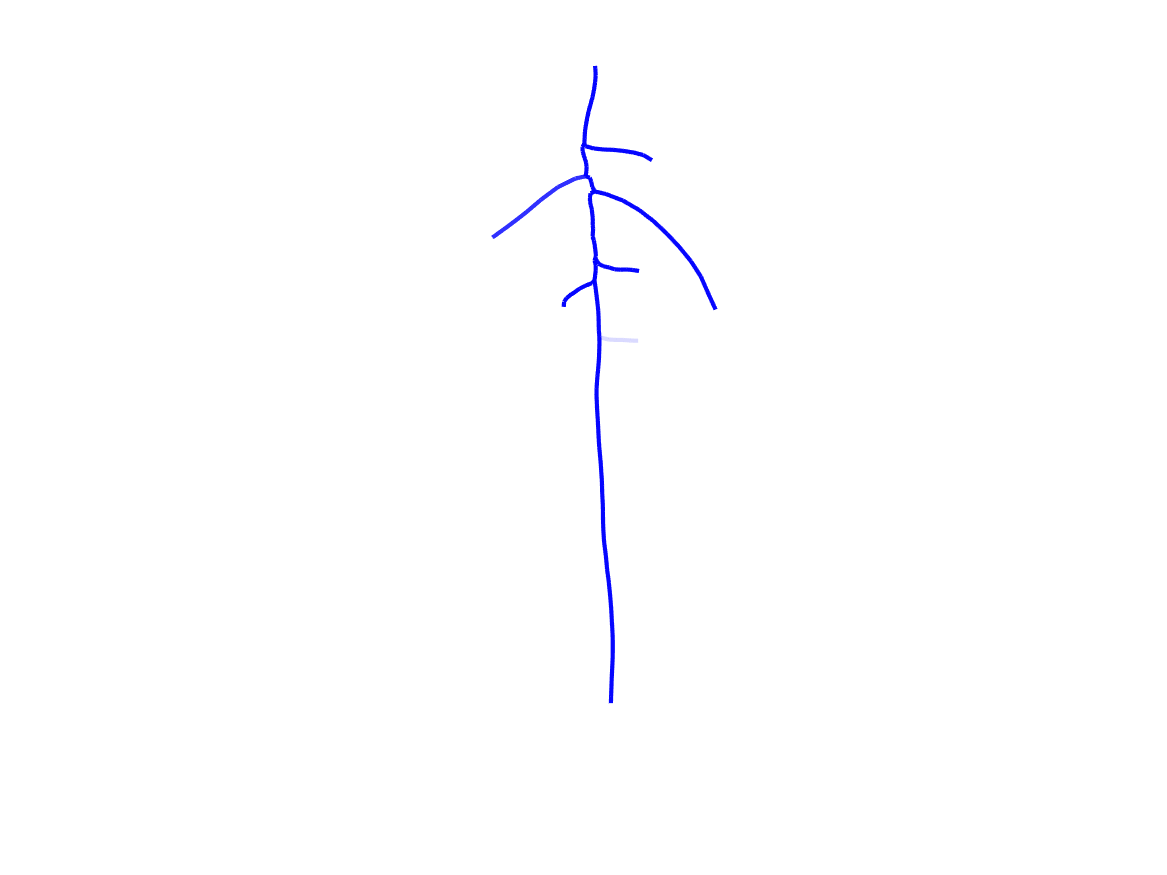}&
\includegraphics[trim = 80mm 80mm 80mm 20mm ,clip,width=2.5cm]{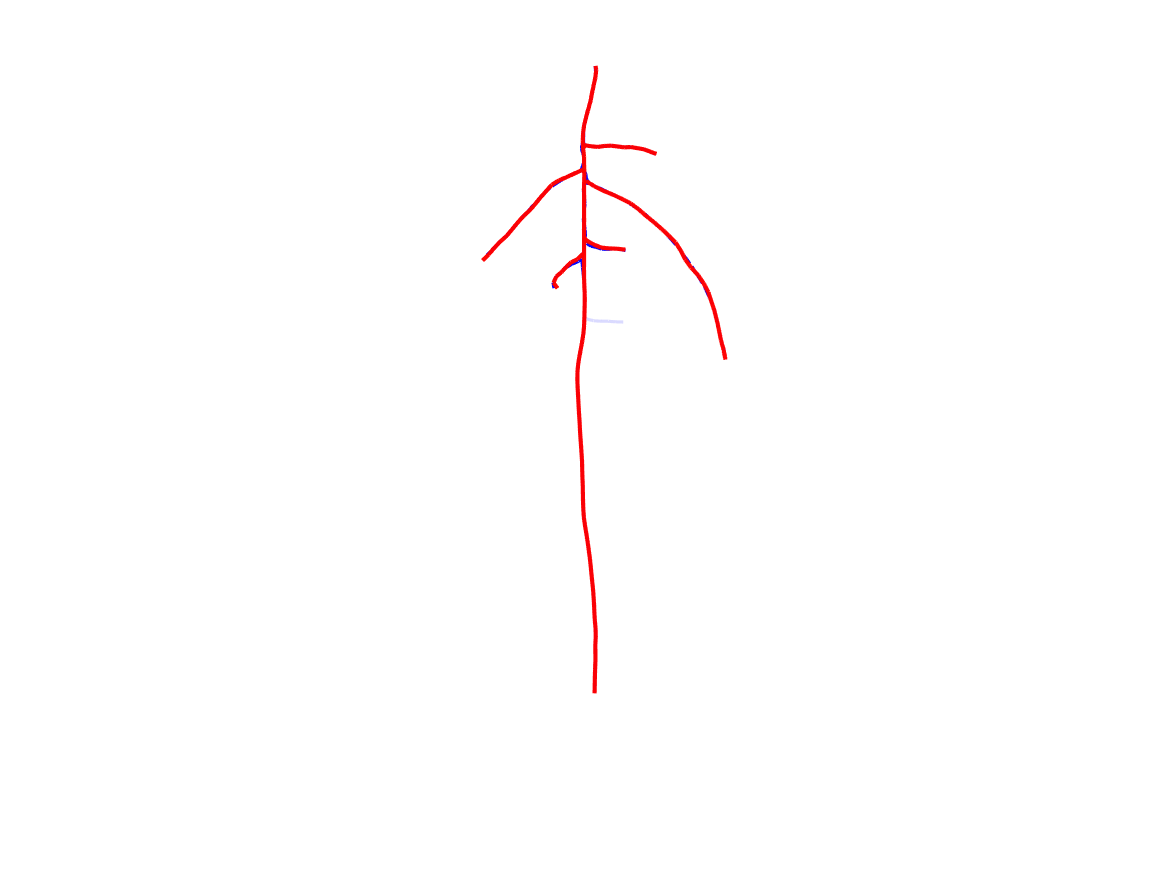}\\
$t=0$ & $t=0.25$ & $t=0.5$ & $t=0.75$ & $t=1$ 
\end{tabular}
\end{center}
\caption{\textbf{Simultaneous mass deletion and creation.} The source (blue at $t=0$) is a millet root with five lateral branches, to which we add one artificial lateral branch having weight zero (i.e., a phantom component), and the target (red at $t=1$) is another millet root with five lateral branches. We first ran the algorithm backwards by matching the target to the source, then used the resulting endpoint shape graph and weight function to determine the precise location at which to augment the source with a phantom branch.} Branches of the transformed source which get ``created'' are colored in progressively opaque shades of blue. Those that get ``deleted" are colored in progressively transparent shades of blue. The estimated geodesic distance is $\overline{\operatorname{dist}}(c_0,c_1)=0.90$.
\label{fig:match_masscreatedeleteroot_source_weight}
\end{figure}

\textbf{Shapes with multiplicities.} 
So far, in most of the examples we presented, the estimated weights were essentially binary as they were meant to account for partial matching constraints. We conclude this section with a final example to illustrate that our weighted shape graph model can be also useful when matching shapes with e.g. different densities or multiplicities. As a proof-of-concept, we consider the problem of matching a single line to a bundle of multiple curves, specifically a portion of the anterior commissure (CA) white matter fiber tract from the publicly available ISMRM 2015 Tractography Challenge repository\footnote{\url{http://www.tractometer.org/ismrm_2015_challenge/}}. This bundle, shown in Figure \ref{fig:match_1bundle}, contains 42 individual curves roughly aligned along a same path. Under a pure geometric matching model, the source curve would typically get folded multiple times as an attempt to compensate for the much higher total mass of the target. Although a natural solution could be to simply re-weight the shapes based on the number of curves in the bundle prior to matching, this procedure may be difficult to automate in practice when, for instance, bundles display crossing and fanning effects, or when some fibers are only partially recovered and possibly split into several components from the application of tractography algorithms. Our approach on the other hand allows us to bypass such a need by estimating local weight factors within the matching itself. As shown in Fig. \ref{fig:match_1bundle}, we are able to recover a transformed source curve that matches the overall geometry of the whole target bundle. The weight function we obtain has an average value of 38.8. We also notice that the estimated weights tend to be smaller near the extremities, which is consistent with the fact that the bundle is fanning in those parts, and thus has a lower fiber density.

\begin{figure}[h!]
\begin{center}
\begin{tabular}{ccccc}
\includegraphics[trim = 40mm 20mm 40mm 40mm ,clip,width=2.5cm]{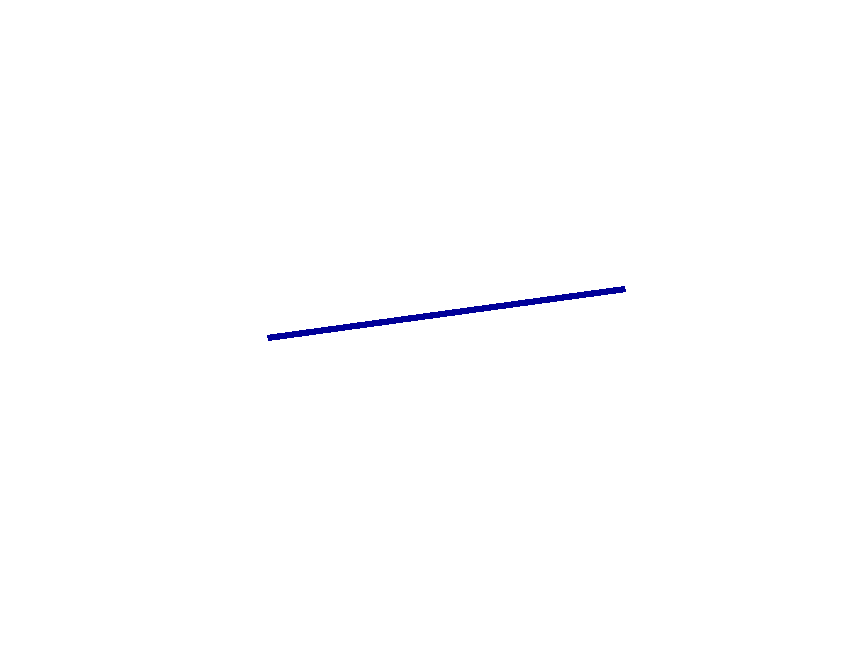}&
\includegraphics[trim = 40mm 20mm 40mm 40mm ,clip,width=2.5cm]{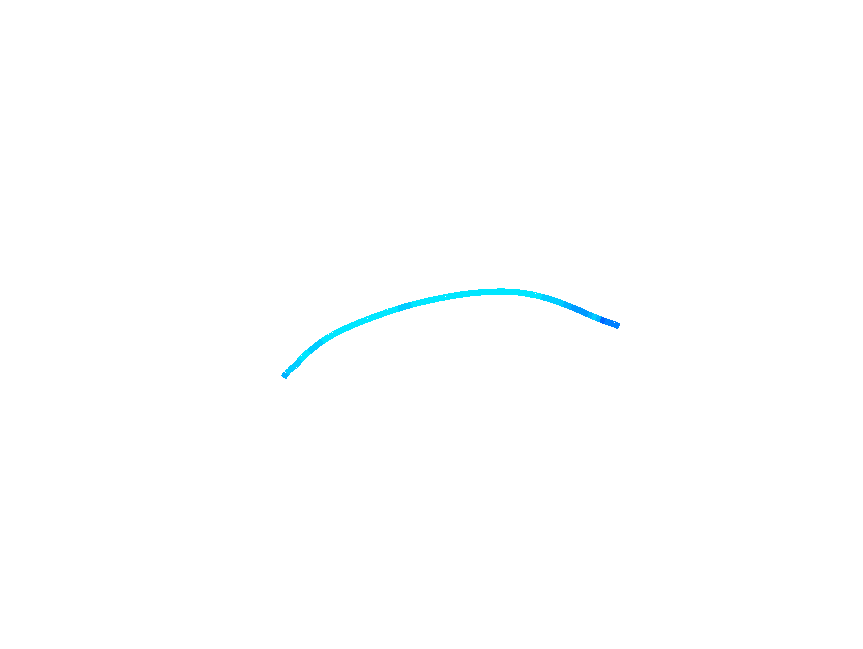}&
\includegraphics[trim = 40mm 20mm 40mm 40mm ,clip,width=2.5cm]{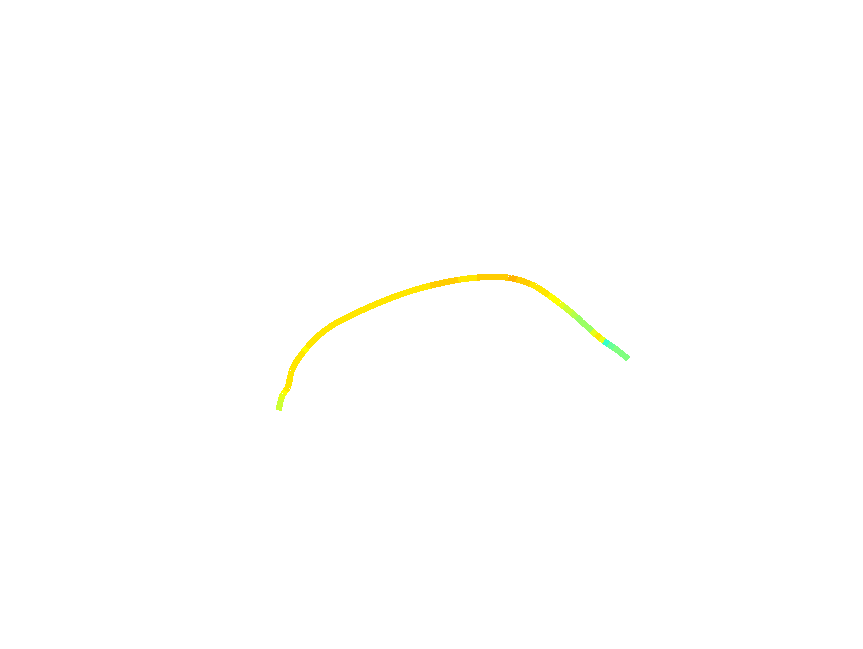}&
\includegraphics[trim = 40mm 20mm 40mm 40mm ,clip,width=2.5cm]{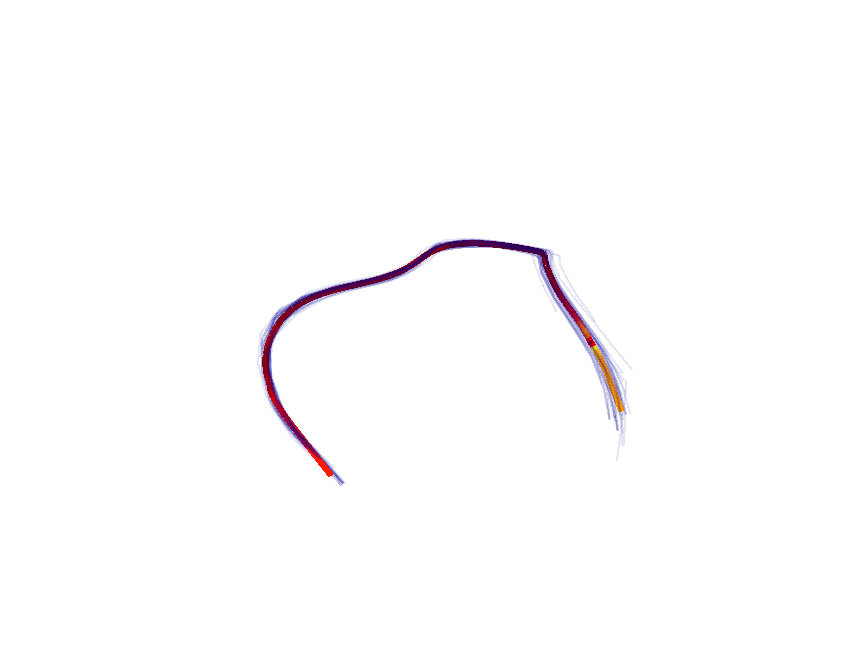}&
\includegraphics[width=0.5cm,height=2.5cm]{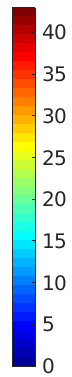}\\
$t=0$ & $t=1/3$ & $t=2/3$ & $t=1$ &
\end{tabular}
\end{center}
\caption{Matching result of a single curve (left) onto a bundle of 42 curves (superimposed in transparent blue on the right image). Here, the color represents the weight on the deforming source curve. The matching at $t=1$ leads to both a reasonable estimate of the average geometry of the bundle and an estimated weight function that provides a measure of the local density of curves in the bundle.}
\label{fig:match_1bundle}
\end{figure}

\section{Conclusion and outlook}
In this paper, we introduced an extended mathematical model and algorithm for the estimation of higher-order Sobolev metrics and their corresponding geodesics between shape graphs. We also proposed to resolve the issue of partial matching constraints in the situation of inconsistent topological structures through the additional estimation of a weight function defined on the source shape graph. We proved the well-posedness of the variational problem and derived an optimization scheme to numerically estimate its solutions. One of the main advantages of our variational framework is that it does not require optimizing explicitly over permutation or reparametrization groups.

There are however some remaining limitations to this approach. Firstly, it is a priori mainly suited for cases in which the target graph structure is topologically a subset of the source, as the estimation of vanishing weights allows us to remove specific portions of the source shape graph during the matching. As we briefly pointed out earlier, the opposite scenario in which the source shape graph has to be matched to a subset of the target could be addressed almost \textit{mutatis mutandis} by instead defining the weights on the target shape graph. Although we opted not to present results in this alternative setting, our open source implementation does in fact encompass this case as well. A much less obvious situation however is the symmetric problem in which one expects missing parts on both the source and target simultaneously. While extending our formulation to estimate weights on both shape graphs may seem a natural solution, one difficulty is to come up with an appropriate penalty function to avoid the trivial solution in which both shapes are entirely erased. Another possible approach, which was illustrated with Fig. \ref{fig:match_masscreatedeleteroot_source_weight}, is to rather model mass creation and deletion on the source shape graph, a downside being that this requires some prior information on the branches or parts of the source that will need to be created. We intend to further explore these different ideas in future work.

A second avenue for improvement would be to develop a faster numerical pipeline to estimate the distance and/or geodesics between shape graphs in view of applications of the method to the statistical analysis of large shape datasets. This could be achieved for instance by focusing on first-order Sobolev metrics (at the expense of solid theoretical results on the existence of solutions), for which the computation of the metric can be significantly simplified thanks to the availability of a general simplifying transformation known as the $F^{a,b}$ transform \cite{needham2020simplifying,sukurdeep2019inexact}. Alternatively, we plan to investigate supervised deep learning approaches as a way to replace our optimization procedure by a simple forward pass through an appropriately trained neural network. Although still in their infancy within the field of elastic shape analysis, such methods have recently shown some success in simpler settings \cite{nunez2020deep,hartman2021supervised}.

%------------------------- APPENDIX ------------------------------%
\appendix
\section{Proof of Lemmas~\ref{lemma:var_norm_continuity} and~\ref{lemma:var_norm_continuity2}}
\begin{proof}[Proof of Lemma~\ref{lemma:var_norm_continuity}]
 By assumption we have $c_p(\theta) \rightarrow c_{\infty}(\theta)$ and $\partial_{\theta}c_p(\theta) \rightarrow \partial_{\theta}c_{\infty}(\theta)$ uniformly on $D$ with $\partial_{\theta} c_p(\theta)\neq 0$ and $\partial_{\theta} c(\theta)\neq 0$ for all $n$ and $\theta \in D$. As $\mathcal{H}$ is continuously embedded into $C^1_0(\R^{d}\times S^{d-1})$, we have $\|\omega\|_{1,\infty} \leq C_{\H} \|\omega\|_{\H}$ for all $\omega \in \H$ and it results that: 
 \begin{align*}
    \|\mu_{c_p} - \mu_{c_\infty}\|_{\V} = \sup_{\|\omega\|_{\H} \leq 1} (\mu_{c_p} - \mu_{c_\infty} | \omega) &\leq \sup_{\|\omega\|_{1,\infty} \leq C_\H} (\mu_{c_p} - \mu_{c_\infty} | \omega) \leq C_\H \sup_{\|\omega\|_{1,\infty} \leq 1} (\mu_{c_p} - \mu_{c_\infty} | \omega).
 \end{align*}
 Therefore, we only need to show that $\sup_{\|\omega\|_{1,\infty} \leq 1} (\mu_{c_p} - \mu_{c_\infty} | \omega) \rightarrow 0$ as $p \rightarrow +\infty$. Let $\omega \in C^1_0(\R^{d}\times \Sp^{d-1})$ with $\|\omega\|_{1,\infty} \leq 1$. We have:
\begin{align*}
 &|(\mu_{c_n} - \mu_{c} | \omega)| \leq \int_D \left| \omega\left(c_p(\theta),\frac{\partial_{\theta} c_p(\theta)}{|\partial_{\theta} c_p(\theta)|}\right) |\partial_{\theta} c_p(\theta)| - \omega\left(c_\infty(\theta),\frac{\partial_{\theta}c_\infty(\theta)}{|\partial_{\theta}c_\infty(\theta)|}\right) |\partial_{\theta}c_\infty(\theta)| \right| d\theta \\
 &\leq \int_D \left| \omega\left(c_p(\theta),\frac{\partial_{\theta} c_p(\theta)}{|\partial_{\theta} c_p(\theta)|}\right) \right| \bigr| |\partial_{\theta} c_p(\theta)| - |\partial_{\theta}c_\infty(\theta)| \bigr| d\theta \\ 
 &\phantom{\leq}+  \int_D \left| \omega\left(c_p(\theta),\frac{\partial_{\theta} c_p(\theta)}{|\partial_{\theta} c_p(\theta)|}\right) - \omega\left(c_\infty(\theta),\frac{\partial_{\theta}c_\infty(\theta)}{|\partial_{\theta}c_\infty(\theta)|}\right) \right| |\partial_{\theta}c_\infty(\theta)| d\theta.
\end{align*}
Denoting by $\lambda^1(D)$ the Lebesgue measure of $D$ and $\ell_{c_\infty}$ the total length of $c_\infty$:
\begin{align*} 
 &|(\mu_{c_n} - \mu_{c} | \omega)| \leq \lambda^1(D) \|\omega\|_{\infty} \|\partial_{\theta} c_p - \partial_{\theta}c_\infty\|_{\infty} + \ell_{c_\infty} \|\omega\|_{1,\infty} \|c_p - c_\infty\|_{1,\infty}) \\
 &\leq \text{Cte} \ \|\omega\|_{1,\infty} \|c_p - c_\infty\|_{1,\infty} \leq \text{Cte} \ \|c_p - c_\infty\|_{1,\infty}
\end{align*}
where $\text{Cte}$ is a constant that only depends on $D$ and $c$. Therefore we get $\sup_{\|\omega\|_{1,\infty} \leq 1} |(\mu_{c_p} - \mu_{c} | \omega)| \leq \text{Cte} \|c_p - c\|_{1,\infty} \rightarrow 0$ and thus $\|\mu_{c_p} - \mu_{c}\|_{\V} \rightarrow 0$ as $p \rightarrow +\infty$.
\end{proof}

\begin{proof}[Proof of Lemma~\ref{lemma:var_norm_continuity2}]
 Similar to the proof of Lemma \ref{lemma:var_norm_continuity}, we only need to show that for all $k=1,\ldots,K$, $\sup_{\|\omega\|_{1,\infty} \leq 1} (\rho^k_m \cdot \mu_{c^k_m} - \rho^k \cdot \mu_{c^k} | \omega) \rightarrow 0$ as $m \rightarrow +\infty$. If $\omega \in C^1_0(\R^{d}\times \Sp^{d-1})$ with $\|\omega\|_{1,\infty} \leq 1$, we have:
 \begin{align*}
  &|(\rho^k_m \cdot \mu_{c^k_m} - \rho^k \cdot \mu_{c^k} | \omega)| \\ 
  &\leq \int_0^1 \left|\omega\left(c^k_m(\theta),\frac{\partial_{\theta}c^k_m(\theta)}{|\partial_{\theta} c^k_m(\theta)|}\right) \rho^k_m(\theta) |\partial_{\theta} c^k_m(\theta)| - \omega\left(c^k(\theta),\frac{\partial_{\theta} c^{k}(\theta)}{|\partial_{\theta}c^{k}(\theta)|}\right) \rho^k(\theta) |\partial_{\theta}c^{k}(\theta)|  \right| d\theta \\
  &\leq \underbrace{\int_0^1 \left|\omega\left(c^k_m(\theta),\frac{\partial_{\theta}c^k_m(\theta)}{|\partial_{\theta}c^k_m(\theta)|}\right) \rho^k_m(\theta) |\partial_{\theta}c^k_m(\theta)| - \omega\left(c^k_m(\theta),\frac{\partial_{\theta}c^k_m(\theta)}{|\partial_{\theta}c^k_m(\theta)|}\right) \rho^k(\theta) |\partial_{\theta}c^k(\theta)| \right| d\theta}_{\doteq (1)} \\
  &\phantom{aaa}+ \underbrace{\int_0^1 \left|\omega\left(c^k_m(\theta),\frac{\partial_{\theta}c^k_m(\theta)}{|\partial_{\theta}c^k_m(\theta)|}\right) \rho^k(\theta) |\partial_{\theta}c^k(\theta)| - \omega\left(c^k(\theta),\frac{\partial_{\theta}c^k(\theta)}{|\partial_{\theta}c^k(\theta)|}\right) \rho^k(\theta) |\partial_{\theta}c^k(\theta)|  \right| d\theta}_{\doteq (2)}
 \end{align*}
 We can then derive the following upper bounds for each term:
 \begin{align*}
  (1) &\leq \|\omega\|_{\infty} \int_0^1 \left|\rho^k_m(\theta) |\partial_{\theta}c^k_m(\theta)| - \rho^k(\theta) |\partial_{\theta}c^k(\theta)| \right| d\theta \\
  &\leq \int_0^1 |\rho^k_m(\theta) - \rho^k(\theta)| |\partial_{\theta}c^k_m(\theta)| d\theta + \int_0^1 \rho^k(\theta) |\partial_{\theta}c^k_m(\theta) - \partial_{\theta} c^k(\theta)| d\theta \\
  &\leq \left(\sup_{m\in \mathbb{N}} \|c^k_m\|_{1,\infty}\right) \|\rho^k_m - \rho^k\|_{L^1} + \|\rho^k\|_{L^1}  \|c^k_m -c^k\|_{1,\infty}.
 \end{align*} 
For the second term, we have
 \begin{align*}
  (2) &\leq \int_0^1 \left|\omega\left(c_m^k(\theta),\frac{\partial_{\theta}c_m^k(\theta)}{|\partial_{\theta}c_m^k(\theta)|}\right) - \omega\left(c^k(\theta),\frac{\partial_{\theta} c^k(\theta)}{|\partial_{\theta} c^k(\theta)|}\right)\right| \rho^k(\theta) |\partial_{\theta}c^k(\theta)| d\theta
 \end{align*} 
 and the same arguments as in the proof of Lemma \ref{lemma:var_norm_continuity} give that $(2) \leq \text{Cte} \ \|c^k_m - c^k\|_{1,\infty}$.
 
Since $\|c^k_m - c^k\|_{1,\infty}$ and $\|\rho^k_m - \rho^k\|_{L^1}$ both converge to $0$, it follows that $\sup_{\|\omega\|_{1,\infty} \leq 1} (\rho^k_m \cdot \mu_{c^k_m} - \rho^k \cdot \mu_{c^k} | \omega) \rightarrow 0$ as $m \rightarrow +\infty$.  
\end{proof}

%BIBLIOGRAPHY --------------------------------------------------
\bibliographystyle{siamplain}
\bibliography{references}

\begin{thebibliography}{10}

\bibitem{amann2000compact}
{\sc H.~Amann}, {\em {Compact embeddings of vector valued Sobolev and Besov
  spaces}}, Glasnik matemati{\v{c}}ki, 35 (2000), pp.~161--177.

\bibitem{ambrosio2000functions}
{\sc L.~Ambrosio, N.~Fusco, and D.~Pallara}, {\em Functions of bounded
  variation and free discontinuity problems}, vol.~254, Clarendon Press Oxford,
  2000.

\bibitem{antonsanti2021partial}
{\sc P.-L. Antonsanti, J.~Glaun{\`e}s, T.~Benseghir, V.~Jugnon, and
  I.~Kaltenmark}, {\em Partial matching in the space of varifolds}, arXiv
  preprint arXiv:2103.12441,  (2021).

\bibitem{atkin1975hopf}
{\sc C.~Atkin}, {\em {The Hopf-Rinow Theorem is false in infinite Dimensions}},
  Bulletin of the London Mathematical Society, 7 (1975), pp.~261--266.

\bibitem{bauer2019relaxed}
{\sc M.~Bauer, M.~Bruveris, N.~Charon, and J.~M{\o}ller-Andersen}, {\em A
  relaxed approach for curve matching with elastic metrics}, ESAIM: Control,
  Optimisation and Calculus of Variations, 25 (2019), p.~72.

\bibitem{bauer2012vanishing}
{\sc M.~Bauer, M.~Bruveris, P.~Harms, and P.~W. Michor}, {\em Vanishing
  geodesic distance for the riemannian metric with geodesic equation the
  kdv-equation}, Annals of Global Analysis and Geometry, 41 (2012),
  pp.~461--472.

\bibitem{bauer2017numerical}
{\sc M.~Bauer, M.~Bruveris, P.~Harms, and J.~M{\o}ller-Andersen}, {\em A
  numerical framework for sobolev metrics on the space of curves}, SIAM Journal
  on Imaging Sciences, 10 (2017), pp.~47--73.

\bibitem{h2metricsGIT}
{\sc M.~Bauer, M.~Bruveris, P.~Harms, and J.~M{\o}ller-Andersen}, {\em
  {H2metrics GitHub Repository}}.
\newblock \url{http://www.github.com/h2metrics/h2metrics}, 2018.

\bibitem{bauer2014overview}
{\sc M.~Bauer, M.~Bruveris, and P.~W. Michor}, {\em Overview of the geometries
  of shape spaces and diffeomorphism groups}, Journal of Mathematical Imaging
  and Vision, 50 (2014), pp.~60--97.

\bibitem{bauer2020sobolev}
{\sc M.~Bauer, C.~Maor, and P.~W. Michor}, {\em Sobolev metrics on spaces of
  manifold valued curves}, arXiv preprint arXiv:2007.13315,  (2020).

\bibitem{bronstein2009partial}
{\sc A.~M. Bronstein, M.~M. Bronstein, A.~M. Bruckstein, and R.~Kimmel}, {\em
  Partial similarity of objects, or how to compare a centaur to a horse},
  International Journal of Computer Vision, 84 (2009), p.~163.

\bibitem{bruveris7completeness}
{\sc M.~Bruveris}, {\em Completeness properties of sobolev metrics on the space
  of curves}, Journal of Geometric Mechanics, 7 (2015), p.~125.

\bibitem{bruveris2016optimal}
{\sc M.~Bruveris}, {\em Optimal reparametrizations in the square root velocity
  framework}, SIAM Journal on Mathematical Analysis, 48 (2016), pp.~4335--4354.

\bibitem{bruveris2014geodesic}
{\sc M.~Bruveris, P.~W. Michor, and D.~Mumford}, {\em Geodesic completeness for
  sobolev metrics on the space of immersed plane curves}, in Forum of
  Mathematics, Sigma, vol.~2, Cambridge University Press, 2014.

\bibitem{bruveris2017completeness}
{\sc M.~Bruveris and J.~M{\o}ller-Andersen}, {\em Completeness of
  length-weighted sobolev metrics on the space of curves}, arXiv preprint
  arXiv:1705.07976,  (2017).

\bibitem{calissano2020populations}
{\sc A.~Calissano, A.~Feragen, and S.~Vantini}, {\em Populations of unlabeled
  networks: Graph space geometry and geodesic principal components}, MOX
  Report,  (2020).

\bibitem{charon2020fidelity}
{\sc N.~Charon, B.~Charlier, J.~Glaun{\`e}s, P.~Gori, and P.~Roussillon}, {\em
  Fidelity metrics between curves and surfaces: currents, varifolds, and normal
  cycles}, in Riemannian Geometric Statistics in Medical Image Analysis,
  Elsevier, 2020, pp.~441--477.

\bibitem{charon2013varifold}
{\sc N.~Charon and A.~Trouv{\'e}}, {\em The varifold representation of
  nonoriented shapes for diffeomorphic registration}, SIAM Journal on Imaging
  Sciences, 6 (2013), pp.~2547--2580.

\bibitem{duncan2018statistical}
{\sc A.~Duncan, E.~Klassen, A.~Srivastava, et~al.}, {\em Statistical shape
  analysis of simplified neuronal trees}, Annals of Applied Statistics, 12
  (2018), pp.~1385--1421.

\bibitem{edelsbrunner2008persistent}
{\sc H.~Edelsbrunner and J.~Harer}, {\em Persistent homology-a survey},
  Contemporary mathematics, 453 (2008), pp.~257--282.

\bibitem{feragen2012toward}
{\sc A.~Feragen, P.~Lo, M.~de~Bruijne, M.~Nielsen, and F.~Lauze}, {\em Toward a
  theory of statistical tree-shape analysis}, IEEE transactions on pattern
  analysis and machine intelligence, 35 (2012), pp.~2008--2021.

\bibitem{feragen2020statistics}
{\sc A.~Feragen and T.~Nye}, {\em Statistics on stratified spaces}, in
  Riemannian Geometric Statistics in Medical Image Analysis, Elsevier, 2020,
  pp.~299--342.

\bibitem{glaunes2008large}
{\sc J.~Glaun{\`e}s, A.~Qiu, M.~I. Miller, and L.~Younes}, {\em Large
  deformation diffeomorphic metric curve mapping}, International journal of
  computer vision, 80 (2008), p.~317.

\bibitem{guo2020statistical}
{\sc X.~Guo, A.~B. Bal, T.~Needham, and A.~Srivastava}, {\em Statistical shape
  analysis of brain arterial networks (ban)}, arXiv preprint arXiv:2007.04793,
  (2020).

\bibitem{guo2020representations}
{\sc X.~Guo and A.~Srivastava}, {\em Representations, metrics and statistics
  for shape analysis of elastic graphs}, in Proceedings of the IEEE/CVF
  Conference on Computer Vision and Pattern Recognition Workshops, 2020,
  pp.~832--833.

\bibitem{hartman2021supervised}
{\sc E.~Hartman, Y.~Sukurdeep, N.~Charon, E.~Klassen, and M.~Bauer}, {\em
  {Supervised deep learning of elastic SRV distances on the shape space of
  curves}}, arXiv preprint arXiv:2101.04929,  (2021).

\bibitem{hsieh2021metrics}
{\sc H.-W. Hsieh and N.~Charon}, {\em Metrics, quantization and registration in
  varifold spaces}, Foundations of Computational Mathematics,  (2021),
  pp.~1--45.

\bibitem{jain2009structure}
{\sc B.~J. Jain and K.~Obermayer}, {\em Structure spaces.}, Journal of Machine
  Learning Research, 10 (2009).

\bibitem{kaltenmark2017general}
{\sc I.~Kaltenmark, B.~Charlier, and N.~Charon}, {\em A general framework for
  curve and surface comparison and registration with oriented varifolds}, in
  Proceedings of the IEEE Conference on Computer Vision and Pattern
  Recognition, 2017, pp.~3346--3355.

\bibitem{kaltenmark2019estimation}
{\sc I.~Kaltenmark and A.~Trouv{\'e}}, {\em Estimation of a growth development
  with partial diffeomorphic mappings}, Quarterly of Applied Mathematics, 77
  (2019), pp.~227--267.

\bibitem{kendall1989survey}
{\sc D.~G. Kendall}, {\em A survey of the statistical theory of shape},
  Statistical Science,  (1989), pp.~87--99.

\bibitem{lahiri2015precise}
{\sc S.~Lahiri, D.~Robinson, and E.~Klassen}, {\em Precise matching of pl
  curves in $\mathbb r^n$ in the square root velocity framework}, Geometry,
  Imaging and Computing, 2 (2015), pp.~133--186.

\bibitem{mennucci2008properties}
{\sc A.~Mennucci, A.~Yezzi, and G.~Sundaramoorthi}, {\em {Properties of
  Sobolev-type metrics in the space of curves}}, Interfaces and Free
  Boundaries, 10 (2008), pp.~423--445.

\bibitem{Michor2005}
{\sc P.~W. Michor and D.~Mumford}, {\em {Vanishing geodesic distance on spaces
  of submanifolds and diffeomorphisms}}, Doc. Math., 10 (2005), pp.~217--245
  (electronic).

\bibitem{michor2007overview}
{\sc P.~W. Michor and D.~Mumford}, {\em An overview of the riemannian metrics
  on spaces of curves using the hamiltonian approach}, Applied and
  Computational Harmonic Analysis, 23 (2007), pp.~74--113.

\bibitem{mio2007shape}
{\sc W.~Mio, A.~Srivastava, and S.~Joshi}, {\em On shape of plane elastic
  curves}, International Journal of Computer Vision, 73 (2007), pp.~307--324.

\bibitem{nardi2016geodesics}
{\sc G.~Nardi, G.~Peyr{\'e}, and F.-X. Vialard}, {\em {Geodesics on shape
  spaces with bounded variation and Sobolev metrics}}, SIAM Journal on Imaging
  Sciences, 9 (2016), pp.~238--274.

\bibitem{needham2020simplifying}
{\sc T.~Needham and S.~Kurtek}, {\em Simplifying transforms for general elastic
  metrics on the space of plane curves}, SIAM Journal on Imaging Sciences, 13
  (2020), pp.~445--473.

\bibitem{nunez2020deep}
{\sc E.~Nunez and S.~H. Joshi}, {\em Deep learning of warping functions for
  shape analysis}, in Proceedings of the IEEE/CVF Conference on Computer Vision
  and Pattern Recognition Workshops, 2020, pp.~866--867.

\bibitem{pan2016current}
{\sc Y.~Pan, G.~E. Christensen, O.~C. Durumeric, S.~E. Gerard, J.~M. Reinhardt,
  and G.~D. Hugo}, {\em Current-and varifold-based registration of lung vessel
  and airway trees}, in Proceedings of the IEEE Conference on Computer Vision
  and Pattern Recognition Workshops, 2016, pp.~126--133.

\bibitem{robinson2012functional}
{\sc D.~T. Robinson}, {\em Functional data analysis and partial shape matching
  in the square root velocity framework},  (2012).

\bibitem{rodola2017partial}
{\sc E.~Rodol{\`a}, L.~Cosmo, M.~M. Bronstein, A.~Torsello, and D.~Cremers},
  {\em Partial functional correspondence}, in Computer Graphics Forum, vol.~36,
  Wiley Online Library, 2017, pp.~222--236.

\bibitem{srivastava2020advances}
{\sc A.~Srivastava, X.~Guo, and H.~Laga}, {\em Advances in geometrical analysis
  of topologically-varying shapes}, in 2020 IEEE 17th International Symposium
  on Biomedical Imaging Workshops (ISBI Workshops), IEEE, 2020, pp.~1--4.

\bibitem{srivastava2010shape}
{\sc A.~Srivastava, E.~Klassen, S.~H. Joshi, and I.~H. Jermyn}, {\em Shape
  analysis of elastic curves in euclidean spaces}, IEEE Transactions on Pattern
  Analysis and Machine Intelligence, 33 (2010), pp.~1415--1428.

\bibitem{srivastava2016functional}
{\sc A.~Srivastava and E.~P. Klassen}, {\em Functional and shape data
  analysis}, vol.~1, Springer, 2016.

\bibitem{sukurdeep2019inexact}
{\sc Y.~Sukurdeep, M.~Bauer, and N.~Charon}, {\em An inexact matching approach
  for the comparison of plane curves with general elastic metrics}, in 53rd
  Asilomar Conference on Signals, Systems, and Computers, 2019, pp.~512--516.

\bibitem{tan2014smoothing}
{\sc Z.~Tan, Y.~C. Eldar, A.~Beck, and A.~Nehorai}, {\em Smoothing and
  decomposition for analysis sparse recovery}, IEEE Transactions on Signal
  Processing, 62 (2014), pp.~1762--1774.

\bibitem{tarjan1972depth}
{\sc R.~Tarjan}, {\em Depth-first search and linear graph algorithms}, SIAM
  journal on computing, 1 (1972), pp.~146--160.

\bibitem{trouve2000diffeomorphic}
{\sc A.~Trouv{\'e} and L.~Younes}, {\em {Diffeomorphic matching problems in one
  dimension: Designing and minimizing matching functionals}}, in European
  Conference on Computer Vision, Springer, 2000, pp.~573--587.

\bibitem{wang2020statistical}
{\sc G.~Wang, H.~Laga, J.~Jia, S.~J. Miklavcic, and A.~Srivastava}, {\em
  Statistical analysis and modeling of the geometry and topology of plant
  roots}, Journal of theoretical biology, 486 (2020), p.~110108.

\bibitem{wang2020efficient}
{\sc K.~Wang, Y.~Yan, and M.~Diaz}, {\em Efficient clustering for stretched
  mixtures: Landscape and optimality}, Advances in Neural Information
  Processing Systems, 33 (2020).

\bibitem{younes1998computable}
{\sc L.~Younes}, {\em Computable elastic distances between shapes}, SIAM
  Journal on Applied Mathematics, 58 (1998), pp.~565--586.

\bibitem{younes2010shapes}
{\sc L.~Younes}, {\em Shapes and diffeomorphisms}, vol.~171, Springer, 2010.

\end{thebibliography}
\end{document}